\documentclass[11pt]{amsart}
\usepackage{amsfonts}
\usepackage{graphicx}
\usepackage{epsfig}
\usepackage{amssymb,amsmath}
\usepackage{amsthm}
\usepackage{setspace}
\usepackage[top = 1in, bottom = 1in, left = 1.2in, right =
1.2in]{geometry}

\newtheorem{theorem}{Theorem}
\newtheorem{acknowledgement}[theorem]{Acknowledgement}

\newtheorem{claim}[theorem]{Claim}

\newtheorem{corollary}[theorem]{Corollary}

\newtheorem{definition}[theorem]{Definition}

\newtheorem{lemma}[theorem]{Lemma}

\newtheorem{proposition}[theorem]{Proposition}
\newtheorem{remark}[theorem]{Remark}

\numberwithin{equation}{section}
\numberwithin{theorem}{section}
\newcommand \del {\partial}
\newcommand\ep {\varepsilon}

\begin{document}
\title{The Fokker-Planck equation with absorbing boundary conditions }
\author{Hyung Ju Hwang}
\address{Department of Mathematics, Pohang University of Science and
Technology, Pohang, GyungBuk 790-784, Republic of Korea}
\email{hjhwang@postech.ac.kr}
\author{Juhi Jang}
\address{Department of Mathematics, University of California, Riverside,
Riverside, CA 92521, USA}
\email{juhijang@math.ucr.edu}
\author{Juan J. L. Vel\'{a}zquez}
\address{Institute of Applied Mathematics, University of Bonn, Endenicher
Allee 60, 53115 Bonn, Germany.}
\email{velazquez@iam.uni-bonn.de}

\begin{abstract}
We study the initial-boundary value problem for the Fokker-Planck equation
in an interval with absorbing boundary conditions. We develop a theory of
well-posedness of classical solutions for the problem. We also prove that
the resulting solutions decay exponentially for long times. To prove these
results we obtain several crucial estimates, which include hypoellipticity
away from the singular set for the Fokker-Planck equation with absorbing
boundary conditions, as well as the H\"{o}lder continuity of the solutions
up to the singular set.
\end{abstract}

\maketitle

%\singlespacing

%\bigskip

%---------------------------------------------------------%---------------------------------------------------------

\section{Introduction}

%---------------------------------------------------------%---------------------------------------------------------

We consider the initial boundary value problem for the following
Fokker-Planck equation in an interval $[0,1]$.
\begin{align}
f_{t}+vf_{x}& =f_{vv},  \label{VFP} \\
f\left( x,v,0\right) & =f_{0}\left( x,v\right) ,  \label{id} \\
f\left( 0,v,t\right) & =0,\text{ for }v>0,~t>0,  \label{BC0} \\
f\left( 1,v,t\right) & =0,~\text{for }v<0,~t>0,  \label{BC1}
\end{align}

\noindent where $f\left( x,v,t\right) \geq 0$ is the distribution of
particles at position $x$, velocity $v$, and time $t$ for $(x,v,t)\in \left[
0,1\right] \times \mathbb{R}\times \mathbb{R}_{+}$ and $f_{0}\left(
x,v\right) \geq 0$ the initial charge distribution.

The kinetic boundary condition given in \eqref{BC0}-(\ref{BC1}) is the
so-called absorbing boundary condition or absorbing barrier (cf. \cite{G1},
\cite{MW1}). If we interpret (\ref{VFP})-(\ref{BC1}) as the equation for the
density in the phase space of a system of particles, the meaning of the
boundary conditions \eqref{BC0}-(\ref{BC1}) is that the particles reaching
the boundary of the domain containing them, can escape but not re-enter it.

Equations with the form (\ref{VFP}) and boundary conditions like (\ref{BC0})
appear in the study of different problems of statistical physics. For
instance, they arise in the study of Brownian particles moving in bounded
domains (cf. \cite{MW1}), or in the study of the statistics of polymer
chains (cf. \cite{Bu}).

The Fokker-Planck operator is a well-known hypoelliptic operator. Diffusion
in $v$ together with the transport term $v\cdot \nabla _{x}$ has a
regularizing effect for solutions not only in $v$ but also in $t$ and $x$,
which can be obtained by applying H\"{o}rmander's commutator (cf. \cite{H})
to the linear Fokker-Planck operator. For more details, see \cite{B}. Note
that these results were obtained in the whole space without boundaries.

On the other hand, the Fokker-Planck operator is also known as a
hypocoercive operator, which concerns the rate of convergence to equilibria.
Indeed, the trend to equilibria with a certain rate has been investigated in
many papers (cf. \cite{DV}, \cite{HJ}, \cite{HN04}, \cite{V}) in the
Maxwellian regime and in the whole space or in the periodic box. For more
details, we refer to \cite{V}.

The hypoelliptic and hypocoercive property have also been explored for other
kinetic equations. Among others, we briefly review theories of existence,
regularity, and asymptotic behaviors for the Vlasov-Poisson-Fokker-Planck
system in the whole space, which is one of the important models in
mathematical physics and has been widely studied. Global existence of
classical solutions were studied in \cite{B2}, \cite{RW}, \cite{VO}.
Asymptotic behaviors and time decay of the solutions in the vacuum regime
were considered in \cite{Carpio98}, \cite{CSV}, \cite{OS}. We mention the
works in \cite{CS}, \cite{Vi}, where the global weak solutions were
constructed, and the work in \cite{B3}, where the smoothing effect was
observed.

Compared to the theory in the case of the whole space, little progress has
been made towards the boundary-value problems for these equations. In \cite%
{BCS}, \cite{Ca}, global weak solutions and asymptotic behaviors for the
Vlasov-Poisson-Fokker-Planck equations were studied in bounded domains with
absorbing and reflective type boundary conditions. In \cite{Mis}, a global
stability of DiPerna-Lions renormalized solutions to some kinetic equations
including the Vlasov-Fokker-Planck equation was studied under the Maxwell
boundary conditions.

However, to our knowledge, the hypoellipticity property for the
Fokker-Planck equation has not been studied in bounded domains other than
the periodic boundary condition, and no convergence rate for solutions of
the Fokker-Planck equation has been investigated for an interval in the
vacuum regime.

In this paper we develop a theory for classical solutions of \eqref{VFP}-(%
\ref{BC1}). We will also prove that the solutions of this problem vanish
exponentially fast as $t\rightarrow \infty .$

From the technical point of view, the main obstruction to develop a theory
for classical solutions of \eqref{VFP}-(\ref{BC1}) is the presence of the
so-called singular set. This set can be defined for some kinetic equations
(cf. \cite{G1}, \cite{G2}, \cite{HV}). In the case of \eqref{VFP}-(\ref{BC1}%
), the singular set reduces to the points $\left( x,v\right) \in \left\{
\left( 0,0\right) ,\left( 1,0\right) \right\} .$ The fact that the solutions
of kinetic equations cannot have arbitrary regularity near the singular set
was first noticed by Guo in the Vlasov-Poisson system (cf. \cite{G1}). In
this paper we will prove that the solutions of (\ref{VFP})-(\ref{BC1}) are
not $C^{\infty }$ in general near the singular set.

Notice that the equation (\ref{VFP}) contains the second derivative that
yields the regularizing effects only in the variable $v.$ On the other hand,
the presence of the transport term $vf_{x}$ has the following consequence
that the solutions of (\ref{VFP}) become $C^{\infty }$ for any $t>0$ in the
set $\left( \left[ 0,1\right] \times \mathbb{R}\right) \setminus \left\{
\left( 0,0\right) ,\left( 1,0\right) \right\} .$ This property is known as
hypoellipticity. However, such regularizing effects do not take place at the
singular set. Indeed, it turns out that there exist some explicit solutions
of (\ref{VFP}) with boundary conditions (\ref{BC0}), and it indicates the
maximum regularity that we can expect is $C_{x,v}^{1/6,1/2}.$

In order to prove the results of this paper we will use extensively maximum
principles and comparison arguments combined with suitable sub and
super-solutions. We will first construct a theory of weak solutions of (\ref%
{VFP})-(\ref{BC1}) by studying a regularized version of this problem
followed by a limit procedure. Using the maximum principle we will derive
suitable $L^{\infty }$ estimates for the corresponding weak solutions as
well as the uniqueness.

As a next step we will prove the hypoellipticity property for the solutions
of (\ref{VFP})-(\ref{BC1}) at the interior points of the domain $\left(
0,1\right) \times \mathbb{R}$ and also at the boundary points which do not
belong to the singular set. The proof of the hypoellipticity property for (%
\ref{VFP})-(\ref{BC1}) is classical at interior points. In order to prove it
at the boundary points we will use an integral representation formula for
the solutions of (\ref{VFP})-(\ref{BC1}) near the boundary points which do
not belong to the singular set.

We will then study the regularity of the solutions of (\ref{VFP})-(\ref{BC1}%
) near the singular set. This will be made using suitable sub and
super-solutions and comparison arguments. We will prove in this way that the
solutions $f\left( \cdot ,t\right) $ of (\ref{VFP})-(\ref{BC1}) belong to $%
C_{x,v}^{1/6-\varepsilon ,1/2-\varepsilon }$ for any $t>0,$ with $%
\varepsilon >0$ arbitrarily small.

We will also prove that the solutions of (\ref{VFP})-(\ref{BC1}) decay
exponentially fast as $t\rightarrow \infty .$ The main idea used in the
proof of this result is that, due to the hypoellipticity property, the
particle fluxes along the boundaries $\left\{ \left( 0,v\right) :v<0\right\}
\cup \left\{ \left( 1,v\right) :v>0\right\} $ are comparable to the total
number of particles at a given time (cf. Section 4). This implies the
exponential decay for the total number of particles of the system. However,
in order to make this argument precise a careful treatment is needed in
order to control the amount of mass near the singular set, because the
hypoellipticity property is not valid there. To control the mass in such
regions we will use again suitable sub and super-solutions.

\subsection{Main Results}

We first introduce notations for the domain and boundaries. Define
\begin{equation}
U_{T}=\Omega \times (0,T):=\{\left( x,v,t\right) \in (0,1)\times \left(
-\infty ,\infty \right) \times (0,T)\},  \label{Ut}
\end{equation}%
where $\Omega =(0,1)\times \left( -\infty ,\infty \right) .$

We also define the incoming, outgoing, and grazing kinetic boundary of $%
U_{T} $ as
\begin{eqnarray*}
\Gamma _{T}^{-} &:&=\Omega \times \left\{ t=0\right\} \cup \left\{
x=0\right\} \times (0,\infty )\times (0,T)\cup \left\{ x=1\right\} \times
(-\infty ,0)\times (0,T), \\
\Gamma _{T}^{+} &:&=\Omega \times \left\{ t=0\right\} \cup \left\{
x=0\right\} \times (-\infty ,0)\times (0,T)\cup \left\{ x=1\right\} \times
(0,\infty )\times (0,T), \\
\Gamma _{T}^{0} &:&=\Omega \times \left\{ t=0\right\} \cup \left\{
x=0\right\} \times \{v=0\}\times (0,T)\cup \left\{ x=1\right\} \times
\{v=0\}\times (0,T).
\end{eqnarray*}%
In addition, we define the incoming, outgoing, and grazing boundary of $%
U_{T} $ as%
\begin{eqnarray*}
\gamma _{T}^{-} &:&=\left\{ x=0\right\} \times (0,\infty )\times (0,T)\cup
\left\{ x=1\right\} \times (-\infty ,0)\times (0,T), \\
\gamma _{T}^{+} &:&=\left\{ x=0\right\} \times (-\infty ,0)\times (0,T)\cup
\left\{ x=1\right\} \times (0,\infty )\times (0,T), \\
\gamma _{T}^{0} &:&=\left\{ x=0\right\} \times \{v=0\}\times (0,T)\cup
\left\{ x=1\right\} \times \{v=0\}\times (0,T).
\end{eqnarray*}

We give a definition of a weak solution of (\ref{VFP})-(\ref{BC1}) in the
following.

\begin{definition}
\label{weak-solution}We say that $f\in L^{\infty }\left( \left[ 0,T\right]
;L^{1}\cap L^{\infty }\left( \Omega \right) \right) $ is a weak solution of (%
\ref{VFP})-(\ref{BC1}) if the function
\begin{equation*}
t \rightarrow \int f(x,v,t) \psi (x,v,t) dxdv
\end{equation*}
is continuous on $[0,T]$ for any test function $\psi \left( x,v,t\right) \in
C_{x,v,t}^{1,2,1}\left( U_{T}\right) $ such that supp$\left( \psi \left(
\cdot ,\cdot ,t\right) \right) \subset \left[ 0,1\right] \times \left[ -R,R%
\right] $ for some $R>0$ and if it satisfies for every $t\in \left[ 0,T%
\right] $ and any test function $\psi \left( x,v,s\right) \in
C_{x,v,s}^{1,2,1}\left( U_{t}\right) $ such that supp$\left( \psi \left(
\cdot ,\cdot ,s\right) \right) \subset \left[ 0,1\right] \times \left[ -R,R%
\right] $ for some $R>0$ and $\psi |_{\gamma _{t}^{+}}=0$,
\begin{align*}
& \int_{\Omega}f\left( x,v,t\right) \psi \left( x,v,t\right) dxdv-
\int_{\Omega}f\left( x,v,0\right) \psi \left( x,v,0\right) dxdv \\
& =\int_{U_t} f\left( x,v,s\right) \left[ \psi _{t}\left( x,v,s\right)
+v\psi _{x}\left( x,v,s\right) +\psi _{vv}\left( x,v,s\right) \right] dxdvds.
\end{align*}
\end{definition}

We are now ready to state our main results. The first result concerns the
existence of a unique weak solution of (\ref{VFP})-(\ref{BC1}).

\begin{theorem}
\label{MainTheorem1} Let $T>0$ and $f_{0}\in L^{1}\cap L^{\infty }\left(
\Omega \right) $ with $f_{0}\geq 0$ given. Then there exists a unique weak
solution $f\in L^{\infty }\left( [0,T];L^{1}\cap L^{\infty }(\Omega )\right)
$ with $f\geq 0$ of the Fokker-Planck equation with the absorbing boundary
condition (\ref{VFP})-(\ref{BC1}). Moreover, the weak solution $f(t)$
satisfies the following bounds.
\begin{equation*}
\Vert f(t)\Vert _{L^{\infty }\left( \Omega \right) }\leq \Vert f_{0}\Vert
_{L^{\infty }\left( \Omega \right) }\text{ and }\Vert f(t)\Vert
_{L^{1}\left( \Omega \right) }\leq \Vert f_{0}\Vert _{L^{1}\left( \Omega
\right) }
\end{equation*}%
for each $t\in \lbrack 0,T]$.
\end{theorem}

The next results concern the regularity of weak solutions.

\begin{theorem}
\label{MainTheorem} Let $f\left( x,v,t\right) $ be the weak solution of (\ref%
{VFP})-(\ref{BC1}) with $f_{0}\in L^{1}\cap L^{\infty }\left( \Omega \right)
$ with $f_{0}\geq 0$. Then the following holds:

\begin{description}
\item[(i)] For each $t>0$, $f\in H^{k,m}_{loc}(\bar{\Omega}\setminus
\{(0,0),\left( 1,0\right) \}),$ where $H^{k,m}=H_{x,v}^{k,m}$ and for any $k$%
, $m\in \mathbb{N}$.

\item[(ii)] For all $t>0$, $f(x,v,t)$ is continuous in $\bar{\Omega}$ such
that $f(0,0,t)=f(1,0,t)=0$ for all $t>0$ and $\lim_{(x,v)\rightarrow
(0,0),(1,0)}f(x,v,t)=0$ for all $t>0$. In fact, $f$ is H\"{o}lder continuous
up to the singular set: $f\in C_{x,v}^{\alpha ,3\alpha }(\bar{\Omega})$ for
any number $0<\alpha <1/6$.
\end{description}
\end{theorem}

\bigskip

The last main theorem shows the exponential trend of the solution to $0$ in $%
L^{1}$ and $L^{\infty}$ sense:

\begin{theorem}
Let $f_{0}\left( x,v\right) \in L^{1}\cap L^{\infty }\left( \Omega \right) $
with $f_{0}\geq 0$ and let $f\left( x,v,t\right) $ be a solution to (\ref%
{VFP})-(\ref{BC1}). Then the following holds.

\begin{description}
\item[(i)] $f$ decays exponentially in time in $L^{1}\left( \Omega \right) .$
In particular, there exists $\kappa >0$ such that%
\begin{equation*}
\left\Vert f\left( t\right) \right\Vert _{L^{1}\left( \Omega \right) }\leq
\left\Vert f_{0}\right\Vert _{L^{1}\left( \Omega \right) }\exp \left(
-\kappa t\right) .
\end{equation*}

\item[(ii)] $f$ decays exponentially in time in $L^{\infty }\left( \Omega
\right) .$ In particular, there exist $\kappa >0$ and $C>0$ such that%
\begin{equation*}
\left\Vert f\left( t\right) \right\Vert _{L^{\infty }\left( \Omega \right)
}\leq C\exp \left( -\kappa t\right) ,
\end{equation*}%
where $C$ depends on $\left\Vert f_{0}\right\Vert _{L^{1}\left( \Omega
\right) }$ and $\left\Vert f_{0}\right\Vert _{L^{\infty }\left( \Omega
\right) }$.
\end{description}
\end{theorem}

%In this paper we will restrict the analysis of the problem (\ref{VFP})-(\ref%
%{BC1}) to the domain \newline
%$\left\{ \left( x,v,t\right) :x\in \left( 0,1\right) ,\ v\in \mathbb{R}\ ,\ t>0\right\} .$

\

Some of the first related results for the problem were obtained in \cite{McK}%
, where the probability distribution for the velocity with which an
accelerated Brownian particle -- a Brownian particle, i.e. a particle whose
paths are the ones associated to the Orstein-Uhlenbeck process -- exits from
the domain $\left\{ x>0\right\} $ is computed. A consequence of that formula
is the following: the probability that a particle leaves the origin with
some positive velocity $v_{0}$ at time $t=0$ but has not left the domain $%
\left\{ x>0\right\} $ at time $t$, decreases as $t^{-{1}/{4}}.$
%\texttt{...changed english a bit and see if
%it is fine...}
This power law in the problem which can be considered as a discrete analogue
of the accelerated random walk problem and its derivation can be found in
\cite{Si}.

The Laplace transform of propagators of the Orstein-Uhlenbeck process in a
half-line with absorbing boundary conditions at $x=0$ was computed in \cite%
{MW1} by using the Wiener-Hopf methods. This yields equations similar to (%
\ref{VFP}) but with an additional friction term. In the limit case in which
the friction coefficient tends to zero, the formulas in \cite{MW1} reduce to
the ones in \cite{McK}. The exit time of the Orstein-Uhlenbeck process in a
suitable asymptotic limit was considered in \cite{HDL}.

The power law decay of the solutions of the equation (\ref{VFP}) obtained by
means of the explicit formulas mentioned above is in contrast with the
exponential decay which we will derive in this paper for the case of
solutions in the interval $0<x<1.$ We remark that the mean exit time for an
accelerated Brownian particle in an interval $0<x<1$ was obtained in \cite%
{MP}.

The approach of the references above that concerned about the asymptotics of
the solutions of (\ref{VFP}) in the half-line or an interval is based on
explicit or semi-explicit representation formulas for the derived
quantities. The approach of this paper relies more on PDE arguments, like
mass balance equations and maximum principle arguments, applied to arbitrary
initial distributions and hopefully can be applied to more general cases.

The paper proceeds as follows. In Section 2 we develop a theory of the
existence and the uniqueness of weak solutions for (\ref{VFP})-(\ref{BC1}).
Section 3 contains a regularity theory which allows us to prove that the
solutions are $C^{\infty }$ outside the singular set and have suitable H\"{o}%
lder estimates near the singular set. Section 4 proves that every solution
of (\ref{VFP})-(\ref{BC1}) decreases exponentially in $L^{1}$ and $L^{\infty
}$ sense as $t\rightarrow \infty .$

%---------------------------------------------------------%---------------------------------------------------------

\section{Weak solutions in an interval}

%---------------------------------------------------------%---------------------------------------------------------

The goal of this section is to construct a weak solution to \eqref{VFP}-%
\eqref{BC1} for a given bounded and integrable initial data $f_{0}\in
L^{1}\cap L^{\infty }\left( \Omega \right) $ with $f_{0}\geq 0$.

%---------------------------------------------------------%---------------------------------------------------------

\subsection{Approximation}

%---------------------------------------------------------%---------------------------------------------------------

The first step for constructing a weak solution is to regularize the
equation \eqref{VFP}, in particular the transport term $v\del_{x}f$ near the
grazing boundary set $(x,v)\in\{(0,0),\left( 1,0\right) \}$. This will be
achieved by approximating it with cut-off functions. Define $\beta
_{\varepsilon }\left( v\right) \in C^{\infty }\left( -\infty ,\infty \right)
$ and $\eta _{\varepsilon }\left( x\right) \in C^{\infty }\left( 0,1 \right)
$ as follows.

\begin{equation*}
\beta _{\varepsilon }\left( v\right) =\left\{
\begin{array}{c}
0,~~~\left\vert v\right\vert <\varepsilon ^{2} \\
\in \left[ -v,v\right] ,~~~\varepsilon ^{2}\leq \left\vert v\right\vert \leq
2\varepsilon ^{2} \\
v,~~~\left\vert v\right\vert >2\varepsilon ^{2}%
\end{array}%
\right.
\end{equation*}%
\begin{equation*}
\eta _{\varepsilon }\left( x\right) =\left\{
\begin{array}{c}
0,~~~0\leq x<\varepsilon ,~1-\varepsilon <x\leq 1 \\
\in \left[ 0,1\right] ,~~~\varepsilon \leq x\leq 2\varepsilon
,~1-2\varepsilon \leq x\leq 1-\varepsilon \\
1,~~~\ \ \ \ \ \ \ \ \ \ \ 2\varepsilon <x<1-2\varepsilon .%
\end{array}%
\right.
\end{equation*}%
We will also approximate the diffusion term $f_{vv}$ by choosing a cut-off
function $\xi \left( \zeta \right) \in C_{c}^{\infty }\left( -\infty ,\infty
\right) $ such that

\begin{equation*}
\int_{-\infty }^{\infty }\xi \left( \zeta \right) ~d\zeta =1,~\int_{-\infty
}^{\infty }\zeta \xi \left( \zeta \right) ~d\zeta =0,~\int_{-\infty
}^{\infty }\zeta ^{2}\xi \left( \zeta \right) ~d\zeta =1.
\end{equation*}%
Then $f_{vv}$ can be approximated as

\begin{equation*}
Q^{\varepsilon }\left[ f\right] \left( x,v,t\right) :=\frac{2}{\varepsilon
^{2}}\int_{-\infty }^{\infty }\left[ f\left( x,v+\varepsilon \zeta ,t\right)
-f\left( x,v,t\right) \right] \xi \left( \zeta \right) d\zeta .
\end{equation*}%
From the Taylor's Theorem, we see that $Q^{\ep}[f]\rightarrow f_{vv}$ as $\ep%
\rightarrow 0$ at least formally.

%\smallskip

We consider the following approximate equation.%
\begin{equation}
f_{t}^{\varepsilon }+\left[ \beta _{\varepsilon }\left( v\right) +\left(
v-\beta _{\varepsilon }\left( v\right) \right) \eta _{\varepsilon }\left(
x\right) \right] f_{x}^{\varepsilon }=Q^{\varepsilon }\left[ f^{\varepsilon }%
\right]  \label{VFP_A}
\end{equation}%
with the same initial and boundary conditions:
\begin{equation}
f^{\varepsilon }|_{t=0}=f_{0},\quad f^{\ep}|_{\gamma _{T}^{-}}=0.
\label{VFP_ABC}
\end{equation}%
The approximate equation \eqref{VFP_A} is essentially a transport equation
combined with the jump process $Q^{\varepsilon },$ where the transport term
is truncated in the small neighborhood of the grazing set whose area is of $%
O(\ep^{3})$. We first study the regularized version (\ref{VFP_A}) of (\ref%
{VFP}) by using the method of characteristics and prove its well-posedness
by exploiting a weak maximum principle.

The corresponding equation of characteristics to (\ref{VFP_A}) reads as
follows.
\begin{eqnarray}
\frac{dX\left( s;x,v,t\right) }{ds} &=&\beta _{\varepsilon }\left( v\right)
+\left( v-\beta _{\varepsilon }\left( v\right) \right) \eta _{\varepsilon
}\left( X\left( s\right) \right) ,\ V\left( s;x,v,t\right) =v,~s<t~,
\label{char} \\
X\left( t;x,v,t\right) &=&x.  \notag
\end{eqnarray}%
For simplicity, we will use $X\left( s\right) $ and $V\left( s\right) $
instead of $X\left( s;x,v,t\right) $ and $V\left( s;x,v,t\right) $
respectively. Due to the cut-off functions and the absorbing boundary
condition, it is not trivial to write the backward characteristics
explicitly. To get around it, for a given $\left( x,v,t\right) ,$ we define $%
0<t_{0}\leq t$ if there exists $t_{0}=t_{0}\left( x,v,t\right) >0$ satisfying%
\begin{eqnarray*}
x &=&\int_{t_{0}}^{t}\left[ \beta _{\varepsilon }\left( v\right) +\left(
v-\beta _{\varepsilon }\left( v\right) \right) \eta _{\varepsilon }\left(
X\left( s\right) \right) \right] ds,\text{ or } \\
x &=&1+\int_{t_{0}}^{t}\left[ \beta _{\varepsilon }\left( v\right) +\left(
v-\beta _{\varepsilon }\left( v\right) \right) \eta _{\varepsilon }\left(
X\left( s\right) \right) \right] ds,
\end{eqnarray*}%
otherwise, set $t_{0}=0.$

We first calculate the Jacobian $J\left( s;t\right) ,$ for $t_{0}<s<t,$ of
the transformation%
\begin{equation*}
\left( x,v\right) \rightarrow \left( X\left( s\right) ,v\right) .
\end{equation*}%
Note that $J\left( s;t\right) \ $measures the change rate of the unit volume
in the phase space along the characteristics as follows. Let%
\begin{equation*}
J\left( s;t\right) :=\det \left(
\begin{array}{cc}
\frac{\partial X\left( s\right) }{\partial x} & \frac{\partial X\left(
s\right) }{\partial v} \\
0 & 1%
\end{array}%
\right) =\frac{\partial X\left( s\right) }{\partial x},
\end{equation*}%
where $X\left( s\right) $ is the characteristics defined in (\ref{char}).
Then the following lemma holds.

\begin{lemma}
\label{Jacobian}For $0\leq t_{0}<s<t\leq T$ with $T>0$ given, we have the
following estimates%
\begin{eqnarray}
1-O\left( \varepsilon \right) \leq \left\vert J\left( s;t\right) \right\vert
&=&\left\vert \frac{\partial X\left( s\right) }{\partial x}\right\vert \leq
1+O\left( \varepsilon \right) ,~  \label{Jacobian1} \\
1-O\left( \varepsilon \right) \leq \left\vert J\left( t;s\right) \right\vert
&=&\left\vert J\left( s;t\right) ^{-1}\right\vert \leq 1+O\left( \varepsilon
\right) ,  \label{Jacobian2}
\end{eqnarray}%
where $O\left( \varepsilon \right) =\varepsilon CTe^{\varepsilon CT}$.
\end{lemma}

\begin{proof}
We integrate (\ref{char}) over time from $t$ to $s$ and differentiate the
resulting equation with respect to $x$ to get%
\begin{equation*}
\frac{\partial X\left( s\right) }{\partial x}=1+\left( v-\beta _{\varepsilon
}\left( v\right) \right) \int_{t}^{s}\eta _{\varepsilon }^{\prime }\left(
X\left( \tau \right) \right) \frac{\partial X\left( \tau \right) }{\partial x%
}d\tau .
\end{equation*}

Then using the definitions of the cut-off functions $(v-\beta _{\varepsilon
}\left( v\right) )=O\left( \varepsilon ^{2}\right) ,\eta _{\varepsilon
}^{\prime }\left( X\left( \tau \right) \right) =O\left( 1/\varepsilon
\right) $ yields%
\begin{equation}
1-C\varepsilon \int_{s}^{t}\left\vert \frac{\partial X\left( \tau \right) }{%
\partial x}\right\vert d\tau \leq \left\vert \frac{\partial X\left( s\right)
}{\partial x}\right\vert \leq 1+C\varepsilon \int_{s}^{t}\left\vert \frac{%
\partial X\left( \tau \right) }{\partial x}\right\vert d\tau  \label{dX}
\end{equation}%
for some constant $C>0$. We now apply a standard Gronwall's inequality to get%
\begin{equation*}
e^{-\varepsilon C|t-s|}\leq \left\vert \frac{\partial X\left( s\right) }{%
\partial x}\right\vert \leq e^{\varepsilon C|t-s|}.
\end{equation*}

Since $\left\vert \frac{\partial X\left( s\right) }{\partial x}\right\vert
\leq e^{\varepsilon CT}$ for all $s<t\leq T,$ using this in (\ref{dX}) leads
to (\ref{Jacobian1})-(\ref{Jacobian2}).
\end{proof}

We now introduce the standard notion of a mild solution to \eqref{VFP_A}-%
\eqref{VFP_ABC}.

\begin{definition}
We say that $F\in C\left( \left[ 0,T\right] ;L^{1}\left( \Omega \right)
\right) \cap L^\infty \left( \left[ 0,T \right] ;L^{\infty }\left( \Omega
\right) \right) $ is a mild solution of (\ref{VFP_A})-(\ref{VFP_ABC}) if it
satisfies for every $t\in \left[ 0,T\right],$
\begin{equation}
F\left( x,v,t\right) =\bar{f}_{0}\left( X\left( t_{0}\right) ,v\right)
+\int_{t_{0}}^{t}Q^{\varepsilon }\left[ F\right] \left( X\left( s\right)
,v\right) ~ds=:\mathcal{T}\left[ F\right] \left( x,v,t\right) ,  \label{mild}
\end{equation}
where $\bar{f}_{0}\left( X\left( t_{0}\right) ,v\right) = f_0 \left( X\left(
{0}\right) ,v\right) $ if $t_{0}=0,$ otherwise $\bar{f}_{0}\left( X\left(
t_{0}\right) ,v\right) =0.$
\end{definition}

We show in the following lemma the existence and the uniqueness of a mild
solution (\ref{mild}) of (\ref{VFP_A})-(\ref{VFP_ABC}).

\begin{lemma}
\label{2} For any given $\varepsilon >0$ and any $f_{0}\in L^{1}\cap
L^{\infty }\left( \Omega \right) $ with $f_{0}\geq 0,$ there exist a time $%
T=T\left( \varepsilon \right) >0$ independent of $f_{0}$ and a unique mild
solution of (\ref{VFP_A})-(\ref{VFP_ABC}) in $\left[ 0,T\right] $.
\end{lemma}

\begin{proof}
We will show the existence by a fixed point argument. Let
\begin{align*}
\mathcal{U}& =\left\{ F\in C\left( \left[ 0,T\right] ;L^{1}\left( \Omega
\right) \right) \cap L^{\infty }\left( \left[ 0,T\right] ;L^{\infty }\left(
\Omega \right) \right) :\sup_{0\leq t\leq T}\left\Vert F\left( \cdot ,\cdot
,t\right) \right\Vert _{L^{\infty }\left( \Omega \right) }\leq 2\left\Vert
f_{0}\right\Vert _{L^{\infty }\left( \Omega \right) },\right. \\
&
\,\,\,\,\,\,\,\,\,\,\,\,\,\,\,\,\,\,\,\,\,\,\,\,\,\,\,\,\,\,\,\,\,\,\,\,\,\,%
\,\,\,\,\,\,\,\,\,\,\,\,\,\,\,\,\,\,\,\,\,\,\,\,\,\,\,\,\,\,\,\,\,\,\,\,\,\,%
\,\,\,\,\,\,\,\,\,\,\,\,\,\,\,\,\,\,\,\,\,\,\,\,\,\,\,\,\,\,\,\,\,\,\,\,\,\,%
\,\,\,\,\,\,\,\,\,\,\,\,\,\,\,\,\,\left. \sup_{0\leq t\leq T}\left\Vert
F\left( \cdot ,\cdot ,t\right) \right\Vert _{L^{1}\left( \Omega \right)
}\leq 2\Vert f_{0}\Vert _{L^{1}\left( \Omega \right) }\right\} \\
& \subset C\left( \left[ 0,T\right] ;L^{1}\left( \Omega \right) \right) \cap
L^{\infty }\left( \left[ 0,T\right] ;L^{\infty }\left( \Omega \right)
\right) .
\end{align*}%
We aim to show that $\mathcal{T}$ maps $\mathcal{U}$ into $\mathcal{U}$ and
is a contraction if $T=T\left( \varepsilon \right) $ is sufficiently small.
We first estimate $\mathcal{T}\left[ F\right] $ for $F\in L^{\infty }\left(
\Omega \right) .$ It is easy to see, for $F\in \mathcal{U,}$
\begin{equation*}
\left\vert Q^{\varepsilon }\left[ F\right] \left( X\left( s\right) ,v\right)
\right\vert \leq \frac{4}{\varepsilon ^{2}}\left\Vert F\left( \cdot ,\cdot
,s\right) \right\Vert _{L^{\infty }\left( \Omega \right) }\int_{-\infty
}^{\infty }\xi \left( \zeta \right) d\zeta \leq \frac{8}{\varepsilon ^{2}}%
\left\Vert f_{0}\right\Vert _{L^{\infty }\left( \Omega \right) }.
\end{equation*}%
Thus we get, for all $t\in \left[ 0,T\right] ,$%
\begin{equation*}
\left\Vert \mathcal{T}\left[ F\right] \left( \cdot ,\cdot ,t\right)
\right\Vert _{L^{\infty }\left( \Omega \right) }\leq \left\Vert \bar{f}%
_{0}\right\Vert _{L^{\infty }\left( \Omega \right) }+\frac{8T}{\varepsilon
^{2}}\left\Vert f_{0}\right\Vert _{L^{\infty }\left( \Omega \right) }\leq
2\left\Vert f_{0}\right\Vert _{L^{\infty }\left( \Omega \right) },
\end{equation*}%
if $\frac{8T}{\varepsilon ^{2}}<1.$ This implies $\mathcal{T}\left[ F\right]
\in L^{\infty }\left( \left[ 0,T\right] ;L^{\infty }\left( \Omega \right)
\right) .$ We now estimate $\mathcal{T}\left[ F\right] $ for $F\in
L^{1}\left( \Omega \right) .$ If $F\in \mathcal{U},$ then%
\begin{eqnarray*}
\left\Vert F\left( \cdot ,\cdot ,t\right) \right\Vert _{L^{1}\left( \Omega
\right) } &\leq &\int_{\Omega }\bar{f}_{0}\left( X\left( t_{0}\right)
,v\right) dxdv+\int_{\Omega }\int_{t_{0}}^{t}\left\vert Q^{\varepsilon }%
\left[ F\right] \right\vert \left( X\left( s\right) ,v\right) dsdxdv \\
&=&I+II.
\end{eqnarray*}%
We first estimate $I$ as follows. Using (\ref{Jacobian2}) in Lemma \ref%
{Jacobian} yields
\begin{eqnarray*}
I &=&\int_{\Omega }\bar{f}_{0}\left( X\left( t_{0}\right) ,v\right)
dxdv=\int_{\Omega _{1}}f_{0}\left( X\left( 0\right) ,v\right) dxdv \\
&=&\int_{\tilde{\Omega}_{1}}f_{0}\left( X\left( 0\right) ,v\right)
\left\vert J\left( t;0\right) \right\vert dX\left( 0\right) dv \\
&\leq &\left( 1+O\left( \varepsilon \right) \right) \int_{\Omega
}f_{0}\left( y,v\right) dydv \\
&\leq &\frac{3}{2}\left\Vert f_{0}\right\Vert _{L^{1}\left( \Omega \right) },
\end{eqnarray*}%
provided $T$ is chosen in such a way that $O\left( \varepsilon \right) \leq
1/2.$ Here we denote by $\Omega _{1},$ through the back-time
characteristics,
\begin{equation}
\Omega _{1}=\left\{ \left( x,v\right) \in \Omega ~|~\left( x,v,t\right)
\text{ connects with }\left( X\left( 0\right) ,v,0\right) \right\} \subset
\Omega .  \label{Omega1}
\end{equation}%
For $II$, if $F\in \mathcal{U},$ then
\begin{eqnarray*}
II &\leq &\frac{2}{\varepsilon ^{2}}\int_{-\infty }^{\infty }\xi \left(
\zeta \right) \int_{\Omega }\int_{t_{0}}^{t}\left[ \left\vert F\right\vert
\left( X\left( s\right) ,v+\varepsilon \zeta ,s\right) +\left\vert
F\right\vert \left( X\left( s\right) ,v,s\right) \right] dsdvdxd\zeta \\
&\leq &\frac{2}{\varepsilon ^{2}}\int_{0}^{t}\int_{-\infty }^{\infty }\xi
\left( \zeta \right) d\zeta \left[ \int_{\Omega }\left[ \left\vert
F\right\vert \left( X\left( s\right) ,v+\varepsilon \zeta ,s\right)
+\left\vert F\right\vert \left( X\left( s\right) ,v,s\right) \right]
\left\vert J\left( t;s\right) \right\vert dX\left( s\right) dv\right] ds \\
&\leq &\frac{4}{\varepsilon ^{2}}\int_{0}^{t}\int_{\Omega }\left\vert
F\right\vert \left( X\left( s\right) ,v,s\right) \left\vert J\left(
t;s\right) \right\vert dX\left( s\right) dvds \\
&\leq &\frac{4}{\varepsilon ^{2}}\left( 1+O\left( \varepsilon \right)
\right) \int_{0}^{t}ds\int_{\Omega }\left\vert F\right\vert \left( X\left(
s\right) ,v,s\right) dX\left( s\right) dv \\
&\leq &\frac{4\left( 1+O\left( \varepsilon \right) \right) T}{\varepsilon
^{2}}\left[ \sup_{0\leq t\leq T}\left\Vert F\left( \cdot ,\cdot ,t\right)
\right\Vert _{L^{1}\left( \Omega \right) }\right] \\
&\leq &\frac{8\left( 1+O\left( \varepsilon \right) \right) T}{\varepsilon
^{2}}\left\Vert f_{0}\right\Vert _{L^{1}\left( \Omega \right) } \\
&\leq &\frac{1}{2}\left\Vert f_{0}\right\Vert _{L^{1}\left( \Omega \right) }
\end{eqnarray*}%
if $T$ is further made in such a way that $\frac{8\left( 1+O\left(
\varepsilon \right) \right) T}{\varepsilon ^{2}}<1/2.$ Then by the estimates
of $I$ and $II,$ we obtain%
\begin{equation*}
\sup_{0\leq t\leq T}\left\Vert F\left( \cdot ,\cdot ,t\right) \right\Vert
_{L^{1}\left( \Omega \right) }\leq 2\Vert f_{0}\Vert _{L^{1}\left( \Omega
\right) }.
\end{equation*}%
For the continuity in time\ of $\mathcal{T}\left[ F\right] $ in $L^{1}\left(
\Omega \right) ,$ we use the absolute continuity in $L^{1}\left( \Omega
\right) $ of $L^{1}$ functions, the continuity in $t$ of $t_{0},X\left(
t_{0}\right) $ as functions of $t$, and the monotone convergence theorem. In
particular, we treat the continuity of $\mathcal{T}\left[ F\right] $ in $%
L^{1}\left( \Omega \right) $ as $t$ goes to $0$ in a more careful way, where
the monotone convergence theorem applies. This can be seen in the following
integral%
\begin{equation*}
\int_{0}^{1}\int_{x/t}^{\infty }f_{0}\left( x,v\right) dvdx\searrow 0\text{
as }t\searrow 0.
\end{equation*}%
Thus $\mathcal{T}\left[ F\right] \in C\left( \left[ 0,T\right] ;L^{1}\left(
\Omega \right) \right) $ and this implies that $\mathcal{T}$ maps $\mathcal{U%
}$ into $\mathcal{U}$. Similar arguments yield%
\begin{equation*}
\begin{split}
& \left\Vert \mathcal{T}\left[ F_{1}\right] \left( \cdot ,\cdot ,t\right) -%
\mathcal{T}\left[ F_{2}\right] \left( \cdot ,\cdot ,t\right) \right\Vert
_{L^{1}\left( \Omega \right) \cap L^{\infty }\left( \Omega \right) } \\
& \quad \quad \quad \leq \frac{8\left( 1+O\left( \varepsilon \right) \right)
T}{\varepsilon ^{2}}\sup_{0\leq t\leq T}\left\Vert F_{1}\left( \cdot ,\cdot
,t\right) -F_{2}\left( \cdot ,\cdot ,t\right) \right\Vert _{L^{1}\left(
\Omega \right) \cap L^{\infty }\left( \Omega \right) }
\end{split}%
\end{equation*}%
so that $\mathcal{T}$ is a contraction if $\frac{8\left( 1+O\left(
\varepsilon \right) \right) T}{\varepsilon ^{2}}<1.$ Therefore if $T=T\left(
\varepsilon \right) <\frac{\varepsilon ^{2}}{8\left( 1+O\left( \varepsilon
\right) \right) }$, then there exists a unique mild solution by a fixed
point theorem.
\end{proof}

We obtain in the following lemma the existence of solutions of (\ref{VFP_A}%
)-(\ref{VFP_ABC}) for an arbitrary time $T$, which does not depend on $%
\varepsilon $.

\begin{corollary}
\label{3}For any given $\varepsilon >0,T>0$ independent of $\varepsilon ,$
and $f_{0}\in L^{1}\cap L^{\infty }\left( \Omega \right) $ with $f_{0}\geq
0, $ there exists a unique mild solution of (\ref{VFP_A})-(\ref{VFP_ABC}) in
$\left[ 0,T\right] $.
\end{corollary}

\begin{proof}
We know the existence time $T=T\left( \varepsilon \right) $ and then we can
apply the same argument as in the proof of Lemma \ref{2} to extend the
existence time to the given time $T.$
\end{proof}

%---------------------------------------------------------%---------------------------------------------------------

\subsection{Well-posedness of the approximate equations}

%---------------------------------------------------------%---------------------------------------------------------

We will show in this subsection the existence and the uniqueness of weak
solutions of the approximate equation (\ref{VFP_A}) with the initial and
boundary conditions (\ref{VFP_ABC}). For that purpose, several maximum
principles will be used in this subsection. Maximum principle properties
have been extensively studied in the analysis of elliptic and parabolic
problems (See \cite{E}, \cite{F}, \cite{GT}, \cite{LSU}, \cite{PS}). We
adapt them to the corresponding definitions and results of this paper.

\begin{definition}
\label{weakSoln}We say that $F\in C\left( \left[ 0,T\right] ;L^{1}\left(
\Omega \right) \right) \cap L^{\infty }\left( \left[ 0,T\right] ;L^{\infty
}\left( \Omega \right) \right) $ is a weak solution of (\ref{VFP_A}) with (%
\ref{VFP_ABC}) if for every $t\in \left[ 0,T\right] $ and any test function $%
\psi \left( x,v,s\right) \in C^{1}\left( U_{t}\right) $ such that supp$%
\left( \psi \left( \cdot ,\cdot ,s\right) \right) \subset \left[ 0,1\right]
\times \left[ -R,R\right] $ for some $R>0$ and $\psi |_{\gamma _{t}^{+}}=0$,
it satisfies
\begin{align*}
& -\int_{U_{t}}F\left( x,v,s\right) \left[ \psi _{t}\left( x,v,s\right)
+\partial _{x}\left( \left[ \beta _{\varepsilon }\left( v\right) +\left(
v-\beta _{\varepsilon }\left( v\right) \right) \eta _{\varepsilon }\left(
x\right) \right] \psi \left( x,v,s\right) \right) \right] dxdvds \\
& +\int_{\Omega }F\left( x,v,t\right) \psi \left( x,v,t\right)
dxdv-\int_{\Omega }f_{0}\left( x,v\right) \psi \left( x,v,0\right) dxdv \\
& =\frac{2}{\varepsilon ^{2}}\int_{U_{t}}F\left( x,v,s\right) \int_{-\infty
}^{\infty }\left[ \psi \left( x,v-\varepsilon \zeta ,s\right) -\psi \left(
x,v,s\right) \right] \xi \left( \zeta \right) d\zeta dxdvds.
\end{align*}
\end{definition}

Then we have the lemma which states the existence of a weak solution to (\ref%
{VFP_A}).

\begin{lemma}
\label{weakSol}Let $T>0$ and let $f_{0}\in L^{1}\cap L^{\infty }\left(
\Omega \right) $ with $f_{0}\geq 0$. Then there exists a weak solution $%
f^{\varepsilon }\in C\left( \left[ 0,T\right] ;L^{1}\left( \Omega \right)
\right) \cap L^\infty \left( \left[ 0,T \right] ;L^{\infty }\left( \Omega
\right) \right) $ to (\ref{VFP_A})-(\ref{VFP_ABC}).
\end{lemma}

\begin{proof}
It is easy to see that a mild solution is a weak solution by multiplying to (%
\ref{mild}) a test function $\psi \left( x,v,s\right) \in C^{1}\left(
U_{t}\right) $ with a compact support and $\psi |_{\gamma _{t}^{+}}=0$ and
by integrating the resulting equation over $x,v,$ and $t$. We omit the
details.
\end{proof}

We will use smooth solutions of the adjoint problem to \eqref{VFP_A} as test
functions in the weak formulation although they may not have compact
supports, since the smooth solutions can be approximated by test functions
with compact supports in Definition \ref{weakSoln}. Thus the solutions
indeed satisfy the formula in Definition \ref{weakSoln}. We first show the
existence of such smooth solutions.\

Define%
\begin{equation}
\mathcal{\bar{L}}\psi :=\psi _{t}+\partial _{x}\left( \left[ \beta
_{\varepsilon }\left( v\right) +\left( v-\beta _{\varepsilon }\left(
v\right) \right) \eta _{\varepsilon }\left( x\right) \right] \psi \right) -%
\bar{Q}^{\varepsilon }\left[ \psi \right] \left( x,v,t\right) =0,~
\label{adjoint}
\end{equation}%
\begin{equation*}
\psi |_{t=T}=\psi _{T},~~\ \psi |_{_{\gamma _{T}^{+}}}=0,
\end{equation*}%
where $\bar{Q}\left[ \psi \right] \left( x,v,t\right) =\frac{2}{\varepsilon
^{2}}\int_{-\infty }^{\infty }\left[ \psi \left( x,v,t\right) -\psi \left(
x,v-\varepsilon \zeta ,t\right) \right] \xi \left( \zeta \right) d\zeta $
and $\psi _{T}\left( x,v\right) \in C^{\infty }\left( \Omega \right) $ given.

We take the data $\psi _{T}\left( x,v\right) $ so as to satisfy the
following compatibility condition:%
\begin{equation}
\psi _{T}\left( x,v\right) =0,\text{ \ for all }\left( x,v\right) \in N,
\label{comp}
\end{equation}%
where $N=\left[ \left\{ x^{2}+\beta _{\varepsilon }^{2}\left( v\right)
<\delta \right\} \cup \left\{ \left( x-1\right) ^{2}+\beta _{\varepsilon
}^{2}\left( v\right) <\delta \right\} \right] \cap \Omega $ for some $\delta
>0$ small.

\smallskip

\begin{lemma}
\label{psiExist}Let $\psi _{T}\left( x,v\right) \in C^{\infty }\left( \Omega
\right) $ be a smooth data at $t=T$ and satisfy (\ref{comp}) with $\psi
_{T}\geq 0.$ Then there exists a smooth solution $\psi \left( x,v,t\right)
\in C^{\infty }\left( U_{T}\right) $ satisfying (\ref{adjoint}).
\end{lemma}

\begin{proof}
We solve the adjoint problem to (\ref{VFP_A}) with the data $\psi _{T}\left(
x,v\right) $ at time $t=T$ and we find a smooth solution $\psi \left(
x,v,t\right) \in C^{\infty }\left( U_{T}\right) $ satisfying%
\begin{equation}
\psi _{t}+\partial _{x}\left( \left[ \beta _{\varepsilon }\left( v\right)
+\left( v-\beta _{\varepsilon }\left( v\right) \right) \eta _{\varepsilon
}\left( x\right) \right] \psi \right) -\bar{Q}\left[ \psi \right] \left(
x,v,t\right) =0,  \label{psi}
\end{equation}%
\begin{equation*}
\psi \left( x,v,T\right) =\psi _{T}\left( x,v\right) ,~~\psi |_{_{\gamma
_{T}^{+}}}=0.
\end{equation*}%
First we can show that there exists a mild solution $\psi \left(
x,v,t\right) \in C\left( \left[ 0,T\right] ;L^{1}\left( \Omega \right) )\cap
L^{\infty }(\left[ 0,T\right] ;L^{\infty }\left( \Omega \right) \right) ~$of
(\ref{psi}) by applying a method similar to that in Lemma \ref{2} and
Corollary \ref{3}. Then we can show that the mild solution of (\ref{psi}) is
indeed smooth, that is, $\psi \in C^{\infty }\left( U_{T}\right) .$ This can
be proved by observing that $\psi _{T}\left( x,v\right) \in C^{\infty
}\left( \Omega \right) $ satisfies the compatibility condition (\ref{comp}).
Then we can apply a fixed point argument to the integral representation for
the derivatives of $\psi $ by differentiating the integral representation
for $\psi $ itself.
\end{proof}

The next few lemmas concern the non-negativity of $\psi $ and $L^{1}$
estimate.

\begin{lemma}
\label{nonnegative}Let $\psi \left( x,v,t\right) \in C^{\infty }\left(
U_{T}\right) $ be a solution of the adjoint equation (\ref{adjoint})
backwards in time. Then it satisfies $\psi \left( x,v,t\right) \geq 0$ for
all $\left( x,v,t\right) \in U_{T}.$
\end{lemma}

\begin{proof}
Suppose that $\psi $ is as in the statement of the lemma. We define $\psi
^{k}=\psi +\left[ k-k\left( t-T\right) \right] e^{-L\left( t-T\right) }$
with $k>0$ small and $L$ depending on $\varepsilon $ to be made precise
later. Using (\ref{adjoint}) we obtain \
\begin{eqnarray*}
\mathcal{\bar{L}}\psi ^{k} &=&-ke^{-L\left( t-T\right) }-L\left[ k-k\left(
t-T\right) \right] e^{-L\left( t-T\right) }+\psi _{t} \\
&&+\left[ \beta _{\varepsilon }\left( v\right) +\left( v-\beta _{\varepsilon
}\left( v\right) \right) \eta _{\varepsilon }\left( x\right) \right]
\partial _{x}\psi +\left( v-\beta _{\varepsilon }\left( v\right) \right)
\eta _{\varepsilon }^{\prime }\left( x\right) \psi \\
&&+\left( k-k\left( t-T\right) \right) \left( v-\beta _{\varepsilon }\left(
v\right) \right) \eta _{\varepsilon }^{\prime }\left( x\right) e^{-L\left(
t-T\right) }-\bar{Q}^{\varepsilon }\left[ \psi \right] \\
&=&\left[ -k+\left( -L+\left( v-\beta _{\varepsilon }\left( v\right) \right)
\eta _{\varepsilon }^{\prime }\left( x\right) \right) \left( k-k\left(
t-T\right) \right) \right] e^{-L\left( t-T\right) }.
\end{eqnarray*}%
By using the fact that $\left( v-\beta _{\varepsilon }\left( v\right)
\right) $ and $\eta _{\varepsilon }^{\prime }\left( x\right) $ are compactly
supported as well as the fact that $\left( k-k\left( t-T\right) \right) >0$
for $t\leq T,$ it then follows that, choosing $L>0$ sufficiently large, we
obtain, for $t\leq T$,%
\begin{equation}
\mathcal{\bar{L}}\psi ^{k}\leq -ke^{-L\left( t-T\right) }<0.  \label{In1}
\end{equation}

We define domains $U_{T_{1},T}=\left\{ \left( x,v,t\right) :x\in \left(
0,1\right) ,\ v\in \mathbb{R},\ 0\leq T_{1}\leq t\leq T\right\} .$ Denote as
$T_{\ast }$ the infimum of the values of $T_{1}\geq 0$ such that $\psi
^{k}>0 $ in $U_{T_{1},T}.$ Notice that, since $\psi ^{k}\geq k>0$ at $t=T,$
we have $T_{\ast }<T$ and by definition $T_{\ast }\geq 0.$ By continuity of $%
\psi ^{k}$ we have $\psi ^{k}\geq 0$ in $\overline{U_{T_{\ast },T}}.$ We now
apply maximum principle arguments in this set. We may assume that the
minimum of $\psi ^{k}$ in $\overline{U_{T_{\ast },T}}$ is $0,$ since
otherwise $\psi ^{k}>0$ in $U_{T}$ and thus we are done. Suppose that the
minimum $0$ of $\psi ^{k}$ at $\overline{U_{T_{\ast },T}}$ is attained at
one interior point of this set. Then $\psi _{t}^{k}\left( x,v,t\right) =\psi
_{x}^{k}\left( x,v,t\right) =0$ while $\bar{Q}^{\varepsilon }\left[ \psi ^{k}%
\right] \left( x,v,t\right) \leq 0$ so that $\mathcal{\bar{L}}\psi
^{k}\left( x,v,t\right) \geq 0$ since $\psi ^{k}=0$ at the point. Thus it
cannot occur due to (\ref{In1}).\texttt{\ }On the other hand, we will prove
now that minimum $0$ cannot be obtained near the singular set. Suppose that $%
\psi ^{k}$ has its minimum in the set $\left\vert v\right\vert \leq
\varepsilon ,\ 0\leq x\leq \varepsilon ,\ 1-\varepsilon \leq x\leq 1.$ Then
the definition of $\beta _{\varepsilon }\left( v\right) $ and $\eta
_{\varepsilon }\left( x\right) $ implies that $\left[ \beta _{\varepsilon
}\left( v\right) +\left( v-\beta _{\varepsilon }\left( v\right) \right) \eta
_{\varepsilon }\left( x\right) \right] \psi _{x}^{k}=0$ in that set. We also
have $\psi _{t}^{k}\geq 0$ and $\bar{Q}^{\varepsilon }\left[ \psi ^{k}\right]
\leq 0$ at that minimum point so that $\mathcal{\bar{L}}\psi ^{k}\geq 0$ at
that point. This is again a contradiction. Therefore, the minimum of $\psi
^{k}$ in $\overline{U_{T_{\ast },T}}$ cannot be achieved in that set. Now
suppose that this minimum $0$ is attained at $\left( x,v,T_{\ast }\right) $
with $\left( x,v\right) \in \Omega .$ Then $\psi _{t}^{k}\left( x,v,T_{\ast
}\right) \geq 0,\psi _{x}^{k}\left( x,v,T_{\ast }\right) =0,\bar{Q}%
^{\varepsilon }\left[ \psi ^{k}\right] \left( x,v,T_{\ast }\right) \leq 0$
so that $\mathcal{\bar{L}}\psi ^{k}\left( x,v,T_{\ast }\right) \geq 0$.
Again it cannot happen. Suppose that $\psi ^{k}$ has its minimum $0$ at $%
\left( 0,v,t\right) $ with $v>0$ and $t>0$ or at $\left( 1,v,t\right) $ with
$v<0$ and $t>0.$ Then $\psi _{t}^{k}=0,\left[ \beta _{\varepsilon }\left(
v\right) +\left( v-\beta _{\varepsilon }\left( v\right) \right) \eta
_{\varepsilon }\left( x\right) \right] \psi _{x}^{k}\geq 0,\ \bar{Q}%
^{\varepsilon }\left[ \psi ^{k}\right] \leq 0$ so that $\mathcal{\bar{L}}%
\psi ^{k}\geq 0.$ This leads to a contradiction. Therefore $\psi ^{k}$ has
its minimum $k$ at $t=T$ or ${x=0,\ v<0}$ or ${x=1,v>0}$.

We then have obtained that $\psi ^{k}\geq k>0$ in $U_{T_{\ast },T}.$ If $%
T_{\ast }>0$ it would be possible to prove that $\psi ^{k}>0$ in some set $%
U_{T_{\ast }-\delta ,T}$ for some $\delta >0$ and this would contradict the
definition of $T_{\ast }.$ Therefore $T_{\ast }=0.$ We then have $\psi
^{k}>0 $ in $U_{T}$ for any $k>0.$ Taking the limit $k\rightarrow 0$ we
obtain $\psi \geq 0$ and complete the proof.
\end{proof}

We now have the following result.

\begin{lemma}
\label{nonnegative copy(1)}Let $\psi \left( x,v,t\right) \in C^{\infty
}\left( U_{T}\right) $ be a solution of the adjoint equation (\ref{adjoint})
and let $F\in C\left( \left[ 0,T\right] ;L^{1}\left( \Omega \right) )\cap
L^{\infty }(\left[ 0,T\right] ;L^{\infty }\left( \Omega \right) \right) $ be
a weak solution of (\ref{VFP_A}) with (\ref{VFP_ABC}). Then we have for any $%
t\in \left[ 0,T\right] ,$%
\begin{equation*}
\int_{\Omega }F\left( x,v,t\right) \psi \left( x,v,t\right)
dxdv-\int_{\Omega }F\left( x,v,0\right) \psi \left( x,v,0\right) dxdv=0.
\end{equation*}
\end{lemma}

\begin{proof}
It follows from Definition \ref{weakSoln} and from (\ref{adjoint}).
\end{proof}

Solutions to the adjoint problem satisfy the following property: $L^{1}$%
-norm of a solution of (\ref{adjoint}) does not increase backward in time.

\begin{lemma}
\label{psiInt}Let $\psi \left( x,v,t\right) \in C^{\infty }\left(
U_{T}\right) $ be a solution of the adjoint equation (\ref{adjoint})
backwards in time. Then it satisfies%
\begin{equation*}
\int_{\Omega }\psi \left( x,v,0\right) dxdv\leq \int_{\Omega }\psi \left(
x,v,T\right) dxdv.
\end{equation*}
\end{lemma}

\begin{proof}
We integrate (\ref{adjoint}) in $x$ and \thinspace $v$ to get%
\begin{eqnarray*}
\frac{d}{dt}\int_{\Omega }\psi \left( x,v,t\right) dxdv &=&\int_{-\infty
}^{\infty }\beta _{\varepsilon }\left( v\right) \psi \left( 0,v,t\right)
dv-\int_{-\infty }^{\infty }\beta _{\varepsilon }\left( v\right) \psi \left(
1,v,t\right) dv \\
&=&\int_{0}^{\infty }\beta _{\varepsilon }\left( v\right) \psi \left(
0,v,t\right) dv-\int_{-\infty }^{0}\beta _{\varepsilon }\left( v\right) \psi
\left( 1,v,t\right) dv.
\end{eqnarray*}%
Now using Lemma \ref{nonnegative}, we can deduce the lemma.
\end{proof}

\

We now go back to our original approximated Fokker-Planck problem and
establish the following maximum and minimum principles for weak solutions of
(\ref{VFP_A}). For that purpose, we first recall a basic lemma in measure
theory.

\begin{lemma}
\label{meas}Let $A$ be a set with a positive measure and let $\delta >0$
small be given. Then there exists a ball $B$ such that%
\begin{equation*}
\frac{\text{meas}\left( B\cap A\right) }{\text{meas}\left( B\right) }%
>1-\delta .
\end{equation*}
\end{lemma}

\begin{proof}
Let $A$ be a set with a positive measure. Let us denote as $B_{r}\left(
x_{0}\right) $ the ball $\left\{ \left\vert x-x_{0}\right\vert <r\right\} .$
Then we get%
\begin{equation*}
\frac{\text{meas}\left( B_{r}\left( x_{0}\right) \cap A\right) }{\text{meas}%
\left( B_{r}\left( x_{0}\right) \right) }\rightarrow 1\text{ ,\ \ }a.e\text{%
.\ }x_{0}\in A\text{\ as }r\rightarrow 0,
\end{equation*}%
where $\chi _{\Omega }$ is the characteristic function on the set $\Omega $
(cf. \cite{St}), whence the result follows, choosing a suitable $x_{0}\in A$
and $r>0$ small.
\end{proof}

We now present the maximum principle for weak solutions of (\ref{VFP_A}):

\begin{lemma}
\label{maxPrinciple} \label{cor1} If $f_{0}\in L^{\infty }\left( \Omega
\right) ,$ then a weak solution $f^{\varepsilon }$ to (\ref{VFP_A})-(\ref%
{VFP_ABC}) satisfies
\begin{equation*}
f^{\varepsilon }\left( x,v,t\right) \leq \left\Vert f_{0}\right\Vert
_{L^{\infty }\left( \Omega \right) }
\end{equation*}%
up to a measure zero set.
\end{lemma}

\begin{proof}
Let $M=\left\Vert f_{0}\right\Vert _{L^{\infty }\left( \Omega \right) }$,
then we want to prove that $f^{\varepsilon }\left( x,v,t\right) \leq M$ for
all $\left( x,v,t\right) \in \left( 0,1\right) \times \left( -\infty ,\infty
\right) \times \left( 0,\infty \right) $ almost everywhere. We prove this by
contradiction. Suppose the weak solution $f^{\varepsilon }\left( \cdot
,\cdot ,T\right) >M$ at time $t=T$ on a set with a positive measure. Then
there is $\kappa >0$ small such that $f^{\varepsilon }>M+\kappa $ on a set
with a positive set, say $A.$ Since $A$ has a positive measure, we apply
Lemma \ref{meas} to ensure that for any given $\delta >0$ small there exists
a ball $B$ $\subset \Omega $ such that%
\begin{equation}
\text{meas}\left( B\cap A\right) >\text{meas}\left( B\right) \left( 1-\delta
\right) .
\end{equation}%
Then we choose a smooth function $\psi _{T}\left( x,v\right) \in C^{\infty
}\left( \Omega \right) $ such that $\psi _{T}\geq 0$, supp$(\psi _{T})$ is
contained in $\bar{B}$, $\psi _{T}$ is uniformly bounded, and
\begin{equation}
\left\vert \int_{\Omega }\frac{\chi _{B}}{\text{meas}\left( B\right) }%
dxdv-\int_{\Omega }{\psi _{T}}\left( x,v\right) dxdv\right\vert <\delta .
\label{psiT}
\end{equation}%
Then by Lemma \ref{psiExist}, Lemma \ref{nonnegative}, Lemma \ref{psiInt},
there exists a smooth test function $\psi \left( x,v,t\right) \in C^{\infty
}\left( U_{T}\right) $ such that $\psi \geq 0$ and $\int_{\Omega }\psi
\left( x,v,0\right) dxdv\leq \int_{\Omega }\psi \left( x,v,T\right) dxdv$ and%
\begin{equation*}
\psi _{t}+\partial _{x}\left( \left[ \beta _{\varepsilon }\left( v\right)
+\left( v-\beta _{\varepsilon }\left( v\right) \right) \eta _{\varepsilon
}\left( x\right) \right] \psi \right) -\bar{Q}^{\varepsilon }\left[ \psi %
\right] \left( x,v,t\right) =0,~\psi \left( x,v,T\right) =\psi _{T}\left(
x,v\right) ,~~\ \psi |_{_{\gamma _{T}^{+}}}=0.
\end{equation*}%
Then we estimate $\int_{\Omega }f^{\varepsilon }\left( x,v,0\right) \psi
\left( x,v,0\right) dxdv$ and $\int_{\Omega }f^{\varepsilon }\left(
x,v,T\right) \psi \left( x,v,T\right) dxdv$ respectively.%
\begin{eqnarray}
\int_{\Omega }f^{\varepsilon }\left( x,v,0\right) \psi \left( x,v,0\right)
dxdv &\leq &M\int_{\Omega }\psi \left( x,v,0\right) dxdv\leq M\int_{\Omega
}\psi \left( x,v,T\right) dxdv  \notag \\
&=&M\int_{\Omega }\psi _{T}\left( x,v\right) dxdv\leq M\left( 1+\delta
\right) \leq M+C_{1}\delta ,  \label{IntatZero}
\end{eqnarray}%
where we used Lemma \ref{psiInt} and (\ref{psiT}). To estimate $\int_{\Omega
}f^{\varepsilon }\left( x,v,T\right) \psi \left( x,v,T\right) dxdv$%
\begin{equation*}
\int_{\Omega }f^{\varepsilon }\left( x,v,T\right) \psi \left( x,v,T\right)
dxdv=\int_{B}f^{\varepsilon }\left( x,v,T\right) \psi _{T}\left( x,v\right)
dxdv
\end{equation*}%
\begin{equation*}
=\int_{B\cap A}f^{\varepsilon }\left( x,v,T\right) \psi _{T}\left(
x,v\right) dxdv+\int_{B\backslash A}f^{\varepsilon }\left( x,v,T\right) \psi
_{T}\left( x,v\right) dxdv=:I+II.
\end{equation*}%
From the construction of $\psi _{T}$ and the definition of the set $A$ with
a positive measure, we deduce%
\begin{eqnarray*}
I &\geq &\left( M+\kappa \right) \int_{B\cap A}\psi _{T}\left( x,v\right)
dxdv \\
&\geq &\left( M+\kappa \right) \left[ \int_{B\cap A}\frac{\chi _{B}}{\text{%
meas}\left( B\right) }dxdv-\int_{B\cap A}\left\{ \frac{\chi _{B}}{\text{meas}%
\left( B\right) }-\psi _{T}\left( x,v\right) \right\} dxdv\right] \\
&\geq &\left( M+\kappa \right) \left[ \frac{\text{meas}\left( B\cap A\right)
}{\text{meas}\left( B\right) }-\delta \right] >\left( M+\kappa \right)
\left( 1-2\delta \right) =M+\kappa -2\left( M+\kappa \right) \delta .
\end{eqnarray*}%
For $II$, we use the fact that $f^{\varepsilon }$ and $\psi _{T}$ are
bounded and (\ref{meas}) to get%
\begin{equation*}
\left\vert II\right\vert \leq \left\Vert f^{\varepsilon }\right\Vert
_{L^{\infty }}\left\Vert \psi _{T}\right\Vert _{L^{\infty }}\text{meas}%
\left( B\backslash A\right) <\left\Vert f^{\varepsilon }\right\Vert
_{L^{\infty }}\left\Vert \psi _{T}\right\Vert _{L^{\infty }}\text{meas}%
\left( B\right) \delta .
\end{equation*}%
Combining the estimates for $I$ and $II,$ we obtain%
\begin{equation}
\int_{\Omega }f^{\varepsilon }\left( x,v,T\right) \psi \left( x,v,T\right)
dxdv\geq M+\kappa -C_{2}\delta ,  \label{IntatT}
\end{equation}%
where $C_{2}$ is independent of $\delta $ and depends only on $M+\kappa
,\left\Vert f^{\varepsilon }\right\Vert _{L^{\infty }},\left\Vert \psi
_{T}\right\Vert _{L^{\infty }},$ and meas$\left( B\right) .$

Now suppose $f^{\varepsilon }$ is a weak solution in Definition \ref%
{weakSoln}. Then we choose our test function $\psi \left( x,v,t\right) $ as
in the above to apply Lemma \ref{nonnegative copy(1)} and get%
\begin{equation*}
\int_{\Omega }f^{\varepsilon }\left( x,v,T\right) \psi \left( x,v,T\right)
dxdv=\int_{\Omega }f^{\varepsilon }\left( x,v,0\right) \psi \left(
x,v,0\right) dxdv.
\end{equation*}%
Then using the estimates (\ref{IntatZero}), (\ref{IntatT}), we deduce%
\begin{equation*}
M+\kappa -C_{2}\delta \leq \int_{\Omega }f^{\varepsilon }\left( x,v,T\right)
\psi \left( x,v,T\right) dxdv=\int_{\Omega }f^{\varepsilon }\left(
x,v,0\right) \psi \left( x,v,0\right) dxdv\leq M+C_{1}\delta .
\end{equation*}%
Thus if $\delta $ is chosen sufficiently small in such a way that $\left(
C_{1}+C_{2}\right) \delta <\kappa /2$, we can get a contradiction.
Therefore, $\left\Vert f^{\varepsilon }\right\Vert _{L^{\infty }\left(
U_{T}\right) }\leq M$, that is, $f^{\varepsilon }$ satisfies the maximum
principle. We complete the proof.
\end{proof}

We also derive the non-negativity of solutions, which is the minimum
principle for weak solutions.

\begin{lemma}
\label{minicomparison}If $f_{0}\left( x,v\right) \geq 0$ and $f^{\varepsilon
}\left( x,v,t\right) $ is a weak solution of (\ref{VFP_A})-(\ref{VFP_ABC}),
then%
\begin{equation*}
f^{\varepsilon }\left( x,v,t\right) \geq 0
\end{equation*}%
up to a measure zero set.
\end{lemma}

\begin{proof}
It can be proved similarly to the proof of Lemma \ref{maxPrinciple} and we
skip its proof.
\end{proof}

The above two lemmas together provide the uniqueness result for solutions of
the approximate Fokker-Planck equation (\ref{VFP_A}) with (\ref{VFP_ABC}).

\begin{corollary}
\label{uniqueness-e}Let $f_{1}^{\ep},f_{2}^{\ep}$ be two weak solutions of (%
\ref{VFP_A}) with the same initial and boundary conditions (\ref{VFP_ABC}).
Then $f_{1}^{\ep}=f_{2}^{\ep}$ in $L^{\infty }(U_{T})$.
\end{corollary}

\begin{proof}
It follows from Lemma \ref{maxPrinciple} and Lemma \ref{minicomparison}.
\end{proof}

Next, we show that the total mass is non-increasing up to a correction $O(\ep%
^{2})$. The subtlety here is that we are not able to use the integration by
parts since regularity has yet to be shown.

\begin{lemma}
\label{mass} Let $f_{0}\in L^{1}\cap L^{\infty }\left( \Omega \right) $ and $%
f_{0}\geq 0$ given. Then the total mass of a mild solution $f^{\varepsilon }$
for $0\leq t\leq T$ satisfies the following inequality.
\begin{equation}
\int_{\Omega }f^{\varepsilon }\left( x,v,t\right) dxdv\leq \int_{\Omega
}f_{0}(x,v)dxdv+r^{\varepsilon }(t)+q^{\ep}(t),  \label{ineq}
\end{equation}%
where $0\leq r^{\ep}(t)\leq 8\ep^{4}CTe^{\ep CT}\Vert f_{0}\Vert _{L^{\infty
}\left( \Omega \right) }$, $|q^{\ep}(t)|\leq 64C\varepsilon ^{2}T^{2}(1+\ep %
T)e^{\ep CT}\Vert f_{0}\Vert _{\infty },$ and $C$ is independent of $T$ and $%
\ep$.
\end{lemma}

\begin{proof}
We will estimate $L^{1}$ norm from the integral form \eqref{mild}. By
integrating \eqref{mild} over $\Omega $, we obtain
\begin{equation}
\begin{split}
\int_{\Omega }f^{\varepsilon }(x,v,t)dxdv& =\int_{\Omega }\bar{f}_{0}\left(
X\left( t_{0}\right) ,v\right) dxdv+\int_{\Omega
}\int_{t_{0}}^{t}Q^{\varepsilon }\left[ f\right] \left( X\left( s\right)
,v\right) ~dsdxdv \\
& =:I+II.
\end{split}
\label{integral_f}
\end{equation}

%As the first step to prove \eqref{ineq}, we assume that
%\begin{equation}
%1/2\leq \frac{\del X(s)}{\del X(0)}\leq 2\;\text{ for }\;0\leq s\leq \tilde{t%
%}  \label{assum}
%\end{equation}%
%for some $\tilde{t}>0$. And we will show the following inequality for $I$ under the
%assumption \eqref{assum}:
%\begin{equation}
%I\leq \int_{\Omega
%}f_{0}(x,v)dxdv+r^{\varepsilon }(t),  \label{con}
%\end{equation}%
%where
%\begin{equation*}
%0\leq r^{\varepsilon }(t)\leq 8\varepsilon ^{2}\Vert f_{0}\Vert _{L^{\infty
%}\left( \Omega \right) } .
%\end{equation*}%
%We remark that the condition \eqref{assum} is not too restrictive because we
%can find such a $\tilde{t}=O(1)$ independent of $\ep$. In fact, this can be
%deduced from the following formula:
%\begin{equation*}
%\frac{\del X(t)}{\del X(0)}=1+\int_{0}^{t}(v-\beta _{\ep}(v))\eta _{\ep%
%}^{\prime }(X(s))\frac{\del X(s)}{\del X(0)}ds
%\end{equation*}%
%by noting that $|(v-\beta _{\ep}(v))\eta _{\ep}^{\prime }(X(s))|\leq C\ep$
%where $C$ is independent of $v,X(s)$ and $\ep$.
We start with the estimation of $I$.
\begin{equation}
I=\int_{\Omega }\bar{f}_{0}\left( X\left( t_{0}\right) ,v\right)
dxdv=\int_{\Omega _{1}}f_{0}(X(0),v)dxdv,  \label{2.3}
\end{equation}%
where $\Omega _{1}$ is defined in (\ref{Omega1}), that is, we only treat the
particles $\left( x,v,t\right) $ which connect to $t_{0}=0.$ Our goal is to
show that
\begin{equation}
I\leq \int_{\Omega }f_{0}(x,v)dxdv+r^{\varepsilon }(t).  \label{con}
\end{equation}%
By the change of variables,
\begin{eqnarray*}
I &=&\int_{\Omega _{1}}f_{0}\left( X\left( 0\right) ,v\right) dxdv=\int_{%
\tilde{\Omega}_{1}}f_{0}\left( X\left( 0\right) ,v\right) \left\vert J\left(
t;0\right) \right\vert dX\left( 0\right) dv, \\
&&
\end{eqnarray*}%
where
\begin{equation*}
J(t;0)=1+\frac{\partial }{\partial X(0)}\left[ \int_{0}^{t}(v-\beta
_{\varepsilon }\left( v\right) )\,\eta _{\varepsilon }\left( X\left(
s\right) \right) ds\right] .
\end{equation*}%
Notice that $\frac{\partial }{\partial X(0)}\left[ \int_{0}^{t}(v-\beta
_{\varepsilon }\left( v\right) )\,\eta _{\varepsilon }\left( X\left(
s\right) \right) ds\right] =O(\varepsilon )$ as in Lemma \ref{Jacobian} and
hence $J(t;0)>0$ for sufficiently small $\varepsilon >0$. Thus we see that
\begin{equation}
\begin{split}
I& =\int_{\Omega _{1}}f_{0}\left( X\left( 0\right) ,v\right) dxdv=\int_{%
\tilde{\Omega}_{1}}f_{0}\left( X\left( 0\right) ,v\right) \left\vert J\left(
t;0\right) \right\vert dX\left( 0\right) dv \\
& =\int_{\tilde{\Omega}_{1}}f_{0}\left( X(0),v\right) dX(0)dv \\
& \quad +\int_{\tilde{\Omega}_{1}}\frac{\partial }{\partial X(0)}\left[
\int_{0}^{t}(v-\beta _{\varepsilon }\left( v\right) )\,\eta _{\varepsilon
}\left( X\left( s\right) \right) ds\right] f_{0}\left( X\left( 0\right)
,v\right) dX(0)dv \\
& =:(i)+(ii).
\end{split}
\label{2.6}
\end{equation}%
For $(i)$, we write it as
\begin{equation}
(i)=\int_{\Omega }f_{0}(X(0),v)dX(0)dv-\int_{\Omega \setminus \tilde{\Omega}%
_{1}}f_{0}(X(0),v)dX(0)dv\leq \int_{\Omega }f_{0}(x,v)dxdv.  \label{L1}
\end{equation}%
For the second term $(ii)$, we write it as
\begin{equation*}
\begin{split}
(ii)& =\int_{\tilde{\Omega}_{1}}\frac{\partial }{\partial X(0)}\left[
\int_{0}^{t}(v-\beta _{\varepsilon }\left( v\right) )\,\eta _{\varepsilon
}\left( X\left( s\right) \right) ds\right] f_{0}\left( X\left( 0\right)
,v\right) dX(0)dv \\
& =\int_{\tilde{\Omega}_{1}\cap \{v>0\}}+\int_{\tilde{\Omega}_{1}\cap
\{v<0\}}=:(ii)^{+}+(ii)^{-}.
\end{split}%
\end{equation*}%
Let us first estimate $(ii)^{+}$. Notice that the first factor $\frac{%
\partial }{\partial {X(0)}}\left[ \int_{0}^{t}(v-\beta _{\varepsilon }\left(
v\right) )\eta _{\varepsilon }\left( X\left( s\right) \right) ds\right] =0$
for $X(s)\leq \varepsilon $ or $2\varepsilon \leq X(s)\leq 1-2\varepsilon $
or $X(s)\geq 1-\varepsilon $ or $v>2\varepsilon ^{2}$ due to the cutoffs.
Thus we only need to integrate it over the set $A_{0}^{+}=\{(X(0),v)\big|%
\,\varepsilon <X(s)<2\varepsilon \text{ and }0<v<2\varepsilon ^{2}\text{ for
some }0<s<t\,\}$ and $A_{1}^{+}=\{(X(0),v)\big|\,1-2\varepsilon
<X(s)<1-\varepsilon \text{ and }0<v<2\varepsilon ^{2}\text{ for some }%
0<s<t\,\}$. On one hand, we see that $\frac{\partial }{\partial X(0)}\left[
\int_{0}^{t}(v-\beta _{\varepsilon }\left( v\right) )\eta _{\varepsilon
}\left( X\left( s\right) \right) ds\right] \leq 0$ over $A_{1}^{+}$ since
the cutoff function $\eta _{\varepsilon }$ decreases for $1-2\varepsilon
<X(s)<1-\varepsilon $. Hence, we deduce that
\begin{equation}
\begin{split}
(ii)^{+}& \leq \int_{A_{0}^{+}\cap \tilde{\Omega}_{1}}\frac{\partial }{%
\partial X(0)}\left[ \int_{0}^{t}(v-\beta _{\varepsilon }\left( v\right)
)\eta _{\varepsilon }\left( X\left( s\right) \right) ds\right]
f_{0}(X(0),v)dX(0)dv \\
& =:r_{+}^{\ep}(t),
\end{split}
\label{L11}
\end{equation}%
where $r_{+}^{\ep}(t)\geq 0$. On the other hand, since $\beta _{\varepsilon
}\left( v\right) +(v-\beta _{\varepsilon }\left( v\right) )\,\eta
_{\varepsilon }\left( X\left( s\right) \right) >0$, we see that if $X(0)\geq
2\ep$, then $X(s)>2\ep$ for all $s>0$. Therefore, %\begin{equation*}
\begin{equation*}
A_{0}^{+}\subset \{(X(0),v)\big|\,0<X(0)<2\ep,\;0<v<2\ep^{2}\,\}=:B^{+},
\end{equation*}%
where $|B^{+}|=4\ep^{3}$. Thus
\begin{equation}
\begin{split}
0\leq r_{+}^{\varepsilon }(t)& \leq \Vert f_{0}\Vert _{L^{\infty }\left(
\Omega \right) }\int_{B^{+}}\frac{\partial }{\partial X(0)}\left[
\int_{0}^{t}(v-\beta _{\varepsilon }\left( v\right) )\eta _{\varepsilon
}\left( X\left( s\right) \right) ds\right] dX(0)dv \\
& \leq \Vert f_{0}\Vert _{L^{\infty }\left( \Omega \right) }\int_{B^{+}}\ep %
Cte^{\ep Ct}dX(0)dv\;\text{ by Lemma }\ref{Jacobian} \\
& \leq \ep Cte^{\ep Ct}\Vert f_{0}\Vert _{L^{\infty }\left( \Omega \right)
}|B^{+}| \\
& \leq 4\ep^{4}Cte^{\ep Ct}\Vert f_{0}\Vert _{L^{\infty }\left( \Omega
\right) }.
\end{split}
\label{2.7}
\end{equation}%
For $(ii)^{-}$, as in the case of $v>0$, we first note that the first factor
vanishes $X(s)\leq \varepsilon $ or $2\varepsilon \leq X(s)\leq
1-2\varepsilon $ or $X(s)\geq 1-\varepsilon $ or $v<-2\varepsilon ^{2}$.
Thus we only need to integrate it over the set $A_{0}^{-}=\{(X(0),v)\big|%
\,\varepsilon <X(s)<2\varepsilon \text{ and }-2\varepsilon ^{2}<v<0\text{
for some }0<s<t\,\}$ and $A_{1}^{-}=\{(X(0),v)\big|\,1-2\varepsilon
<X(s)<1-\varepsilon \text{ and }-2\varepsilon ^{2}<v<0\text{ for some }%
0<s<t\,\}$. This time, we see that $\frac{\partial }{\partial X(0)}\left[
\int_{0}^{t}(v-\beta _{\varepsilon }\left( v\right) )\eta _{\varepsilon
}\left( X\left( s\right) \right) ds\right] \leq 0$ over $A_{0}^{-}$ and
hence we deduce that
\begin{equation}
\begin{split}
(ii)^{-}& \leq \int_{A_{1}^{-}\cap \tilde{\Omega}_{1}}\frac{\partial }{%
\partial X(0)}\left[ \int_{0}^{t}(v-\beta _{\varepsilon }\left( v\right)
)\eta _{\varepsilon }\left( X\left( s\right) \right) ds\right]
f_{0}(X(0),v)dX(0)dv \\
& =:r_{-}^{\ep}(t).
\end{split}
\label{L12}
\end{equation}%
Moreover, we see that %\begin{equation*}
\begin{equation*}
A_{1}^{-}\subset \{(X(0),v)\big|\,1-2\ep<X(0)<1,\;-2\ep^{2}<v<0\,\}=:B^{-},
\end{equation*}%
where $|B^{-}|=4\ep^{3}$ and hence by the same argument as in \eqref{2.7},
we obtain
\begin{equation}
0\leq r_{-}^{\ep}(t)\leq 4\ep^{4}Cte^{\ep Ct}\Vert f_{0}\Vert _{L^{\infty
}\left( \Omega \right) }.  \label{L13}
\end{equation}

Combining \eqref{2.3}, \eqref{2.6}-\eqref{L13}, we obtain
\begin{equation*}
I\leq \int_{\Omega }f_{0}(X(0),v)dX(0)dv+r^{\ep}(t),
\end{equation*}%
where $r^{\ep}(t):=r_{+}^{\ep}(t)+r_{-}^{\ep}(t)$ and complete the proof of %
\eqref{con}.

We now turn to the second term $II$ in \eqref{integral_f}.
\begin{equation*}
\begin{split}
II& =\int_{\Omega }\int_{t_{0}}^{t}Q^{\varepsilon }\left[ f\right] \left(
X\left( s\right) ,v\right) dsdxdv \\
& =\frac{2}{\varepsilon ^{2}}\left\{ \int_{\Omega
}\int_{t_{0}}^{t}\int_{-\infty }^{\infty }\left\{ f^{\varepsilon }\left(
X(s),v+\varepsilon \zeta ,s\right) -f^{\varepsilon }\left( X(s),v,s\right)
\right\} \xi \left( \zeta \right) d\zeta dsdxdv\right\} \\
& =\frac{2}{\varepsilon ^{2}}\big\{(iii)-(iv)\big\}.
\end{split}%
\end{equation*}%
We will estimate $(iv)$ first. Recall that for $t_{0}<s<t$,
\begin{equation*}
x=X(s)+\int_{s}^{t}\left[ \beta _{\varepsilon }(v)+(v-\beta _{\varepsilon
}(v))\eta _{\epsilon }(X(\tau ))\right] d\tau .
\end{equation*}%
Consider the following change of variables: $(x,v,s)\rightarrow (y=X(s),v,s)$%
. Then by the definition of $t_{0}$, we see that the domain of integration
changes from $(x,v,s)\in (0,1)\times (-\infty ,\infty )\times (t_{0},t)$ to $%
(y=X(s),v,s)\in (0,1)\times (-\infty ,\infty )\times (0,t)$. Hence, by
Fubini's Theorem,
\begin{equation*}
\begin{split}
(iv)& =\int_{0}^{t}\int_{\Omega }f^{\varepsilon }\left( y,v,s\right) \left(
1+\frac{\partial }{\partial y}\left[ \int_{s}^{t}\left[ \beta _{\varepsilon
}(v)+(v-\beta _{\varepsilon }(v))\eta _{\epsilon }(X(\tau ))\right] d\tau %
\right] \right) dydvds \\
& =\int_{0}^{t}\int_{\Omega }f^{\varepsilon }\left( y,v,s\right) dydvds \\
& \quad \quad +\int_{0}^{t}\int_{\Omega }f^{\varepsilon }\left( y,v,s\right)
\left[ \int_{s}^{t}(v-\beta _{\varepsilon }(v))\eta _{\epsilon }^{\prime
}(X(\tau ))\frac{\partial X(\tau )}{\partial y}d\tau \right] dydvds \\
& =\int_{0}^{t}\int_{\Omega }f^{\varepsilon }\left( y,v,s\right)
dydvds+(iv)_{2}.
\end{split}%
\end{equation*}%
Now for the second term $(iv)_{2}$, since $|(v-\beta _{\ep}(v))\eta _{\ep%
}^{\prime }(X(\tau ))|\leq C\ep$ and $(v-\beta _{\ep}(v))\eta _{\ep}^{\prime
}(X(\tau ))=0$ if $|v|>2\ep^{2}$ or $X(\tau )\leq \varepsilon $ or $%
2\varepsilon \leq X(\tau )\leq 1-2\varepsilon $ or $X(\tau )\geq
1-\varepsilon $, by using Lemma \ref{Jacobian} and Lemma \ref{cor1}
\begin{equation*}
\begin{split}
|(iv)_{2}|& =\left\vert \int_{0}^{t}\int_{A}\int_{s}^{t}(v-\beta
_{\varepsilon }(v))\eta _{\epsilon }^{\prime }(X(\tau ))\frac{\partial
X(\tau )}{\partial y}f^{\varepsilon }\left( y,v,s\right) d\tau
dydvds\right\vert \\
& \leq {C\varepsilon t^{2}e^{\ep Ct}\Vert f_{0}\Vert _{\infty }}\left\vert
\int_{A}dydv\right\vert ,
\end{split}%
\end{equation*}%
where $A=\{(y,v):|v|\leq 2\varepsilon ^{2}\text{ and }\varepsilon <X(\tau
)<2\varepsilon \text{ or }1-2\varepsilon <X(\tau )<1-\varepsilon \text{ for
some }s<\tau <t\}$. Note that $A\subset \{(y,v):|v|\leq 2\varepsilon ^{2}%
\text{ and }y\leq 2\varepsilon (1+\ep t)\text{ or }1-y\leq 2\varepsilon (1+%
\ep t)\}$ and hence
\begin{equation*}
|(iv)_{2}|\leq 16C\varepsilon ^{4}t^{2}(1+\ep t)e^{\ep Ct}\Vert f_{0}\Vert
_{\infty }\,,
\end{equation*}%
where C is a constant independent of $t$ and $\ep$.

For $(iii)$, we consider the following change of variables: $y=X(s)$, $%
w=v+\varepsilon \zeta $ and $s=s$. Then similarly, by Fubini's Theorem we
obtain
\begin{equation*}
\begin{split}
(iii)& =\int_{-\infty }^{\infty }\int_{0}^{t}\int_{\Omega }f^{\varepsilon
}\left( y,w,s\right) \left( 1+\frac{\partial }{\partial y}\left[
\int_{s}^{t}(w-\varepsilon \zeta -\beta _{\varepsilon }(w-\varepsilon \zeta
))\eta _{\epsilon }(X_{w-\varepsilon \zeta }(\tau ))d\tau \right] \right)
dydwds\xi (\zeta )d\zeta \\
& =\int_{0}^{t}\int_{\Omega }f^{\varepsilon }\left( y,w,s\right) dydwds \\
& \quad +{\int_{-\infty }^{\infty }\int_{0}^{t}\int_{\Omega }f^{\varepsilon
}\left( y,w,s\right) \left[ \int_{s}^{t}(w-\varepsilon \zeta -\beta
_{\varepsilon }(w-\varepsilon \zeta ))\eta _{\epsilon }^{\prime
}(X_{w-\varepsilon \zeta }(\tau ))\frac{\partial X_{w-\varepsilon \zeta
}(\tau )}{\partial y}d\tau \right] dydwds\xi (\zeta )d\zeta } \\
& =\int_{0}^{t}\int_{\Omega }f^{\varepsilon }\left( y,w,s\right)
dydwds+(iii)_{2}.
\end{split}%
\end{equation*}%
As in the previous case, it is easy to see that the second term $|(iii)_{2}|$
is bounded by $16C\varepsilon ^{4}t^{2}(1+\ep t)e^{\ep Ct}\Vert f_{0}\Vert
_{\infty }$. Therefore, we deduce that
\begin{equation*}
II=\int_{\Omega }\int_{t_{0}}^{t}Q^{\varepsilon }\left[ f\right] \left(
X\left( s\right) ,v\right) dsdxdv=\frac{2}{\ep^{2}}\big\{(iii)_{2}-(iv)_{2}%
\big\}=:q^{\ep}(t),
\end{equation*}%
where $|q^{\ep}(t)|\leq 64C\varepsilon ^{2}t^{2}(1+\ep t)e^{\ep Ct}\Vert
f_{0}\Vert _{\infty }$ for a constant $C$ independent of $t$ and $\ep$. This
completes the proof of the lemma.
\end{proof}

%----------------------------------------

\subsection{Well-posedness of weak solutions for the Fokker-Planck equation}

%----------------------------------------

We now obtain a weak limit of the approximating sequence $\left\{
f^{\varepsilon }\right\} $ as a candidate for a weak solution$.$

\begin{proposition}
\label{weakConv}Let $T>0$ and let $f_{0}\in L^{1}\cap L^{\infty }\left(
\Omega \right) $ with $f_{0}\geq 0$. Then $f^{\varepsilon }$ converges
weakly to $f$ in $L^{\infty }\left( \left[ 0,T\right] ;L^{1}\cap L^{\infty
}\left( \Omega \right) \right) $ with $f\geq 0.$ Moreover, the following
holds.
\begin{equation*}
f\left( x,v,t\right) \leq \left\Vert f_{0}\right\Vert _{L^{\infty }\left(
\Omega \right) }\text{ and }\int_{\Omega }f\left( x,v,t\right) dxdv\leq
\int_{\Omega }f_{0}(x,v)dxdv.
\end{equation*}
\end{proposition}

\begin{proof}
It follows from Lemma \ref{cor1}, Lemma \ref{minicomparison}, Lemma \ref%
{mass}, and by taking the limit in the weak toplogy as $\varepsilon
\rightarrow 0.$
\end{proof}

We now prove Theorem \ref{MainTheorem1}, which is the existence of a weak
solution of (\ref{VFP}) in the following.

\begin{proof}[Proof of Theorem \protect\ref{MainTheorem1} (Existence)]
We show that \label{weak copy(1)}$f$ in the Proposition \ref{weakConv}\
above is indeed a weak solution of (\ref{VFP}). We first show the weak
continuity of $f(t)$. Let a test function $\psi (x,v,t)$ compactly supported
be given and $t_{1},t_{2}\in \lbrack 0,T]$ and let $\varepsilon >0$ be
given. Note that for the solution $f^{\varepsilon }$ of the regularized
problem (\ref{VFP_A})-(\ref{VFP_ABC}),
\begin{equation*}
\begin{split}
& \int_{\Omega }f^{\varepsilon }(t_{1})\psi (t_{1})dxdv-\int_{\Omega
}f^{\varepsilon }(t_{2})\psi (t_{2})dxdv \\
& \quad =\int_{\Omega }f^{\varepsilon }(t_{1})[\psi (t_{1})-\psi
(t_{2})]dxdv+\int_{\Omega }[f^{\varepsilon }(t_{1})-f^{\varepsilon
}(t_{2})]\psi (t_{2})dxdv. \\
&
\end{split}%
\end{equation*}%
Since $f^{\varepsilon }$ is in $C\left( \left[ 0,T\right] ;L^{1}\left(
\Omega \right) \right) \cap L^{\infty }\left( \left[ 0,T\right] ;L^{\infty
}\left( \Omega \right) \right) $, both terms on the right-hand side can be
made small uniformly in $\varepsilon $ if $|t_{1}-t_{2}|$ is sufficiently
small. Since $f^{\varepsilon }$ converges weakly to $f$, by taking $%
\varepsilon \rightarrow 0$, we deduce the weak continuity of $f(t)$ and in
particular, $\int_{\Omega }f(t)\psi (t)dxdv$ is well-defined for each $t\in
\lbrack 0,T]$. Now by using Lemma \ref{weakSol} and Proposition \ref%
{weakConv} and by noting that in $L^{1}\left( U_{t}\right) ,$ as $%
\varepsilon \rightarrow 0$
\begin{align*}
\frac{2}{\varepsilon ^{2}}\int_{-\infty }^{\infty }\left[ \psi \left(
x,v-\varepsilon \zeta ,s\right) -\psi \left( x,v,s\right) \right] \xi \left(
\zeta \right) d\zeta & \rightarrow \psi _{vv}\left( x,v,s\right) , \\
\partial _{x}\left( \left[ \beta _{\varepsilon }\left( v\right) +\left(
v-\beta _{\varepsilon }\left( v\right) \right) \eta _{\varepsilon }\left(
x\right) \right] \psi \left( x,v,s\right) \right) & \rightarrow \partial
_{x}\left( v\psi \left( x,v,s\right) \right) ,
\end{align*}%
we can easily deduce that $f$ is indeed a weak solution.
\end{proof}

We can now derive the maximum and minimum principles for the Fokker-Planck
operator.

Define

\begin{equation}
\mathcal{M}f:=f_{t}+vf_{x}-f_{vv}.  \label{M}
\end{equation}

We can deduce the maximum and minimum principle for weak solutions of (\ref%
{VFP}).

\begin{lemma}
\label{maxPrinciple copy(1)} The operator $\mathcal{M}$ defined in (\ref{M})
has a maximum principle for weak solutions: Let $f\in L^{\infty }(\left[ 0,T%
\right] ;L^{1}\cap L^{\infty }\left( \Omega \right) )$ be a weak solution of
(\ref{VFP})-(\ref{BC1}) as in Definition \ref{weak-solution} and let $%
f_{0}\in L^{1}\cap L^{\infty }\left( \Omega \right) $ with $f_{0}\geq 0$.
Then we have%
\begin{equation*}
f\left( x,v,t\right) \leq \left\Vert f_{0}\right\Vert _{L^{\infty }\left(
\Omega \right) }.
\end{equation*}%
up to a measure zero set.
\end{lemma}

\begin{proof}
It is analogous to Lemma \ref{maxPrinciple}.
\end{proof}

\begin{lemma}
\label{maxPrinciple copy(2)} The operator $\mathcal{M}$ defined in (\ref{M})
has a minimum principle for weak solutions: Let $f\in L^{\infty }(\left[ 0,T%
\right] ;L^{1}\cap L^{\infty }\left( \Omega \right) )$ be a weak solution of
(\ref{VFP})-(\ref{BC1}) as in Definition \ref{weak-solution} and let $%
f_{0}\in L^{1}\cap L^{\infty }\left( \Omega \right) $ with $f_{0}\geq 0,$
then we have%
\begin{equation*}
f\left( x,v,t\right) \geq 0
\end{equation*}%
up to a measure zero set.
\end{lemma}

\begin{proof}
It is analogous to Lemma \ref{minicomparison}.
\end{proof}

We then show the uniqueness of weak solutions of \eqref{VFP}-(\ref{BC1}) as
follows.

\begin{proof}[Proof of Theorem 1.2 (Uniqueness)]
Let $f_{1},f_{2}$ be two weak solutions of (\ref{VFP}) with the same initial
and boundary conditions (\ref{id})-(\ref{BC1}). Then $f_{1}=f_{2}$ in $%
L^{\infty }(U_{T})$. The proof is analogous to Corollary \ref{uniqueness-e}.
\end{proof}

%---------------------------------------------------------%---------------------------------------------------------

Before we conclude this section, we present the maximum and minimum
principles for classical solutions of (\ref{VFP}). Some of the results will
be used in Section 4 after we establish the regularity of weak solutions.

We begin with the maximum principle in bounded domains. Let $%
U_{T}^{1}=\left( 0,1\right) \times \left( -L,L\right) \times \left(
0,T\right) $ with $L>0,$ $T>0.$

\begin{lemma}
\label{maxStrongBnd}The operator $\mathcal{M}$ defined in (\ref{M}) has a
maximum principle: Let $f$ be in $C_{x,v,t}^{1,2,1}\left( U_{T}^{1}\right)
\cap C\left( \bar{U}_{T}^{1}\right) $ and satisfy $\mathcal{M}f\leq 0$, then
$f$ attains its maximum either at $t=0$ or at $x=0,v>0$ or at $x=1,v<0$ or
at $v=\pm L$.
\end{lemma}

\begin{proof}
We extend $f$ to the domain outside of $\left( -L,L\right) $ with respect to
$v$ by defining it to be zero. First we assume that $\mathcal{M}f<0$. We
prove this, case by case. First we suppose the solution $f$ attains its
maximum at an interior point $\left( x,v,t\right) \in U_{T}^{1}$. Then $%
f_{t}(x,v,t)=f_{x}\left( x,v,t\right) =0$ while $f_{vv}\leq 0$ so that $%
\mathcal{M}f\left( x,v,t\right) \geq 0$. Thus it cannot occur. Now suppose
its maximum is attained at $\left( x,v,T\right) $ with $\left( x,v\right)
\in \Omega .$ Then $f_{x}\left( x,v,T\right) =0$, $f_{t}(x,v,T)\geq 0$ and $%
f_{vv}\leq 0$ so that $\mathcal{M}f\left( x,v,t\right) \geq 0$. Again, it
cannot happen. Lastly we suppose that $f$ has its maximum at $\left(
0,v,t\right) $ with $v<0$ and $0<t$ or at $\left( 1,v,t\right) $ with $v>0$
and $0<t.$ Then for $x=0,v<0$ or $x=1,v>0,$ we have $f_{t}(x,v,t)\geq 0,$ $%
vf_{x}\left( 0,v,t\right) \geq 0\,,$ and $f_{vv}\leq 0$ so that $\mathcal{M}%
f\left( x,v,t\right) \geq 0$. Therefore $f$ has a maximum at the kinetic
boundary. Next we will show the lemma in the case of $\mathcal{M}f\leq 0.$
In this case, we use $g=f-kt,$ $k>0$ to derive the maximum principle by
letting $k\rightarrow 0.$ We skip the details. This completes the proof.
\end{proof}

In order to prove the maximum principle for unbounded domains with respect
to $v$, we will find a barrier function near $v=\infty .$

\begin{lemma}
\label{barrier}There exists \ a super-solution $\phi \left( v,t\right) \in
C_{v,t}^{2,1}\left( \left( -\infty ,\infty \right) \times \left( 0,T\right)
\right) $ with $\phi \geq 0$ of $\mathcal{M}$ satisfying $\mathcal{M}\phi
\geq 0$ and $\phi \rightarrow \infty $ as $\left\vert v\right\vert
\rightarrow \infty $ uniformly in $t\in \left[ 0,T\right] $. This will be
called a barrier function at infinity.
\end{lemma}

\begin{proof}
We find $\phi $ of the form $\phi \left( v,t\right) =a_{0}\left( t\right)
+a_{1}\left( t\right) v^{2}$ and plug in it into $\mathcal{M}\phi \geq 0$.
Then we have%
\begin{equation*}
\frac{d}{dt}a_{0}\left( t\right) +\frac{d}{dt}a_{1}\left( t\right) v^{2}\geq
2a_{1}\left( t\right) .
\end{equation*}

Indeed, there exist many super-solutions which satisfy the equation above.
For instance, $a_{0}\left( t\right) =e^{2t},a_{1}\left( t\right) =e^{2t}$ or
$a_{0}\left( t\right) =2t+1,a_{1}\left( t\right) =1$ will work.
\end{proof}

We can obtain the maximum principle for unbounded domains in $v$.

\begin{lemma}
\label{maxStrong}The operator $\mathcal{M}$ defined in (\ref{M}) has a
maximum principle: Let $f$ be in $C_{x,v,t}^{1,2,1}\left( U_{T}\right) \cap
C\left( \bar{U}_{T}\right) $ and satisfy $\mathcal{M}f\leq 0$, then $f$
attains its maximum at its kinetic boundary $\Gamma _{T}^{-}$, i.e., either
at $t=0$ or at $x=0,v>0$ or at $x=1,v<0$.
\end{lemma}

\begin{proof}
Fix $w\in \mathbb{R},~\lambda >0,$ and let $\phi $ be a barrier function at
infinity as in Lemma \ref{barrier}. We then define%
\begin{equation*}
g\left( x,v,t\right) =f\left( x,v,t\right) -\lambda \phi \left( v-w,t\right)
.
\end{equation*}%
Then we have $\mathcal{M}g\leq 0.$ Thus Lemma \ref{maxStrongBnd} applies to $%
g$ in $U_{T}^{1}=\left( 0,1\right) \times \left( w-r,w+r\right) \times
\left( 0,T\right) $ for any $r>0$. For any $0<x<1,-\infty <v<\infty ,$ $%
g\left( x,v,0\right) =f\left( x,v,0\right) -\lambda \phi \left( v-w,0\right)
\leq f\left( x,v,0\right) $. For $0<v<\infty ,0<t<T,$ $g\left( 0,v,t\right)
=f\left( 0,v,t\right) -\lambda \phi \left( v-w,t\right) \leq f\left(
0,v,t\right) $ and for $-\infty <v<0,0<t<T,$ $g\left( 1,v,t\right) =f\left(
1,v,t\right) -\lambda \phi \left( v-w,t\right) \leq f\left( 1,v,t\right) .$
Now for $v=w\pm r,$ $g\left( x,v,t\right) =f\left( x,v,t\right) -\lambda
\phi \left( \pm r,t\right) \leq \sup_{x,v}f\left( x,v,0\right) $ if $r>0$ is
sufficiently large. Thus $g\left( x,w,t\right) \leq \sup_{\Gamma
_{T}^{-}}f\left( x,v,t\right) $ for all $\left( x,w,t\right) \in \bar{U}%
_{T}. $ Letting $\lambda \rightarrow 0,$ we get $f\left( x,w,t\right) \leq
\sup_{\Gamma _{T}^{-}}f\left( x,v,t\right) $ for all $\left( x,w,t\right)
\in \bar{U}_{T}.$ This completes the proof.
\end{proof}

We can also derive a minimum principle for the Fokker-Planck operation.

\begin{lemma}
\label{minPrinciple copy(1)} The operator $\mathcal{M}$ defined in (\ref{M})
has a minimum principle: Let $f$ be in $C_{x,v,t}^{1,2,1}\left( U_{T}\right)
\cap C\left( \bar{U}_{T}\right) $ and satisfy $\mathcal{M}f\geq 0.$ Then $f$
has a minimum at its kinetic boundary $\Gamma _{T}^{-}$, i.e., either at $%
t=0 $ or at $x=0,v>0$ or at $x=1,v<0$.
\end{lemma}

\begin{proof}
It is analogous to Lemma \ref{maxStrong}.
\end{proof}

\section{Regularity}

%---------------------------------------------------------%---------------------------------------------------------

In this section, we will establish the regularity (hypoellipticity) of the
weak solutions obtained in the previous section by studying the adjoint
problem. As a preparation, we first recall the fundamental solution to the
forward Fokker-Planck equation in the whole space.

%---------------------------------------------------------%---------------------------------------------------------

\subsection{Preliminaries}

\subsubsection{Fundamental solution in the absence of boundary}

%---------------------------------------------------------%---------------------------------------------------------

The fundamental solution $G$ for the Fokker-Planck equation \eqref{VFP} in
the whole space $(x,v,t)\in \mathbb{R}\times\mathbb{R}\times \mathbb{R}_+$
is given by (for instance, see \cite{IK64})
\begin{equation}
\begin{split}  \label{G}
G(x,v,t;\xi,\nu,\tau)&=G(x-\xi,v,\nu,t-\tau) \\
&=\frac{3^{\frac{1}{2}}}{2\pi (t-\tau)^2}\exp \left(-\frac{%
3|x-\xi-(t-\tau)(v+\nu)/2|^2}{(t-\tau)^3} -\frac{|v-\nu|^2}{4(t-\tau)}%
\right).
\end{split}%
\end{equation}

Any solution of the linear problem \eqref{VFP} with initial data $f_{0}\in
L^{1}\cap L^{\infty }(\mathbb{R}^{2})$ has the integral expression
\begin{equation*}
f(x,v,t)=\int_{\mathbb{R}^{2}}G(x,v,t;\xi ,\nu ,0)f_{0}(\xi ,\nu )d\xi d\nu
\,.
\end{equation*}

\bigskip

For the purpose of our study on the boundary hypoelliptic regularity, we
further investigate in the following lemma the behavior of the fundamental
solution $G$ near $x=0$ in the integral form. Since the behavior near $x=1$
can be studied in a similar manner, we will skip it.

\begin{lemma}
\label{limit} The fundamental solution $G$ given in \eqref{G} satisfies the
following right limit at $x=0$ in the integral form. Let $t>0$ and $v>0$ be
given fixed positive time and velocity and let $\lambda $ be a given
integrable and continuous function. Then we have
\begin{equation}
\begin{split}
& \lim_{x\rightarrow 0^{+}}\int_{0}^{t}ds\int_{\mathbb{R}}dw\lambda
(w,s)G(x,v,w,t-s) \\
& \quad \quad \quad =\frac{\lambda (v,t)}{v}+\int_{0}^{t}ds\int_{\mathbb{R}%
}dw\lambda (w,s)G(0,v,w,t-s).
\end{split}
\label{int}
\end{equation}
\end{lemma}

\begin{proof}
Let $\epsilon >0$ be a given arbitrary small number. Let $x>0$ be
sufficiently small, say $x=o(\epsilon )$ so that $\lim_{\epsilon \rightarrow
0^{+}}x/\epsilon =0$. We now divide the integral on the left-hand side of %
\eqref{int} into two parts. One is when $t-\epsilon <s<t$ and the other is
its complement: $0<s<t-\epsilon $. Then
\begin{equation}
\begin{split}
& \int_{0}^{t}ds\int_{\mathbb{R}}dw\lambda (w,s)G(x,v,w,t-s) \\
& =\int_{t-\epsilon }^{t}ds\int_{\mathbb{R}}dw\lambda
(w,s)G(x,v,w,t-s)+\int_{0}^{t-\epsilon }ds\int_{\mathbb{R}}dw\lambda
(w,s)G(x,v,w,t-s) \\
& =:(i)+(ii).
\end{split}
\label{iii}
\end{equation}%
We first compute the first part $(i)$. Since $G$ is integrable, we can write
$(i)$ as
\begin{equation}
\begin{split}
(i)& =\lambda (v,t)\int_{t-\epsilon }^{t}ds\int_{\mathbb{R}}dwG(x,v,w,t-s) \\
& \quad +\int_{t-\epsilon }^{t}ds\int_{\mathbb{R}}dw\left( \lambda
(w,s)-\lambda (v,t)\right) G(x,v,w,t-s)=:\lambda (v,t)\,(i)_{1}+(i)_{2}.
\end{split}
\label{(i)}
\end{equation}

The second integral $(i)_2$ in \eqref{(i)} converges to 0 as $%
\epsilon\rightarrow 0$. This can be established by splitting the integral
into two parts: $|w-v|<\delta$ and $|w-v|>\delta$. When $w$ is close enough
to $v$: $|w-v|<\delta$, one can use the continuity of $\lambda$ and when $%
|w-v|>\delta$ and $0<t-s<\epsilon$, it can be shown that the remaining
integral can be made as small as possible. We omit the details.

In what follows, we will show that the first integral in \eqref{(i)}, $%
(i)_{1}\rightarrow 1/v$ as $\epsilon \rightarrow 0$. To do so, we use the
explicit expression for $G$: first, substituting $t-s$ with a new variable
(denoted again as $s$) in $(i)_{1}$, we see that
\begin{equation*}
(i)_{1}=\frac{3^{\frac{1}{2}}}{2\pi }\int_{0}^{\epsilon }\int_{\mathbb{R}}%
\frac{1}{s^{2}}\exp \left( -\frac{3|x-s(v+w)/2|^{2}}{s^{3}}-\frac{|v-w|^{2}}{%
4s}\right) dwds.
\end{equation*}%
We first note that the $w$-integral can be explicitly computed: for any
fixed $x,s,v>0$,
\begin{equation*}
\begin{split}
& \int_{\mathbb{R}}\exp \left( -\frac{3|x-s(v+w)/2|^{2}}{s^{3}}-\frac{%
|v-w|^{2}}{4s}\right) dw \\
& =\int_{\mathbb{R}}\exp \left( -\frac{|w-v-\frac{3}{2}(\frac{x}{s}-v)|^{2}}{%
s}-\frac{3|\frac{x}{s}-v|^{2}}{4s}\right) dw \\
& =\exp \left( -\frac{3|{x}-sv|^{2}}{4s^{3}}\right) \int_{\mathbb{R}}\exp
\left( -\frac{|\tilde{w}|^{2}}{s}\right) d\tilde{w}=\pi ^{\frac{1}{2}}s^{%
\frac{1}{2}}\exp \left( -\frac{3|{x}-sv|^{2}}{4s^{3}}\right)
\end{split}%
\end{equation*}%
and thus $(i)_{1}$ can be rewritten as
\begin{equation}
(i)_{1}=\frac{3^{\frac{1}{2}}}{2\pi ^{\frac{1}{2}}}\int_{0}^{\epsilon }\frac{%
1}{s^{\frac{3}{2}}}\exp \left( -\frac{3|{x}-sv|^{2}}{4s^{3}}\right) ds=\frac{%
3^{\frac{1}{2}}}{2\pi ^{\frac{1}{2}}}\underbrace{\left( \int_{0}^{(1-\alpha )%
\frac{x}{v}}+\int_{(1-\alpha )\frac{x}{v}}^{(1+\alpha )\frac{x}{v}%
}+\int_{(1+\alpha )\frac{x}{v}}^{\epsilon }\right) }%
_{=:(i)_{11}+(i)_{12}+(i)_{13}}  \label{(i)1}
\end{equation}%
for a sufficiently small positive number $\alpha $ to be determined. We will
estimate each term respectively. For $(i)_{11}$, notice that $0<s<(1-\alpha )%
\frac{x}{v}$ implies $x-vs>\frac{\alpha }{1-\alpha }vs>0$, which in turn
implies $-|x-sv|^{2}<-(\frac{\alpha }{1-\alpha })^{2}v^{2}s^{2}$. Hence
\begin{equation*}
(i)_{11}=\int_{0}^{(1-\alpha )\frac{x}{v}}\frac{1}{s^{\frac{3}{2}}}\exp
\left( -\frac{3|{x}-sv|^{2}}{4s^{3}}\right) ds<\int_{0}^{(1-\alpha )\frac{x}{%
v}}\frac{1}{s^{\frac{3}{2}}}\exp \left( -\frac{3\alpha ^{2}v^{2}}{4(1-\alpha
)^{2}}\frac{1}{s}\right) ds.
\end{equation*}%
Now to see the dependence on $\alpha $ of the integral, we make the
substitution, $t=\frac{s}{\alpha ^{2}v^{2}}$:
\begin{equation*}
(i)_{11}<\int_{0}^{\frac{(1-\alpha )}{\alpha ^{2}v^{2}}\frac{x}{v}}\frac{1}{%
\alpha ^{3}v^{3}t^{\frac{3}{2}}}\exp \left( -\frac{3}{4(1-\alpha )^{2}}\frac{%
1}{t}\right) \alpha ^{2}v^{2}dt=\frac{1}{\alpha v}\int_{0}^{\frac{(1-\alpha
)x}{\alpha ^{2}v^{3}}}\frac{1}{t^{\frac{3}{2}}}\exp \left( -\frac{3}{%
4(1-\alpha )^{2}t}\right) dt.
\end{equation*}%
Since $v$ is bounded away from zero, the integrand is uniformly bounded, and
therefore, we deduce that
\begin{equation}
(i)_{11}\leq \frac{C_{1}x}{\alpha ^{3}v^{4}}\;\text{ for some uniform
constant }C_{1}.  \label{(i)11}
\end{equation}%
For $(i)_{13}$, $(1+\alpha )\frac{x}{v}<s$ implies $vs-x>\frac{\alpha }{%
1+\alpha }vs>0$ and thus $-|x-vs|^{2}<-\frac{\alpha ^{2}}{(1+\alpha )^{2}}%
v^{2}s^{2}$. Hence we get
\begin{equation*}
(i)_{13}=\int_{(1+\alpha )\frac{x}{v}}^{\epsilon }\frac{1}{s^{\frac{3}{2}}}%
\exp \left( -\frac{3|{x}-sv|^{2}}{4s^{3}}\right) ds<\int_{(1+\alpha )\frac{x%
}{v}}^{\epsilon }\frac{1}{s^{\frac{3}{2}}}\exp \left( -\frac{3\alpha
^{2}v^{2}}{4(1+\alpha )^{2}}\frac{1}{s}\right) ds.
\end{equation*}%
As before, letting $t=\frac{s}{\alpha ^{2}v^{2}}$, we obtain
\begin{equation*}
(i)_{13}<\int_{\frac{(1+\alpha )}{\alpha ^{2}v^{2}}\frac{x}{v}}^{\frac{%
\epsilon }{\alpha ^{2}v^{2}}}\frac{1}{\alpha ^{3}v^{3}t^{\frac{3}{2}}}\exp
\left( -\frac{3}{4(1+\alpha )^{2}}\frac{1}{t}\right) \alpha ^{2}v^{2}dt<%
\frac{1}{\alpha v}\int_{0}^{\frac{\epsilon }{\alpha ^{2}v^{2}}}\frac{1}{t^{%
\frac{3}{2}}}\exp \left( -\frac{3}{4(1+\alpha )^{2}t}\right) dt
\end{equation*}%
and hence we deduce that
\begin{equation}
(i)_{13}\leq \frac{C_{2}\epsilon }{\alpha ^{3}v^{3}}\;\text{ for some
uniform constant }C_{2}.  \label{(i)13}
\end{equation}%
The integral $(i)_{12}$ is when $s$ is very close to $x/v$. Notice that $%
(1-\alpha )\frac{x}{v}<s<(1+\alpha )\frac{x}{v}$ as well as $|s-\frac{x}{v}|<%
\frac{\alpha x}{v}$. Therefore, we can bound $(i)_{12}$ as
\begin{equation}
I^{-}\leq (i)_{12}=\int_{|s-\frac{x}{v}|<\frac{\alpha x}{v}}\frac{1}{s^{%
\frac{3}{2}}}\exp \left( -\left\vert \frac{\sqrt{3}v}{2s^{\frac{3}{2}}}%
\left( s-\frac{x}{v}\right) \right\vert ^{2}\right) ds\leq I^{+},
\label{(i)12}
\end{equation}%
where
\begin{equation*}
I^{\pm }=\int_{|s-\frac{x}{v}|<\frac{\alpha x}{v}}\frac{1}{((1\mp \alpha )%
\frac{x}{v})^{\frac{3}{2}}}\exp \left( -\left\vert \frac{\sqrt{3}v}{2((1\pm
\alpha )\frac{x}{v})^{\frac{3}{2}}}\left( s-\frac{x}{v}\right) \right\vert
^{2}\right) ds.
\end{equation*}%
Letting $z=\frac{\sqrt{3}v}{2((1\pm \alpha )\frac{x}{v})^{\frac{3}{2}}}%
\left( s-\frac{x}{v}\right) $, we rewrite $I^{\pm }$ as follows:
\begin{equation}
I^{\pm }=\frac{2}{\sqrt{3}v}\left( \frac{1\pm \alpha }{1\mp \alpha }\right)
^{\frac{3}{2}}\int_{|z|<\frac{\sqrt{3}\alpha v^{\frac{3}{2}}}{2(1\pm \alpha
)^{\frac{3}{2}}\sqrt{x}}}e^{-z^{2}}dz.  \label{I}
\end{equation}%
From \eqref{(i)11}-\eqref{I}, it is clear that if we choose $\alpha
=\epsilon ^{1/4}$, the following holds
\begin{equation*}
\lim_{\epsilon \rightarrow 0}(i)_{11}=\lim_{\epsilon \rightarrow
0}(i)_{13}=0\;\text{ and }\;\lim_{\epsilon \rightarrow 0}I^{\pm
}=\lim_{\epsilon \rightarrow 0}(i)_{12}=\frac{2\sqrt{\pi }}{\sqrt{3}v}
\end{equation*}%
and therefore, from \eqref{(i)1} we conclude that
\begin{equation*}
(i)_{1}\rightarrow \frac{1}{v}\text{ as }\epsilon \rightarrow 0.
\end{equation*}

It now remains to compute the limit of $(ii)$ in \eqref{iii}. It is clear
that for $t-s>\epsilon >x=o(\epsilon )$, there exists a uniform constant $%
C_{3}>0$ so that
\begin{equation*}
\begin{split}
\frac{1}{(t-s)^{2}}\exp & \left( -\frac{3|x-(t-s)(v+w)/2|^{2}}{(t-s)^{3}}-%
\frac{|v-w|^{2}}{4(t-s)}\right) \\
& \leq \frac{1}{(t-s)^{2}}\exp \left( -C_{3}\frac{|v+w|^{2}}{(t-s)}-\frac{%
|v-w|^{2}}{4(t-s)}\right)
\end{split}%
\end{equation*}%
and therefore, by the dominated convergence theorem, we can pass to the
limit:
\begin{equation*}
\lim_{\epsilon \rightarrow 0}(ii)=\int_{0}^{t}ds\int_{\mathbb{R}}dw\lambda
(w,s)G(0,v,w,t-s).
\end{equation*}%
This completes the proof of Lemma.
\end{proof}

As Lemma \ref{limit} indicates, the fundamental solution for the
Fokker-Planck equation is different from the heat kernel: due to the
hyperbolic (transport) nature of the Fokker-Planck equation (\ref{VFP}), $G$
displays more singular behavior in $x$ than the non-degenerate variable $v$.
This lemma will be used crucially for the boundary hypoellipticity result.

\

%---------------------------------------------------------%---------------------------------------------------------

\subsubsection{Adjoint problem}

%---------------------------------------------------------%---------------------------------------------------------

We recall from Definition \ref{weak-solution} that the weak solutions
\newline
$f\in L^{\infty }\left( \left[ 0,T\right] ;L^{1}\cap L^{\infty }\left(
\Omega \right) \right) $ to (\ref{VFP})-(\ref{BC1}) with initial data $%
f_{0}\in L^{1}\cap L^{\infty }\left( \Omega \right) $ satisfy
\begin{equation}
\int_{U_{t}}\mathcal{M}^{\ast }(\psi )f+\int_{\Omega }\psi
(t)f(t)=\int_{\Omega }\psi (0)f_{0},  \label{wf}
\end{equation}%
where
\begin{equation}
\mathcal{M}^{\ast }(\psi )=-\psi _{t}-v\psi _{x}-\psi _{vv}  \label{M*}
\end{equation}%
for every $t\in \left[ 0,T\right] $ and any test function $\psi \left(
x,v,s\right) \in C^{1}\left( U_{t}\right) $ satisfying supp$\left( \psi
\left( \cdot ,\cdot ,s\right) \right) \subset \left[ 0,1\right] \times \left[
-R,R\right] $ for some $R>0$ and $\psi |_{\gamma _{t}^{+}}=0$.

The adjoint problem is to solve the following adjoint equation (the backward
Fokker-Planck equation)
\begin{equation}
\mathcal{M}^{\ast }(\psi )=0\text{ for }t<T  \label{back}
\end{equation}%
for a given data at $t=T$ so that $\psi \in L^1\cap L^{\infty }(\Omega)$.
% and satisfies the adjoint boundary condition $\psi(0,v,t)=0$ for $v<0$.

Notice that $g(x,v,t):=f(x,-v,T-t)$, where $f$ is the solution to the
forward Fokker-Planck equation, solves the backward Fokker-Planck %
\eqref{back} for $t<T$. Thus the transformation $t\rightarrow t_{0}-t$ and $%
v\rightarrow -v$ in $G$ yields the fundamental solution to the backward
Fokker-Planck equation.

%---------------------------------------------------------%---------------------------------------------------------

\subsection{Hypoellipticity away from the singular set}

\subsubsection{Interior hypoellipticity}

%---------------------------------------------------------%---------------------------------------------------------

The goal of this subsection is to prove the following interior
hypoellipticity:

\begin{proposition}
\label{Prop}Let $f$ be the weak solution with given initial data $f_{0}\in
L^{1}\cap L^{\infty }(\Omega )$. Then for each $t>0$, $f\in H_{\text{loc}%
}^{k,m}(\Omega ),$ where $H^{k,m}=H_{x,v}^{k,m}$.
\end{proposition}

\begin{proof}
Let $(x_{0},v_{0},t_{0}),$ where $x_{0}>0$ and $t_{0}>0,$ be given. Suppose
the data $\varphi $ at $t=t_{0}$ is supported in the interior: supp$%
\,\varphi \subset B_{\rho }(x_{0},v_{0})\subset \Omega $ and $\varphi \in
C(B_{\rho })$. We consider the following backward Fokker-Planck equation in
the whole space:
\begin{equation}
\mathcal{M}^{\ast }(\phi )=0\;\text{ for }t<t_{0}\text{ where }\;\phi
(x,v,t_{0})=\varphi (x,v).
\end{equation}%
Notice that we can solve this equation in the whole space via the
fundamental solution and moreover, the solution $\phi $ is smooth due to the
hypoellipticity of the Fokker-Planck operator: $\phi \in C^{\infty }(\mathbb{%
R}^{2}\times (-\infty ,t_{0}))$. For a detailed discussion on
hypoellipticity, see \cite{H, HN04, V}. Choose $0<\rho <\rho _{1}<\rho _{2}$
so that $B_{\rho _{2}}(x_{0},v_{0})$ is contained in the interior and
consider a smooth cutoff function $\zeta \in C^{\infty }(\mathbb{R}^{2})$
such that
\begin{equation*}
\zeta =%
\begin{cases}
1 & \text{on }B_{\rho _{1}}, \\
0 & \text{on }\mathbb{R}^{2}\setminus B_{\rho _{2}}.%
\end{cases}%
\end{equation*}%
Letting ${\psi }=\phi \zeta $, we see that ${\psi }$ satisfies the following
\begin{equation}
-\mathcal{M}^{\ast }({\psi })={\psi }_{t}+v{\psi }_{x}+{\psi }_{vv}=%
\underbrace{v\zeta _{x}\phi +2\zeta _{v}\phi _{v}+\zeta _{vv}\phi }_{=:R}.
\end{equation}%
Since $\zeta _{x},\zeta _{v},\zeta _{vv}$ are smooth, supported in $%
\overline{B}_{\rho _{2}}\setminus B_{\rho _{1}},$ and $\phi $ is also smooth
in there, we deduce that $R$ is smooth and supp$\,R\subset \overline{B}%
_{\rho _{2}}\setminus B_{\rho _{1}}$. Note that $\psi (0,v,t)=0$ for $v<0$.
We will use this ${\psi }=\phi \zeta $ as a test function in \eqref{wf} to
get
\begin{equation}
\int_{B_{\rho }}\varphi f(t_{0})=\int_{0}^{t_{0}}\int_{\overline{B}_{\rho
_{2}}\setminus B_{\rho _{1}}}dxdv\,Rf+\int_{\Omega }dxdv\,\psi (0)f_{0}.
\label{(3.15)}
\end{equation}%
Notice that the right-hand side is bounded by $\Vert f_{0}\Vert _{L^{1}}$.
Thus we deduce that
\begin{equation*}
\int_{B_{\rho }}\varphi f(t_{0})\leq C.
\end{equation*}%
Since $\varphi \in C(B_{\rho })$ can be taken arbitrarily, by density
argument, this can be extended for all functions in $L^{2}(B_{\rho })$. Thus
by duality, $f\in L^{2}(B_{\rho })$. Now if we take a test function: $\psi =%
\frac{\partial ^{k^{\prime }+m^{\prime }}\phi }{\partial x^{k^{\prime
}}\partial v^{m^{\prime }}}\zeta $ for $k^{\prime }\leq k$, $m^{\prime }\leq
m$ in \eqref{wf}, by the duality characterization of $H_{x,v}^{k,m}(B_{\rho
})$, we conclude that $f\in H_{x,v}^{k,m}(B_{\rho })$. This completes the
proof.
\end{proof}

%---------------------------------------------------------%---------------------------------------------------------

\subsubsection{Boundary hypoellipticity}

%---------------------------------------------------------%---------------------------------------------------------

The goal of this subsection is to prove the boundary hypoellipticity away
from the singular set $\{(0,0),\left( 1,0\right) \}$. Before we prove it, we
derive a lemma which will be used to obtain the boundary hypoellipticity.

Let $v_{0}<0$. Choose $\delta >0$ such that $2\delta <|v_{0}|$. We consider
the following backward Fokker-Planck problem:
\begin{equation}
\begin{cases}
\mathcal{M}^{\ast }(\phi )=0\text{ for }t<t_{0} \\
\phi (x,v,t_{0})=\varphi (x,v)\text{ where }\text{supp}\,\varphi \subset
B_{\delta }(0,v_{0})\text{ and }\varphi \in C(B_{\delta }) \\
\phi (0,v,t)=0\text{ for }|v-v_{0}|<2\delta .%
\end{cases}
\label{adjP}
\end{equation}%
We show the existence of a solution $\phi $ to \eqref{adjP}.

\begin{lemma}
\label{lambda1}There exists $\lambda (w,t)\in L^{1}\left( \mathbb{R}\times %
\left[ 0,t_{0}\right] \right) $ such that its support in $w$ lies in $%
|w-v_{0}|\leq 2\delta $ and it is smooth in $|w-v_{0}|<2\delta $ and that $%
\phi (x,v,t)$ defined by the following expression
\begin{equation}
\begin{split}
\phi (x,v,t)=& \int_{\mathbb{R}^{2}}d\xi dw\varphi (\xi ,w)G(x-\xi
,-v,-w,t_{0}-t) \\
& \quad +\int_{t_{0}}^{t}ds\int_{v_{0}-2\delta }^{v_{0}+2\delta }dw\lambda
(w,s)G(x,-v,-w,s-t)
\end{split}
\label{phi}
\end{equation}%
solves the problem \eqref{adjP}. Here $G$ is the fundamental solution to the
forward Fokker-Planck equation given in \eqref{G}.
\end{lemma}

\begin{proof}
Denote the right-hand side of \eqref{phi} as $\Phi \lbrack \lambda ](x,v,t)$%
. Then it is clear that $\mathcal{M}(\Phi \lbrack \lambda ])=0$ in $\Omega $
and $\Phi \lbrack \lambda ](t=t_{0})=\varphi $. We want to show that there
exists a $\lambda =\lambda (w,t)$ such that $\lambda =0$ for $%
|v-v_{0}|>2\delta $ (by defining $\lambda =0$ for $|v-v_{0}|>2\delta $ for
instance) and $\Phi \lbrack \lambda ]=0$ for $x=0$, $|v-v_{0}|<2\delta $.
Let us first see what equation $\lambda $ would obey to satisfy the desired
properties. $\Phi \lbrack \lambda ]=0$ for $|v-v_{0}|<2\delta $ at $x=0$ is
equivalent to
\begin{equation}
0=\overline{\phi }(0,v,t)+\lim_{x\rightarrow
0^{+}}\int_{t_{0}}^{t}ds\int_{v_{0}-2\delta }^{v_{0}+2\delta }dw\lambda
(w,s)G(x,-v,-w,s-t)\text{ for }|v-v_{0}|<2\delta ,  \label{cond}
\end{equation}%
where $\overline{\phi }$ is the first term of $\Phi \lbrack \lambda ]$: the
homogeneous solution in the whole space. Now by Lemma \ref{limit}, %
\eqref{cond} can be written as follows.
\begin{equation}
0=\overline{\phi }(0,v,t)-\frac{\lambda (v,t)}{v}+\int_{t_{0}}^{t}ds%
\int_{v_{0}-2\delta }^{v_{0}+2\delta }dw\lambda (w,s)G(0,-v,-w,s-t).
\end{equation}%
Here instead of $1/v$, $-1/v$ comes out in front of $\lambda $ because $G$
is evaluated at $-v$ and $-w$.

Note that $|G(0,v,w,s-t)|$ is bounded by $e^{-\frac{A}{s-t}}$ for $v,w\in
(v_{0}-2\delta ,v_{0}+2\delta )$. Thus $\lambda $ satisfies the following
integral equation: for $|v-v_{0}|<2\delta $ and $t<t_{0},$
\begin{equation}
\lambda (v,t)=q(v,t)+\int_{t_{0}}^{t}ds\int_{|w-v_{0}|<2\delta }dw\lambda
(w,t)K(v,w,s-t),  \label{lambda}
\end{equation}%
where $q$ is a given smooth function and the given smooth kernel $K$ has the
following bound:
\begin{equation*}
|K(v,w,s-t)|\leq Ce^{-\frac{A}{|t-s|}}\text{ for }|v-v_{0}|<2\delta \text{
and }|w-v_{0}|<2\delta
\end{equation*}%
for some positive $A>0$. Hence by a fixed point argument, we can find a $%
\lambda $ satisfying \eqref{lambda}. This completes the proof.
\end{proof}

\label{br}

\begin{proof}[Proof of Theorem 1.3 (i)]
Proposition \ref{Prop} proves the hypoellipticity away from the boundary. It
then suffices to establish the regularity near the boundary: $%
(0,v_{0},t_{0}) $ where $v_{0}\neq 0$ and $t_{0}>0$ since the other boundary
$(1,v_{0},t_{0}) $ where $v_{0}\neq 0$ and $t_{0}>0$ can be treated
similarly. We divide into two cases: when $v_{0}<0$ and $v_{0}>0$.

We first treat the case when $x_{0}=0,v_{0}<0.$ To do so, we choose a smooth
cutoff function $\zeta \in C^{\infty }(\mathbb{R}^{2})$ such that
\begin{equation*}
\zeta =%
\begin{cases}
1 & \text{on }B_{\delta }(0,v_{0}) \\
0 & \text{on }\mathbb{R}^{2}\setminus B_{2\delta }(0,v_{0}).%
\end{cases}%
\end{equation*}%
Following the same argument as for the interior regularity, we pick a test
function $\psi $ as ${\psi }=\phi \zeta $ where $\phi $ is the solution to %
\eqref{adjP} in Lemma \ref{lambda1}. First we see that ${\psi }$ satisfies
the following
\begin{equation}
-\mathcal{M}^{\ast }({\psi })={\psi }_{t}+v{\psi }_{x}+{\psi }_{vv}={v\zeta
_{x}\phi +2\zeta _{v}\phi _{v}+\zeta _{vv}\phi =:R}\,.
\end{equation}%
Since $\zeta _{x},\zeta _{v},\zeta _{vv}$ are smooth and supported in $%
\overline{B}_{2\delta }\setminus B_{\delta }$ and $\phi $ is also smooth in
there, we deduce that $R$ is smooth and supp$\,R\subset \overline{B}%
_{2\delta }\setminus B_{\delta }$. Moreover, since $\phi (0,v,t)=0$ for $%
|v-v_{0}|<2\delta $ and $\zeta (0,v)=0$ for $|v-v_{0}|\geq 2\delta >0$, $%
\psi (0,v,t)=0$ for all $v<0$. Thus we can use this ${\psi }=\phi \zeta $ as
a test function in \eqref{wf} by restricting to
\begin{equation}
\int_{B_{\delta }\,\cap \,\Omega }\varphi f(t_{0})=\int_{0}^{t_{0}}\int_{%
\overline{B}_{2\delta }\setminus B_{\delta }\,\cap \Omega
}dxdv\,Rf+\int_{\Omega }dxdv\,\psi (0)f_{0}.  \label{(3.22)}
\end{equation}%
Notice that the right-hand side is bounded by $\Vert f_{0}\Vert _{L^{1}}$.
Thus we deduce that
\begin{equation*}
\int_{B_{\delta }\,\cap \,\Omega }\varphi f(t_{0})\leq C.
\end{equation*}%
Since $\varphi \in C(B_{\delta })$ can be taken arbitrarily, by density
argument, this can be extended for all functions in $L^{2}(B_{\delta }\,\cap
\,\Omega )$. Thus by duality, $f\in L^{2}(B_{\delta }\,\cap \,\Omega )$. Now
if we take a test function: $\psi =\frac{\partial ^{k^{\prime }+m^{\prime
}}\phi }{\partial x^{k^{\prime }}\partial v^{m^{\prime }}}\zeta $ for $%
k^{\prime }\leq k$, $m^{\prime }\leq m$ in \eqref{wf}, by the duality
characterization of $H_{x,v}^{k,m}(B_{\delta }\,\cap \,\Omega )$, we
conclude that $f\in H_{x,v}^{k,m}(B_{\delta }\,\cap \Omega \,)$.

It now remains to treat when $x_{0}=0$ and $v_{0}>0$. This can be treated in
the same way as in the interior case: since there is no restriction on the
boundary values of the test functions $\psi $ for $v>0$, by a suitable
choice of a cutoff function, we can easily find an appropriate test function
localized near $(0,v_{0})$ for $v_{0}>0$. We omit the details.\
\end{proof}

\begin{remark}
\label{Rem}Notice that as in the proof of the hypoellipticity (see %
\eqref{(3.15)} and \eqref{(3.22)}) the supremum norm of $f$ away from the
singular set is bounded by $\Vert f_{0}\Vert _{L^{1}}$.
%\texttt{...Do we want to be more precise?...}
\end{remark}

%---------------------------------------------------------%---------------------------------------------------------

\subsubsection{Optimal estimates for the derivatives near the singular set}

We derive in this subsection the following estimates near the singular set
for the derivatives of solutions to (\ref{VFP})-(\ref{BC1}) by a scaling
argument and the hypoellipticity.

\begin{lemma}
\label{hypoellipticity-der}Let $f \left( x,v,t\right) \in C^{\infty }\left(
U_{T}\right) $ be a solution of (\ref{VFP})-(\ref{BC1}). Then it satisfies%
\begin{equation}
\left( \left\vert v\right\vert ^{3}+\left\vert x-x_{0}\right\vert \right)
\left\Vert f _{x}\right\Vert _{L^{\infty }}+\left( \left\vert v\right\vert
^{3}+\left\vert x-x_{0}\right\vert \right) ^{2/3}\left\Vert
f_{vv}\right\Vert _{L^{\infty }}+\left( \left\vert v\right\vert
^{3}+\left\vert x-x_{0}\right\vert \right) ^{1/3}\left\Vert f
_{v}\right\Vert _{L^{\infty }}\leq C,  \label{derOfAdj}
\end{equation}%
where $x_{0}=0$ or $1$ and $C$ depends only $\left\Vert f _{0}\right\Vert
_{L^{1}(\Omega )}$ and $\left\Vert f _{0}\right\Vert _{L^{\infty }(\Omega )}$%
.
\end{lemma}

\begin{proof}
First we use the hypoellipticity to get, for $1/2\leq \left( \left\vert
V\right\vert ^{3}+\left\vert X\right\vert \right) ^{1/3}\leq 1$ and for $%
1/2\leq \left( \left\vert V\right\vert ^{3}+\left\vert X-1\right\vert
\right) ^{1/3}\leq 1,$
\begin{equation*}
\left\Vert f_{X}\right\Vert _{L^{\infty }}+\left\Vert f_{VV}\right\Vert
_{L^{\infty }}+\left\Vert f_{V}\right\Vert _{L^{\infty }}\leq C,
\end{equation*}%
where $C$ depends only on $\left\Vert f_{0}\right\Vert _{L^{1}(\Omega )}$
and $\left\Vert f_{0}\right\Vert _{L^{\infty }(\Omega )}$. We now scale
\thinspace $X,V,\tau $ as follows.%
\begin{equation*}
v=RV,~x-x_{0}=R^{3}\left( X-x_{0}\right) ,~t=R^{2}\tau .
\end{equation*}%
Then we have, for $R/2\leq \left( \left\vert v\right\vert ^{3}+\left\vert
x\right\vert \right) ^{1/3}\leq R$ and for $R/2\leq \left( \left\vert
v\right\vert ^{3}+\left\vert x-1\right\vert \right) ^{1/3}\leq R,$
\begin{equation*}
R^{3}\left\Vert f_{x}\right\Vert _{L^{\infty }}+R^{2}\left\Vert
f_{vv}\right\Vert _{L^{\infty }}+R\left\Vert f_{v}\right\Vert _{L^{\infty
}}\leq C,
\end{equation*}%
where $C$ depends only on $\left\Vert f_{0}\right\Vert _{L^{1}(\Omega )}$
and $\left\Vert f_{0}\right\Vert _{L^{\infty }(\Omega )}$. This implies (\ref%
{derOfAdj}) and completes the proof.
\end{proof}

\subsection{Power law estimates for the solution}

We now derive estimates for the solutions of the Fokker-Planck equation near
the singular point $(x,v)\in \{(0,0),(1,0)\}$. The asymptotic behavior of
these solutions has been found in some of the explicit solutions obtained in
the physical literature (cf. \cite{Bu}, \cite{MW1}). We summarize the main
properties of the relevant solutions and prove them in detail here.

We will use repeatedly in this subsection the usual asymptotic notation.
More precisely, we will say that $f\left( x\right) \sim g\left( x\right) $
as $x\rightarrow x_{0}$ if $\lim_{x\rightarrow x_{0}}\frac{f\left( x\right)
}{g\left( x\right) }=1,$ for $x_{0}\in \left[ -\infty ,\infty \right] .$ On
the other hand, we will use the symbol $\simeq $ in heuristic, nonrigorous
arguments to indicate that two functions have a similar behavior in some
region.

\subsubsection{Construction of Super-solutions}

Our goal is to construct super-solutions which will allow us to control the
singular set $\{(x,v):(0,0),(1,0)\}$ based on the study of the self-similar
behavior of solutions to the Fokker-Planck equation. We begin by recalling
the steady Fokker-Planck equation.
\begin{equation}
vF_{x}=F_{vv}.  \label{steady}
\end{equation}

\begin{lemma}
\label{30} There exist positive steady solutions $f_{k}^{\ast },\ $with $%
k=0,1,$ to \eqref{steady}, which blow up at the singular set, namely

\begin{enumerate}
\item $f_{k}^{\ast }(x,v)>0,$

\item $\liminf_{r\rightarrow 0,(x,v)\in \partial B_{r}(k,0)\cap \Omega
}f_{k}^{\ast }(x,v)=\infty $.
\end{enumerate}
\end{lemma}

\begin{proof}
We seek a solution $F$ to \eqref{steady} of the following form
\begin{equation}
F(x,v)=x^{\alpha }\Phi (-\frac{v^{3}}{9x})  \label{Phi}
\end{equation}%
for $\alpha <0$ to be determined. Plugging in the ansatz \eqref{Phi} into %
\eqref{steady} and letting $z=-\frac{v^{3}}{9x}$, we deduce that $\Phi =\Phi
(z)$ satisfies the following ODE
\begin{equation}
z\Phi _{zz}+(\frac{2}{3}-z)\Phi _{z}+\alpha \Phi =0.  \label{kummer}
\end{equation}%
It is well-known that the solutions to \eqref{kummer} are given by Kummer
functions $M(-\alpha ,\frac{2}{3},z)$ and $U(-\alpha ,\frac{2}{3},z)$ (see
\cite{Ab}). We are interested in positive solutions for $z\in \mathbb{R}$
which grow at infinity. We recall the integral representation formula for $M$
(see 13.2.1 in \cite{Ab}):
\begin{equation}
\frac{\Gamma (\frac{2}{3}+\alpha )\Gamma (-\alpha )}{\Gamma (\frac{2}{3})}%
M(-\alpha ,\frac{2}{3},z)=\int_{0}^{1}e^{zt}t^{-\alpha -1}(1-t)^{\alpha -%
\frac{1}{3}}dt,  \label{13.2.1}
\end{equation}%
which is valid as long as $\alpha <0$ ($b=\frac{2}{3}$ is already positive).
For sufficiently small negative $\alpha \sim 0$, we see that
\begin{equation*}
\frac{\Gamma (\frac{2}{3}+\alpha )\Gamma (-\alpha )}{\Gamma (\frac{2}{3})}>0
\end{equation*}%
and that for any real value $z\in \mathbb{R}$, the integral on the
right-hand side of \eqref{13.2.1} is positive. Hence we deduce that for such
negatively small $\alpha $ and for $z\in \mathbb{R}$, $M(-\alpha ,\frac{2}{3}%
,z)>0$, which implies that the corresponding $F$, denoted by $f_{0}^{\ast }$%
,
\begin{equation*}
f_{0}^{\ast }(x,v):=x^{\alpha }M(-\alpha ,\frac{2}{3},-\frac{v^{3}}{9x})
\end{equation*}%
in \eqref{Phi} is also positive. It now remains to show that $f_{0}^{\ast }$
blows up at the origin, the item (2) in the above. This can verified by
noting the asymptotic behavior of $M(a,b,z)$ (see 13.1.4 and 13.1.5 in \cite%
{Ab}):
\begin{equation}
\begin{split}
M(a,b,z)& \sim \frac{\Gamma (b)}{\Gamma (a)}e^{z}z^{a-b}\quad \text{ for }%
z\rightarrow \infty , \\
M(a,b,z)& \sim \frac{\Gamma (b)}{\Gamma (b-a)}(-z)^{-a}\;\text{ for }%
z\rightarrow -\infty .
\end{split}
\label{asympt}
\end{equation}%
Therefore, applying \eqref{asympt} to our case: $z=-\frac{v^{3}}{9x}$, $%
a=-\alpha $, $b=\frac{2}{3}$, we obtain the following asymptotic behavior of
$f_{0}^{\ast }$:
\begin{equation*}
\begin{split}
f_{0}^{\ast }(x,v)& \sim \frac{\Gamma (\frac{2}{3})}{\Gamma (-\alpha )}%
x^{\alpha }e^{-\frac{v^{3}}{9x}}(-\frac{v^{3}}{9x})^{-\alpha -\frac{2}{3}}\;%
\text{ for }v<0, \\
f_{0}^{\ast }(x,v)& \sim \frac{\Gamma (\frac{2}{3})}{\Gamma (\frac{2}{3}%
+\alpha )}x^{\alpha }(\frac{v^{3}}{9x})^{\alpha }\;\text{ for }v>0,
\end{split}%
\end{equation*}%
as $x\rightarrow 0.$ Thus $f_{0}^{\ast }(x,v)\rightarrow \infty $ for $v<0$
and $x\rightarrow 0$, and hence we conclude that
\begin{equation*}
\liminf_{r\rightarrow 0,(x,v)\in \partial B_{r}(0,0)\cap \{0\leq x\leq
1\}}f_{0}^{\ast }(x,v)=\infty .
\end{equation*}%
The construction of $f_{1}^{\ast }$, whose self-similarity is centered at $%
(1,0)$, can be done in the same way and we omit the details.
\end{proof}

The next goal is to construct super-solutions that control the singular
behavior near the singular set via self-similarity. The first step is to
find the regular self-similar solution to \eqref{steady}. To this end, we
define also $\Lambda $ by means of
\begin{equation*}
\Lambda \left( \zeta \right) =\Phi \left( -\zeta ^{3}\right)
\end{equation*}%
and we seek a solution to \eqref{steady} of the form
\begin{equation*}
F(x,v)=x^{\alpha }\Lambda \left( \frac{v}{(9x)^{\frac{1}{3}}}\right)
\end{equation*}%
for $\alpha >0$ this time. Then it is easy to check that $\Lambda $
satisfies the following ODE

\begin{equation}
\Lambda ^{\prime \prime }\left( \zeta \right) +3\zeta ^{2}\Lambda ^{\prime
}\left( \zeta \right) -9\alpha \zeta \Lambda \left( \zeta \right)=0 .
\label{LambdEq}
\end{equation}

%\texttt{... All this text must be now clarified. ...}\

%\texttt{... I am changing much the writing of the following result and its
%Proof. ...}\

%\bigskip

We are interested in the construction of solutions of (\ref{LambdEq}) which
are polynomially bounded. We have the following result.

\begin{claim}
\label{32} For any $0<\alpha <\frac{1}{6},$ there exists a solution $\Lambda
(\zeta )$ of (\ref{LambdEq}) with the form:%
\begin{equation}
\Lambda (\zeta )=U(-\alpha ,\frac{2}{3},-\zeta ^{3})\ \ ,\ \ \zeta \in
\mathbb{R,}  \label{LU}
\end{equation}%
where we denote as $U(a,b,z)$ the Tricomi confluent hypergeometric function.
%\texttt{... I have changed the notation of }$U$\texttt{\ to }$U(-\alpha ,%
%\frac{2}{3},-\zeta ^{3}).$ \texttt{This is the standard notation. Revise and
%change everywhere. WARNING: We must change also the Kummer function to }$%
%M(-\alpha ,\frac{2}{3},z)$\texttt{\ ...}
The function $\Lambda (\zeta )$ has the following properties.

\begin{enumerate}
\item $\Lambda (\zeta )>0$ for any $\zeta \in \mathbb{R}$.

\item there exists a positive constant $K_{+}>0$ such that
%\texttt{... I change
%this. The coefficient seems to be one for }$\zeta \rightarrow -\infty .$%
%\texttt{\ ...}\
\begin{equation}
\Lambda (\zeta )\sim
\begin{cases}
K_{+}|\zeta |^{3\alpha },\quad \zeta \rightarrow \infty , \\
|\zeta |^{3\alpha },\quad \zeta \rightarrow -\infty .%
\end{cases}
\label{asymptotics}
\end{equation}

\item The function $\Lambda (\zeta )$, up to a multiplicative constant, is
the only solution of (\ref{LambdEq}) which is polynomially bounded for large
$|\zeta |.$%\texttt{... See how this text looks now. ...}\
\end{enumerate}
\end{claim}

\begin{proof}[Proof of Claim \protect\ref{32}]
Due to the relation between (\ref{LambdEq}) and (\ref{kummer}), we need to
study the solutions of this last equation (\ref{LambdEq}), which are
algebraically bounded. The only solutions of the equation \eqref{kummer}
which do not grow exponentially for large $z>0$ are proportional to
\begin{equation}
\Phi (z)=U(-\alpha ,\frac{2}{3},z).  \label{U}
\end{equation}

In order to study the properties of $\Phi (z)$ for negative values of $z$ we
use that (cf. \cite{Ab}, 13.1.3):%
\begin{equation}
U(a,b,z)=\frac{\pi }{\sin (\pi b)}\left( \frac{M(a,b,z)}{\Gamma
(1+a-b)\Gamma (b)}-z^{1-b}\frac{M(1+a-b,2-b,z)}{\Gamma (a)\Gamma (2-b)}%
\right) ,\quad b\notin \mathbb{Z}.  \label{UM}
\end{equation}

The function $M(a,b,z)$ is analytic for all $z\in \mathbb{C}$. Notice that,
combining (\ref{UM}) and (\ref{LU}) we obtain the following representation
formula for $\Lambda (\zeta )$ :%
\begin{equation}
\Lambda (\zeta )=\frac{\pi }{\sin (\frac{2}{3}\pi )}\left( \frac{M(-\alpha ,%
\frac{2}{3},-\zeta ^{3})}{\Gamma (\frac{1}{3}-\alpha )\Gamma (\frac{2}{3})}%
+\zeta \frac{M(\frac{1}{3}-\alpha ,\frac{4}{3},-\zeta ^{3})}{\Gamma (-\alpha
)\Gamma (\frac{4}{3})}\right) ,\ \ \zeta \in \mathbb{R}.  \label{Lamb}
\end{equation}

Formula (\ref{Lamb}) provides a representation formula for $\Lambda (\zeta )$
in terms of the analytic functions $M(-\alpha ,\frac{2}{3},-\zeta ^{3}),\ M(%
\frac{1}{3}-\alpha ,\frac{4}{3},-\zeta ^{3}).$ This formula shows that $%
\Lambda (\zeta )$ is analytic in $\zeta \in \mathbb{C}$.

We can compute the asymptotics of $\Lambda (\zeta )$ as $\zeta \rightarrow
-\infty $ by using (\ref{LU}) and 13.5.2 in \cite{Ab}.
%\texttt{... Warning: I
%am correcting this formula number. Check. I think that this is the correct
%number for the asymptotics of }$U.$\texttt{\ }$U(a,b,z)\sim \left( z\right)
%^{-a}$\texttt{\ as }$z\rightarrow \infty .$\texttt{\ ...}\
Then we deduce that
\begin{equation}
\Lambda (\zeta )\sim \left\vert \zeta \right\vert ^{3\alpha }\ \ \text{as\ \
}\zeta \rightarrow -\infty .  \label{Aplus}
\end{equation}

On the other hand, using 13.5.1 in \cite{Ab} %\texttt{... Check this. ...}
as well as (\ref{Lamb}) we obtain the asymptotics: %\texttt{... Remove
%intermediate steps and keep just the final formula. (See if it is worth to
%keep some of this intermediate text. ...}\
\begin{equation*}
M\left( a,b,z\right) \sim \frac{\Gamma \left( b\right) e^{\pm i\pi a}}{%
\Gamma \left( b-a\right) }\left( z\right) ^{-a},\ \ z\rightarrow -\infty ,
\end{equation*}%
where the sign $+$ is used if $-\frac{\pi }{2}<\arg \left( z\right) <\frac{%
3\pi }{2}$ and the sign $-$ is used if $-\frac{3\pi }{2}<\arg \left(
z\right) <-\frac{\pi }{2}.$ In the choice of the branch of the function $%
\left( \cdot \right) ^{\alpha }$ we must choose the branch of the function $%
M $ which is analytic. We can obtain easily the asymptotics of the function $%
M$ which is analytic. This means that, choosing $z=re^{i\theta }$ with $%
\theta \in \left( -\frac{\pi }{2},\frac{\pi }{2}\right) $ we obtain the
asymptotics $M\left( a,b,re^{\pm i\pi }\right) \sim \frac{\Gamma \left(
b\right) }{\Gamma \left( b-a\right) }\left( r\right) ^{-a},$ because the
exponentials cancel out. Then, we must use the formulae:
\begin{eqnarray*}
M\left( a,b,-r\right) &\sim &\frac{\Gamma \left( b\right) }{\Gamma \left(
b-a\right) }\left( r\right) ^{-a},\ r\rightarrow \infty , \\
\zeta M(\frac{1}{3}-\alpha ,\frac{4}{3},-\zeta ^{3}) &\sim &\frac{\Gamma
\left( \frac{4}{3}\right) }{\Gamma \left( 1+\alpha \right) }\left( \zeta
\right) ^{3\alpha },\ \zeta \rightarrow \infty .
\end{eqnarray*}%
Choose $\zeta >0.$ Then the phase factor cancels out. We then have, using (%
\ref{Lamb}),
\begin{eqnarray*}
\Lambda (\zeta ) &=&\frac{\pi }{\sin (\frac{2}{3}\pi )}\left( \frac{%
M(-\alpha ,\frac{2}{3},-\zeta ^{3})}{\Gamma (\frac{1}{3}-\alpha )\Gamma (%
\frac{2}{3})}+\zeta \frac{M(\frac{1}{3}-\alpha ,\frac{4}{3},-\zeta ^{3})}{%
\Gamma (-\alpha )\Gamma (\frac{4}{3})}\right) \\
&\sim &\frac{\pi }{\sin (\frac{2}{3}\pi )}\left[ \frac{1}{\Gamma (\frac{1}{3}%
-\alpha )}\frac{1}{\Gamma \left( \frac{2}{3}+\alpha \right) }+\frac{1}{%
\Gamma (-\alpha )\Gamma \left( 1+\alpha \right) }\right] \left\vert \zeta
\right\vert ^{3\alpha },\ \zeta \rightarrow \infty .
\end{eqnarray*}%
Using the following formulae,
\begin{equation*}
\Gamma \left( \frac{1}{3}-x\right) \Gamma \left( \frac{2}{3}+x\right) =\frac{%
\pi }{\sin \left( \pi \left( x+\frac{2}{3}\right) \right) }\quad \text{and}%
\quad \Gamma \left( -x\right) \Gamma \left( 1+x\right) =-\frac{\pi }{\sin
\left( \pi x\right) },
\end{equation*}%
we see that
\begin{equation*}
\Lambda (\zeta )\sim \frac{1}{\sin (\frac{2}{3}\pi )}\left[ \sin \left( \pi
\left( \alpha +\frac{2}{3}\right) \right) -\sin \left( \pi \alpha \right) %
\right] \left\vert \zeta \right\vert ^{3\alpha },\ \zeta \rightarrow \infty .
\end{equation*}%
Elementary trigonometric formulas show that
\begin{equation*}
\Lambda (\zeta )\sim \frac{2\sin \left( \frac{\pi }{3}\right) }{\sin (\frac{2%
}{3}\pi )}\cos \left( \pi \left( \alpha +\frac{1}{3}\right) \right)
\left\vert \zeta \right\vert ^{3\alpha },\ \ \zeta \rightarrow \infty .
\end{equation*}%
Whence

\begin{equation}
\Lambda (\zeta )\sim K_{+}\left\vert \zeta \right\vert ^{3\alpha }\ \ \text{%
as\ \ }\zeta \rightarrow \infty  \label{Aminus}
\end{equation}%
with $K_{+}=2\cos \left( \pi \left( \alpha +\frac{1}{3}\right) \right) .$
Notice that $K_{+}>0$ if $0<\alpha <\frac{1}{6}.$

It only remains to prove that $\Lambda (\zeta )>0$ for any $\zeta \in
\mathbb{R}$ and the considered range of values of $\alpha .$ To this end,
notice that if $\alpha \rightarrow 0$ we have $\Lambda (\zeta )\rightarrow
1>0$ uniformly in compact sets of $\zeta .$ %\texttt{... Precise
%values. ...}\
The functions $\Lambda (\zeta )\equiv \Lambda (\zeta ,\alpha )$
%\texttt{... Notation....}\
considered as functions of $\alpha ,$ change in a continuous manner. On the
other hand, the asymptotic behaviors (\ref{Aplus}), (\ref{Aminus}) imply
that the functions $\Lambda (\zeta ,\alpha )$ are positive for large values
of $\left\vert \zeta \right\vert .$ %\texttt{... Check coefficients.
%Perhaps precise more to have uniformicity of the expansions. ...}\
If $\Lambda (\cdot ,\alpha )$ has a zero at some $\zeta =\zeta _{0}\in
\mathbb{R}$ and $0<\alpha <\frac{1}{6}$, then there should exist, by
continuity, $0<\alpha _{\ast }<\frac{1}{6}$ and $\zeta _{\ast }\in \mathbb{R}
$ such that $\Lambda (\zeta _{\ast },\alpha _{\ast })=\Lambda _{\zeta
}(\zeta _{\ast },\alpha _{\ast })=0.$ The uniqueness theorem for ODEs then
implies that $\Lambda (\cdot ,\alpha _{\ast })=0,$ but this would contradict
the asymptotics (\ref{Aplus}), (\ref{Aminus}), whence the result follows.
\end{proof}

We now let %\texttt{... There is something strange with this formula. The
%sign used above is different. I suspect that the sign above is correct, but
%this must be checked. ...}
\begin{equation}
F_{0}(x,v):=x^{\alpha }\Lambda \left( \frac{v}{(9x)^{\frac{1}{3}}}\right) ,
\label{F0}
\end{equation}%
where $\Lambda $ is obtained in Claim \ref{32}. Then from \eqref{asymptotics}
we deduce that $F_{0}$ is a positive steady solution to the Fokker-Planck
equation and that when $x\rightarrow 0^{+}$, $F_{0}(x,v)\simeq x^{\alpha
}+|v|^{3\alpha }$. By the same argument, one can find $F_{1}(x,v)>0$, a
steady solution to the Fokker-Planck equation such that when $x\rightarrow
1^{-}$, $F_{1}(x,v)\simeq (1-x)^{\alpha }+|v|^{3\alpha }$. These $F_{0}$ and
$F_{1}$ will be used in the construction of a super-solution to \eqref{VFP},
which now follows.

Let $\Psi =\Psi (y,\xi )$ be a self-similar type solution to \eqref{VFP} of
the following form
\begin{equation}
f(x,v,t)=\Psi (\frac{x}{t^{3/2}},\frac{v}{t^{{1}/{2}}}).  \label{Psi}
\end{equation}%
If $\Psi $ is a solution to \eqref{VFP}, it should satisfy the following PDE
\begin{equation*}
-\frac{3}{2}y\Psi _{y}-\frac{1}{2}\xi \Psi _{\xi }+\xi \Psi _{y}=\Psi _{\xi
\xi }.
\end{equation*}%
We will not attempt to solve this partial differential equation since we
only need a super-solution of \eqref{VFP}, but try to find a super-solution $%
Z$ to the self-similar equation above, namely satisfying
\begin{equation}
Z_{\xi \xi }+\frac{1}{2}\xi Z_{\xi }+(\frac{3}{2}y-\xi )Z_{y}\leq 0.
\label{W}
\end{equation}

We first show that one can construct $Z_{0}$ satisfying \eqref{W} in a small
neighborhood of the singular set $\left( 0,0\right) $.

\begin{lemma}
\label{exist-W} There exists a sufficiently small $R_{0}(y,\xi )$ such that
(i) $|R_{0}|\ll F_{0}$ for $|\xi |^{3}+|y|\ll 1$ and that (ii) $Z_{0}(y,\xi
):=F_{0}(y,\xi )+R_{0}(y,\xi )$, where $F_{0}$ is given by \eqref{F0},
satisfies \eqref{W} for $|\xi |^{3}+|y|\ll 1$. Notice that $Z_{0}>0$.
\end{lemma}

\begin{proof}
Notice that $(F_{0})_{\xi \xi }-\xi (F_{0})_{y}=0,$ where $F_{0}(y,\xi
)=y^{\alpha }Q(-\frac{\xi ^{3}}{9y})$ with $Q$ obtained in Claim \ref{32}
(we use $Q$ instead of $\Lambda $ to distinguish the different argument) and
hence
\begin{equation*}
(F_{0})_{\xi \xi }+\frac{1}{2}\xi (F_{0})_{\xi }+(\frac{3}{2}y-\xi
)(F_{0})_{y}=\frac{1}{2}\xi (F_{0})_{\xi }+\frac{3}{2}y(F_{0})_{y}.
\end{equation*}%
But then
\begin{equation*}
\begin{split}
\frac{1}{2}\xi (F_{0})_{\xi }+\frac{3}{2}y(F_{0})_{y}& =\frac{1}{2}\xi \left[
y^{\alpha }(-\frac{3\xi ^{2}}{9y})Q^{\prime }\right] +\frac{3}{2}y\left[
y^{\alpha }(\frac{\xi ^{3}}{9y^{2}})Q^{\prime }+\alpha y^{\alpha -1}Q\right]
\\
& =\frac{3}{2}\alpha F_{0},\text{ since the first two terms cancel each
other out.}
\end{split}%
\end{equation*}%
For $R_{0}$, we have the following inequality to be solved:
\begin{equation*}
(R_{0})_{\xi \xi }+\frac{1}{2}\xi (R_{0})_{\xi }+(\frac{3}{2}y-\xi
)(R_{0})_{y}\leq -\frac{3}{2}\alpha F_{0}.
\end{equation*}%
With the ansatz $R_{0}=y^{\beta }\varphi (-\frac{\xi ^{3}}{9y})$, the
left-hand side reads
\begin{equation*}
\begin{split}
& (R_{0})_{\xi \xi }+\frac{1}{2}\xi (R_{0})_{\xi }+(\frac{3}{2}y-\xi
)(R_{0})_{y} \\
& =-\xi y^{\beta -1}\left\{ z\varphi _{zz}+(\frac{2}{3}-z)\varphi _{z}+\beta
\varphi \right\} +\frac{3}{2}\beta y^{\beta }\varphi
\end{split}%
\end{equation*}%
and hence the above inequality for $R_{0}$ reduces to
\begin{equation*}
-\xi y^{\beta -1}\left\{ z\varphi _{zz}+(\frac{2}{3}-z)\varphi _{z}+\beta
\varphi \right\} +\frac{3}{2}\beta y^{\beta }\varphi \leq -\frac{3}{2}\alpha
F_{0}.
\end{equation*}%
Choose $\beta =\frac{2}{3}+\alpha $, then it suffices to find $\varphi (z)$
satisfying
\begin{equation}
z^{1/3}\left\{ z\varphi _{zz}+(\frac{2}{3}-z)\varphi _{z}+(\frac{2}{3}%
+\alpha )\varphi \right\} +(1+\frac{3}{2}\alpha )y^{2/3}\varphi \leq -\frac{3%
}{2}\alpha Q.  \label{varphi}
\end{equation}%
To do so, we will solve the following ODE
\begin{equation}
z\varphi _{zz}+(\frac{2}{3}-z)\varphi _{z}+(\frac{2}{3}+\alpha )\varphi =-%
\frac{\gamma }{z^{1/3}}Q  \label{varphiE}
\end{equation}%
for some constant $\gamma >\frac{3}{2}\alpha $. Let $M(-(\frac{2}{3}+\alpha
),\frac{2}{3},z)$ and $U(-(\frac{2}{3}+\alpha ),\frac{2}{3},z)$ be two
independent solutions to the homogeneous part. Then, by variation of
constants, the solution of \eqref{varphiE} is given by
\begin{equation}
\begin{split}
\varphi (z)& =-M(-(\frac{2}{3}+\alpha ),\frac{2}{3},z)\int_{z}^{\infty }%
\frac{\gamma Q(\eta )U(-(\frac{2}{3}+\alpha ),\frac{2}{3},\eta )}{\eta
^{4/3}W(\eta )}d\eta \\
& -U(-(\frac{2}{3}+\alpha ),\frac{2}{3},z)\int_{0}^{z}\frac{\gamma Q(\eta
)M(-(\frac{2}{3}+\alpha ),\frac{2}{3},\eta )}{\eta ^{4/3}W(\eta )}d\eta ,
\end{split}
\label{varphiS}
\end{equation}%
where
\begin{equation*}
W(\eta )=\left\vert
\begin{array}{cc}
M(-(\frac{2}{3}+\alpha ),\frac{2}{3},\eta ) & U(-(\frac{2}{3}+\alpha ),\frac{%
2}{3},\eta ) \\
M_{\eta }(-(\frac{2}{3}+\alpha ),\frac{2}{3},\eta ) & U_{\eta }(-(\frac{2}{3}%
+\alpha ),\frac{2}{3},\eta )%
\end{array}%
\right\vert =-\frac{\eta ^{-\frac{2}{3}}e^{\eta }}{\Gamma (-\frac{2}{3}%
-\alpha )}.
\end{equation*}%
Here we have used the fact that $W(\eta )$ satisfies the ODE: %\texttt{...
%WARNING: The letter }$W$\texttt{\ has been used above with other meaning.
%Change it somewhere. Revise this new text. ...}\
\begin{equation*}
\frac{dW(\eta )}{d\eta }+\left( \frac{2}{3z}-1\right) W(\eta )=0
\end{equation*}%
with the asymtotics:%
\begin{equation*}
W(\eta )\sim -\frac{\eta ^{-\frac{2}{3}}e^{\eta }}{\Gamma (-\frac{2}{3}%
-\alpha )}\ \ \text{as\ \ }\eta \rightarrow 0,
\end{equation*}%
which can be computed using the asymptotics of the functions $M(-(\frac{2}{3}%
+\alpha ),\frac{2}{3},\eta ),\ U(-(\frac{2}{3}+\alpha ),\frac{2}{3},\eta )$
as $\eta \rightarrow 0$ (cf. \cite{Ab}). Then we have
\begin{equation}
|R_{0}|=\left\vert y^{\frac{2}{3}+\alpha }\varphi (-\frac{\xi ^{3}}{9y}%
)\right\vert \ll F_{0}=y^{\alpha }Q(-\frac{\xi ^{3}}{9y})\text{ for }|\xi
|^{3}+|y|\ll 1.  \label{R0F0}
\end{equation}%
This can be checked by comparing the asymptotic behavior of $\varphi $ with
that of $Q$. We first recall from \eqref{asymptotics}
\begin{equation*}
Q(z)=O(|z|^{\alpha }),\quad |z|\rightarrow \infty .
\end{equation*}%
Moreover, by \eqref{asympt} and 13.5.2 in \cite{Ab}, and from (\ref{varphiS}%
) we deduce that
\begin{equation*}
\varphi (z)=O(|z|^{\frac{2}{3}+\alpha }),\quad |z|\rightarrow \infty ,
\end{equation*}%
which implies \eqref{R0F0}. Finally, together with \eqref{R0F0}, we deduce
that $\varphi $ for $\gamma >\frac{3}{2}\alpha $, which is a solution to %
\eqref{varphiE}, satisfies the inequality \eqref{varphi}. This completes the
proof of the lemma.
\end{proof}

One can also construct a super-solution $Z_{1}$ of self-similar type near $%
(1,0)$ in the same way. We omit the details.

%---------------------------------------------------------%---------------------------------------------------------

\subsection{H\"{o}lder estimates for the solution near the singular set}

%We have established the smoothing effect (hypoellipticity) of our solution $% f $ away from the singular set.
In this section, we will prove that our solution $f$ is continuous \textit{%
up to} the singular set $\{(x,v)=(0,0),(1,0)\}$, in fact H\"{o}lder
continuous by means of maximum principles: we will apply comparison
principles to the solution $f$ with a suitable super-solution $\bar{f}$ that
controls the singular set.

%\label{conti}

\subsubsection{The adjoint problem}

We study in this subsection the adjoint problem $\mathcal{M}^{\ast }\varphi
=0$ in $U_{T}$ backward in time with the corresponding absorbing boundary:%
\begin{eqnarray}
\varphi _{t}+v\varphi _{x}+\varphi _{vv} &=&0,  \label{adj1} \\
\varphi \left( x,v,T\right) &=&\varphi _{0}\left( x,v\right) \geq 0,
\label{adj2} \\
\varphi \left( 0,v,t\right) &=&0,~v<0,~t>0,  \label{adj3} \\
\varphi \left( 1,v,t\right) &=&0,~v>0,~t>0.  \label{adj4}
\end{eqnarray}

Then we obtain the following results concerning the existence of solutions,
the hypoellipticity away from the singular set $\left\{ \left( 0,0\right)
,\left( 1,0\right) \right\} $, and the non-negativity of solutions,
similarly to the original Fokker-Planck problem. We will skip the proofs.

\begin{proposition}
\label{hyperellip} Let data $\varphi _{0}\in L^{\infty }(\Omega )\cap
L^{1}(\Omega )$ with $\varphi _{0}\geq 0$ be given at $t=T$. Then there
exists a solution $\varphi \left( x,v,t\right) \in C\left( \left[ 0,T\right]
;L^{\infty }(\Omega )\cap L^{1}(\Omega )\right) $ to (\ref{adj1})-(\ref{adj4}%
). Moreover, for each $t>0$, $\varphi \in H_{loc}^{k,m}(\bar{\Omega}%
\setminus \{(0,0),\left( 1,0\right) \}),$ where $H^{k,m}=H_{x,v}^{k,m}$.
\end{proposition}

\begin{proof}
It is analogous to Theorem 1.3 (i).
\end{proof}

\begin{lemma}
\label{nonnegAdj} Let $\varphi \left( x,v,t\right) \in C^{\infty }\left(
U_{T}\right) $ be a solution of the adjoint equation (\ref{adj1})-(\ref{adj4}%
) backward in time. Then it satisfies $\varphi \left( x,v,t\right) \geq 0$
for all $\left( x,v,t\right) \in U_{T}.$
\end{lemma}

\begin{proof}
It is analogous to Lemma \ref{minicomparison}.
\end{proof}

We define a super-solution to the operator $\mathcal{M}$ \ in (\ref{M}) as
follows.

\begin{definition}
\label{super_sol_weak}Let $\psi $ $\in C\left( \bar{U}_{T}\smallsetminus
\left\{ \left( 0,0\right) ,\left( 1,0\right) \right\} \times \left[ 0,T%
\right] \right) .$ Then we say $\psi $ is a super-solution of (\ref{adj1}),
that is, $\mathcal{M}\psi \geq 0$ if for every $t\in \left[ 0,T\right] $ and
any test function $\varphi \left( x,v,s\right) \in C_{x,v,t}^{1,2,1}\left(
U_{t}\right) \cap C\left( \bar{U}_{t}\right) $ with $\varphi \geq 0$ such
that supp$\left( \varphi \left( \cdot ,\cdot ,s\right) \right) \subset \left[
0,1\right] \times \left[ -R,R\right] $ for some $R>0$ and $\varphi |_{\gamma
_{t}^{+}}=0$, it satisfies
\begin{align*}
& -\int_{U_{t}}\psi \left( x,v,s\right) \left[ \varphi _{t}\left(
x,v,s\right) +\partial _{x}\left( v\varphi \left( x,v,s\right) \right)
+\varphi _{vv}\left( x,v,s\right) \right] dxdvds \\
& +\int_{\Omega }\psi \left( x,v,t\right) \varphi \left( x,v,t\right)
dxdv-\int_{\Omega }\psi \left( x,v,0\right) \varphi \left( x,v,0\right) dxdv
\\
& +\int_{0}^{t}\int_{-\infty }^{0}v\psi \left( 1,v,s\right) \varphi \left(
1,v,s\right) dvds-\int_{0}^{t}\int_{0}^{\infty }v\psi \left( 0,v,s\right)
\varphi \left( 0,v,s\right) dvds \\
& \geq 0.
\end{align*}
\end{definition}

\begin{remark}
In the definition \ref{super_sol_weak}, we can extend the definition to a
more general region. For instance, the interval could be smaller than $\left[
0,1\right] ,$ which will be used in Lemma \ref{max-sup}.
\end{remark}

\

\subsubsection{Maximum principle}

Let
\begin{equation}
\hat{f}_{0}(x,v,t)=\min \{KZ_{0}(\frac{x}{t^{3/2}},\frac{v}{t^{1/2}}),\;1\},
\label{fhat}
\end{equation}%
where $K>1$ and $Z_{0}$ is obtained in Lemma \ref{exist-W}. Then since $%
Z_{0} $ satisfies \eqref{W}, $\hat{f}_{0}$ is a super-solution of the
Fokker-Planck equation. We further define $\bar{f}_{\varepsilon }$ by
\begin{equation}
\bar{f}_{\varepsilon }=C\hat{f}_{0}+\ep f_{1}^{\ast }+\ep f_{2}^{\ast },
\label{super_sol}
\end{equation}%
where $C=\Vert f_{0}\Vert _{\infty },$ $f_{1}^{\ast }$ is a singular
solution near the singular set $\left( 0,0\right) $ and $f_{2}^{\ast }$ is a
singular solution near the singular set $\left( 1,0\right) $ constructed in
Lemma \ref{30}.

The following maximum principle plays a key role.

\begin{lemma}
\label{max-sup} Let $f\left( x,v,t\right) $ be a solution to (\ref{VFP})-(%
\ref{BC1}). \label{comparison}If $f(x,v,0)\leq C\hat{f}_{0}(x,v,0)$, then $%
f(x,v,t)\leq C\hat{f}_{0}(x,v,t)$ for all $t>0$ and $\text{ for all }%
\;(x,v)\in \bar{\Omega}\setminus \{(0,0),(1,0)\}$. Here $C=\Vert f_{0}\Vert
_{\infty }$ and $\hat{f}_{0}$ is given in (\ref{fhat}).
\end{lemma}

\begin{proof}
We first introduce a cut-off function $\zeta \left( x,v\right) $ near the
singular set $\left\{ \left( 0,0\right) ,\left( 1,0\right) \right\} $ in the
phase plane $\bar{\Omega}=\left[ 0,1\right] \times \left( -\infty ,\infty
\right) ~$such that for any $\rho >0$ small,
\begin{equation}
\zeta _{\rho }\left( x,v\right) =\left\{
\begin{array}{c}
0,~~~\left( \left\vert v\right\vert ^{3}+\left\vert x\right\vert \right)
^{1/3}<\rho ,~\left( \left\vert v\right\vert ^{3}+\left\vert x-1\right\vert
\right) ^{1/3}<\rho \\
\in \left[ 0,1\right] ,~~~\ \ \ \ \ \rho \leq \left( \left\vert v\right\vert
^{3}+\left\vert x\right\vert \right) ^{1/3}\leq 2\rho ,~\rho \leq \left(
\left\vert v\right\vert ^{3}+\left\vert x-1\right\vert \right) ^{1/3}\leq
2\rho \\
1,\text{ \ \ \ \ \ \ \ \ \ \ \ \ \ \ \ \ \ \ \ \ \ \ \ \ \ \ \ \ \ \ \ \ \ \
\ \ \ \ \ \ \ \ \ \ \ \ otherwise}~%
\end{array}%
\right.  \label{cutoff}
\end{equation}

Let $\psi \left( x,v,t\right) =$ $\bar{f}_{\varepsilon }\left( x,v,t\right)
-f\left( x,v,t\right) $ with $\bar{f}_{\varepsilon }$ being the
super-solution in Definition \ref{super_sol_weak}, given in (\ref{super_sol}%
). Let $\varphi \left( x,v,t\right) $ be a solution of the following adjoint
problem: for any given $h>0$ small,%
\begin{eqnarray*}
\varphi _{t}+v\varphi _{x}+\varphi _{vv} &=&0,\ \text{for }h<x<1-h,~-\infty
<v<\infty , \\
\varphi \left( x,v,T\right) &=&\varphi _{0}\left( x,v\right) ,\text{ } \\
\varphi \left( h,v,t\right) &=&0,~v<0,~\varphi \left( 1-h,v,t\right) =0,~v>0,
\end{eqnarray*}%
where $\varphi _{0}\in C\left( B_{\rho }\left( x_{0},v_{0}\right) \right) $
is an arbitrary test function. Then we put $\psi $ and $\bar{\varphi}=\zeta
_{\rho }\varphi $ with the cut-off function $\zeta _{\rho }$ in (\ref{cutoff}%
) into the definition \ref{super_sol_weak} with a smaller domain $\left(
h,1-h\right) \times \left( -\infty ,\infty \right) \times \left( 0,t\right)
~ $to get%
\begin{align*}
& I+II+III:= \\
& -\int_{0}^{T}\int_{-\infty }^{\infty }\int_{h}^{1-h}\psi \left(
x,v,t\right) \left[ \bar{\varphi}_{t}\left( x,v,t\right) +\partial
_{x}\left( v\bar{\varphi}\left( x,v,t\right) \right) +\bar{\varphi}%
_{vv}\left( x,v,t\right) \right] dxdvdt \\
& +\int_{-\infty }^{\infty }\int_{h}^{1-h}\psi \left( x,v,t\right) \bar{%
\varphi}\left( x,v,T\right) dxdv-\int_{\Omega }\psi \left( x,v,0\right) \bar{%
\varphi}\left( x,v,0\right) dxdv \\
& +\int_{0}^{T}\int_{-\infty }^{0}v\psi \left( 1-h,v,t\right) \bar{\varphi}%
\left( 1-h,v,t\right) dvdt-\int_{0}^{t}\int_{0}^{\infty }v\psi \left(
h,v,s\right) \bar{\varphi}\left( h,v,s\right) dvdt \\
& \geq 0.
\end{align*}%
Notice that the compact support restriction of test functions in the
definition \ref{super_sol_weak} can be extended to $\varphi $ without
compact support by an approximation argument.

For I, we use the estimates, similar to (\ref{derOfAdj}) for solutions of
the Fokker-Planck problem, for the derivatives of solutions to the adjoint
problem together with the estimates for the derivatives of the cut-off
function $\zeta _{\rho }$:
\begin{equation*}
\left\Vert \zeta _{\rho ,x}\right\Vert _{L^{\infty }}\leq \frac{C}{\rho ^{3}}%
,~\left\Vert \zeta _{\rho ,v}\right\Vert _{L^{\infty }}~\leq \frac{C}{\rho }%
,~\left\Vert \zeta _{\rho ,vv}\right\Vert _{L^{\infty }}\leq \frac{C}{\rho
^{2}}.
\end{equation*}%
Thus we get%
\begin{eqnarray*}
&&-\int_{U_{T}}\psi \left( x,v,t\right) \left[ \bar{\varphi}_{t}\left(
x,v,t\right) +\partial _{x}\left( v\bar{\varphi}\left( x,v,t\right) \right) +%
\bar{\varphi}_{vv}\left( x,v,t\right) \right] dxdvdt \\
&=&-\int_{U_{T}}\psi \left( x,v,t\right) \zeta \left( x,v\right) \left[
\varphi _{t}\left( x,v,t\right) +\partial _{x}\left( v\varphi \left(
x,v,t\right) \right) +\varphi _{vv}\left( x,v,t\right) \right] dxdvdt \\
&&-\int_{U_{T}}\psi \left( x,v,t\right) \left[ v\varphi \zeta _{\rho
,x}+2\zeta _{\rho ,v}\varphi _{v}+\varphi \zeta _{\rho ,vv}\right] dxdvdt \\
&=&-\int_{U_{T}}\psi \left( x,v,t\right) \left[ v\varphi \zeta _{\rho
,x}+2\zeta _{\rho ,v}\varphi _{v}+\varphi \zeta _{\rho ,vv}\right] dxdvdt,
\end{eqnarray*}%
which leads to%
\begin{equation*}
\left\vert I\right\vert \leq C\rho ^{2},
\end{equation*}%
where $C$ depends on $T$,$\left\Vert \varphi _{0}\right\Vert _{L^{\infty
}\left( \Omega \right) },$ and $h$.

Thus we have%
\begin{eqnarray*}
&&\int_{\Omega }\psi \left( x,v,T\right) \bar{\varphi}\left( x,v,T\right)
dxdv \\
&\geq &\int_{\Omega }\psi \left( x,v,T\right) \zeta \left( x,v\right)
\varphi _{0}\left( x,v\right) dxdv-I-III \\
&\geq &\int_{\Omega }\psi \left( x,v,0\right) \zeta \left( x,v\right)
\varphi \left( x,v,0\right) dxdv-III+\mathcal{O}\left( \rho ^{2}\right) \\
&\geq &-III+\mathcal{O}\left( \rho ^{2}\right) .
\end{eqnarray*}%
Letting $\rho \rightarrow 0$ and , we obtain%
\begin{equation*}
\int_{\Omega }\psi \left( x,v,T\right) \bar{\varphi}\left( x,v,T\right)
dxdv\geq -III.
\end{equation*}%
Then we let $h\rightarrow 0$ to get $III\leq 0$ and%
\begin{equation*}
\int_{\Omega }\psi \left( x,v,t\right) \varphi _{0}\left( x,v\right)
dxdv\geq 0
\end{equation*}%
since $\psi \left( 1,v,s\right) \geq 0$ for $v<0,s>0$ and $\psi \left(
0,v,s\right) \geq $ $0$ for $v>0,s>0,$ and $\varphi \geq 0.$ Noticing $%
\varphi _{0}\left( x,v\right) $ is arbitrary, we can deduce that if $\bar{f}%
_{\varepsilon }\left( x,v,0\right) \geq C\hat{f}_{0}\geq f\left(
x,v,0\right) $, then $\bar{f}_{\varepsilon }\left( x,v,t\right) \geq f\left(
x,v,t\right) .$ Finally we let $\varepsilon \rightarrow 0$ to complete the
proof.
\end{proof}

We are now ready to prove Theorem 1.3 (ii).

\begin{proof}[Proof of Theorem 1.3 (ii)]
Recall (\ref{super_sol}). Since $C=\Vert f_{0}\Vert _{\infty },$ we see that
$f(x,v,0)\leq \bar{f}^{\ep}(x,v,0)$ for all $(x,v)$ and for any $\ep>0$. By
the maximum principle as in the proof of Lemma \ref{max-sup}, we deduce that
\begin{equation*}
f(x,v,t)\leq \bar{f}^{\ep}(x,v,t)\;\text{ for }\;t>0,\;(x,v)\in \bar{\Omega}%
\setminus \lbrack B_{\delta }(0,0)\cup B_{\delta }(1,0)],
\end{equation*}%
where $B_{\delta }(x_{0},v_{0}):=\{(x,v):|x-x_{0}|+|v-v_{0}|^{3}<\delta
^{3}\}$ and $\delta >0$ is arbitrary. Letting $\ep\rightarrow 0$, we further
deduce that
\begin{equation*}
f(x,v,t)\leq C\hat{f}_{0}(x,v,t)\;\text{ for all }\;(x,v)\in \bar{\Omega}%
\setminus \{(0,0),(1,0)\},
\end{equation*}%
which immediately implies that for $|x|+|v|^{3}\ll 1$
\begin{equation}
f(x,v,t)\leq C(|x|^{\alpha }+|v|^{3\alpha }).  \label{hold0}
\end{equation}%
We repeat the argument near $(1,0)$ by using $\hat{f}_{1}$, which is defined
by using the super-solution $Z_{1}$ of self-similar type near $(1,0)$ to
establish $f(x,v,t)\leq C(|x-1|^{\alpha }+|v|^{3\alpha })$.

It now remains to show the H\"{o}lder continuity. We already know that $f$
is smooth away from the singular set from the hypoellipticity result. Choose
$1\gg \delta _{0}>0$ such that on
\begin{equation*}
\mathcal{B}_{\delta _{0}}:=\Omega \cap \{B_{\delta _{0}}(0,0)\cup B_{\delta
_{0}}(1,0)\},
\end{equation*}%
$f$ satisfies \eqref{hold0}. We will focus on the part near $(0,0)$. Then it
suffices to show that for each $t>0$,
\begin{equation*}
|f(x_{1},v_{1},t)-f(x_{2},v_{2},t)|\leq C(|x_{1}-x_{2}|^{\alpha
}+|v_{1}-v_{2}|^{3\alpha })\text{ for any }(x_{i},v_{i})\in \mathcal{B}%
_{\delta _{0}},\;i=1,2.
\end{equation*}%
If $(x_{i},v_{i})$ for either $i=1$ or $i=2$ is a singular point $(0,0)$, we
are done because of \eqref{hold0}. Suppose not. We may assume that $%
(x_{1},v_{1})\in \partial B_{\delta _{1}}\cap \mathcal{B}_{\delta _{0}}$ and
$(x_{2},v_{2})\in \partial B_{\delta _{2}}\cap \mathcal{B}_{\delta _{0}}$
for $0<\delta _{2}\leq \delta _{1}\leq \delta _{0}$ and let $\rho
^{3}:=|x_{1}-x_{2}|+|v_{1}-v_{2}|^{3}>0$. If $\rho ^{3}\geq \frac{1}{100}%
(\delta _{1}^{3}+\delta _{2}^{3})$, then by the triangle inequality,
\begin{equation*}
\begin{split}
|f(x_{1},v_{1},t)-f(x_{2},v_{2},t)|& \leq
|f(x_{1},v_{1},t)-f(0,0,t)|+|f(0,0,t)-f(x_{2},v_{2},t)| \\
& \leq C(\delta _{1}^{3\alpha }+\delta _{2}^{3\alpha })\leq C\rho ^{3\alpha }
\\
& \leq C(|x_{1}-x_{2}|^{\alpha }+|v_{1}-v_{2}|^{3\alpha }).
\end{split}%
\end{equation*}%
If $\rho ^{3}<\frac{1}{100}(\delta _{1}^{3}+\delta _{2}^{3})$, then $\rho
\ll \delta _{1}\sim \delta _{2}$, for instance we have $\rho <\frac{1}{10}%
\delta _{2}$ and $\delta _{2}\leq \delta _{1}\leq 2\delta _{2}$. And hence,
the distance between $B_{2\rho }(x_{2},v_{2})\cap \mathcal{B}_{\delta _{0}}$
and the singular point $(0,0)$ is strictly positive and it contains $%
(x_{1},v_{1})$. Therefore, we can use the hypoellipticity result on $%
B_{2\rho }(x_{2},v_{2})\cap \mathcal{B}_{\delta _{0}}$ to conclude that it
is in fact smooth. To make it precise, we introduce a rescaling as follows:%
\begin{equation*}
f\left( x,v,t\right) =\delta ^{3\alpha }\bar{f}\left( \bar{x},\bar{v},\bar{t}%
\right) \ \ ,\ \ x=\rho ^{3}\bar{x}\ \ ,\ \ v=\rho \bar{v}\ \ ,\ \
t=t_{0}+\rho ^{2}\bar{t}\ \ ,\ \ t_{0}\geq 0,
\end{equation*}%
where $\delta ^{3}:=\max \left( \delta _{1}^{3},\delta _{2}^{3}\right) .$
Then $\bar{f}\left( \bar{x},\bar{v},\bar{t}\right) $ solves the
Fokker-Planck equation (\ref{VFP}) and due to (\ref{hold0}), it is bounded $%
0\leq \bar{f}\leq C$ in the set $\left\vert \bar{x}\right\vert +\left\vert
\bar{v}\right\vert ^{3}\leq \delta ^{3}/\rho ^{3}$. By applying Lemma \ref%
{hypoellipticity-der}, it then follows that%
\begin{equation*}
\left( \left\vert \bar{x}\right\vert +\left\vert \bar{v}\right\vert
^{3}\right) \left\vert \partial _{\bar{x}}\bar{f}\right\vert +\left(
\left\vert \bar{x}\right\vert +\left\vert \bar{v}\right\vert ^{3}\right)
^{1/3}\left\vert \partial _{\bar{v}}\bar{f}\right\vert \leq C.
\end{equation*}%
Returning to the original variables $\left( x,v,t\right) ,$ we obtain the
following estimates.%
\begin{equation*}
\left( \left\vert x\right\vert +\left\vert v\right\vert ^{3}\right)
\left\vert \partial _{x}f\right\vert +\left( \left\vert x\right\vert
+\left\vert v\right\vert ^{3}\right) ^{1/3}\left\vert \partial
_{v}f\right\vert \leq C\delta ^{3\alpha },
\end{equation*}%
for all $\left( x,v\right) \in \Omega $ with $\left\vert x\right\vert
+\left\vert v\right\vert ^{3}\leq \delta ^{3}.$ We now estimate the
difference $f\left( x_{1},v_{1},t\right) -f\left( x_{2},v_{2},t\right) $ as
follows.%
\begin{equation*}
\left\vert f\left( x_{1},v_{1},t\right) -f\left( x_{2},v_{2},t\right)
\right\vert \leq \left\vert \partial _{x}f\right\vert \left( \tilde{x},%
\tilde{v},t\right) \left\vert x_{1}-x_{2}\right\vert +\left\vert \partial
_{v}f\right\vert \left( \tilde{x},\tilde{v},t\right) \left\vert
v_{1}-v_{2}\right\vert ,
\end{equation*}%
for some $\left( \tilde{x},\tilde{v}\right) $ with $\min \left( \delta
_{1}^{3},\delta _{2}^{3}\right) \leq \left\vert \tilde{x}\right\vert
+\left\vert \tilde{v}\right\vert ^{3}\leq \max \left( \delta _{1}^{3},\delta
_{2}^{3}\right) =\delta ^{3}.$ Thus we have%
\begin{eqnarray*}
\left\vert f\left( x_{1},v_{1},t\right) -f\left( x_{2},v_{2},t\right)
\right\vert &\leq &C\delta ^{3\alpha }\left[ \frac{\left\vert
x_{1}-x_{2}\right\vert }{\min \left( \delta _{1}^{3},\delta _{2}^{3}\right) }%
+\frac{\left\vert v_{1}-v_{2}\right\vert }{\min \left( \delta _{1},\delta
_{2}\right) }\right] \\
&\leq &C\delta ^{3\alpha }\left[ \frac{\rho ^{3}}{\min \left( \delta
_{1}^{3},\delta _{2}^{3}\right) }+\frac{\rho }{\min \left( \delta
_{1},\delta _{2}\right) }\right] \\
&\leq &C\delta ^{3\alpha }\left[ \frac{\rho ^{3\alpha }\delta ^{3-3\alpha }}{%
\min \left( \delta _{1}^{3},\delta _{2}^{3}\right) }+\frac{\rho ^{3\alpha
}\delta ^{1-3\alpha }}{\min \left( \delta _{1},\delta _{2}\right) }\right] \\
&\leq &C\rho ^{3\alpha }\left[ \frac{\delta ^{3}}{\min \left( \delta
_{1}^{3},\delta _{2}^{3}\right) }+\frac{\delta }{\min \left( \delta
_{1},\delta _{2}\right) }\right] .
\end{eqnarray*}%
Notice that the third inequality above is due to our assumption that $\rho
^{3}<\frac{1}{100}(\delta _{1}^{3}+\delta _{2}^{3})$ and $\alpha <1/6.$
Since $\delta _{1}\sim \delta _{2}\sim \delta ,$ we can conclude that%
\begin{equation*}
\left\vert f\left( x_{1},v_{1},t\right) -f\left( x_{2},v_{2},t\right)
\right\vert \leq C\rho ^{3\alpha }\leq C\left( \left\vert
x_{1}-x_{2}\right\vert ^{\alpha }+\left\vert v_{1}-v_{2}\right\vert
^{3\alpha }\right) .
\end{equation*}%
This completes the proof of the H\"{o}lder continuity.
\end{proof}

\section{Exponential convergence rate on a bounded interval}

%---------------------------------------------------------%---------------------------------------------------------
%---------------------------------------------------------%---------------------------------------------------------

In this section, we will show that the solutions for the Fokker-Planck
equation (\ref{VFP}) with the initial and absorbing boundary conditions (\ref%
{id})-(\ref{BC1}) decay exponentially as time goes to infinity. We begin
with the following technical lemma, which shows the exponential convergence
under a certain set of assumptions to be verified later.

\begin{lemma}
\label{exp} Let $\{z_{n}\}$, $\{M_{n}\}$ be given such that $z_{n}\geq 0$
and $M_{n}\geq 0$ and that $M_{n}\leq M_{n-1}$ for all $n\in \mathbb{N}$.
Suppose that there exist $0<\theta ,\beta <1$, $A\geq C>0$, $T>0$ such that
the following holds

\begin{enumerate}
\item If $z_{n}\geq AM_{n-1}$, $z_{n+1}\leq \theta z_{n}$ for $\theta <1,$

\item If $z_{n}<AM_{n-1}$, $M_{n+T}\leq \beta M_{n-1}$ for $\beta <1,$

\item $z_{n+1}\leq \theta \max \{z_{n},CM_{n-1}\}.$
\end{enumerate}

Then there exists $0<\mu <1$ satisfying $z_{n}+M_{n}\leq c\mu ^{n}.$
\end{lemma}

\begin{proof}
Define a sequence
\begin{equation*}
\omega _{k}:=\max \left\{ \frac{z_{2k(T+1)+1}}{A}\;,\;M_{2k(T+1)}\right\}
\end{equation*}%
for $k\geq 0$. We will show that there exists $0<\gamma <1$ independent of $%
k $ such that for each $k\geq 0,$
\begin{equation}
\omega _{k+1}\leq \gamma \omega _{k}.  \label{seq}
\end{equation}%
We remark that \eqref{seq} implies the exponential decay with the rate $\mu
=O(\gamma ^{{1}/{2(T+1)}})$.

What follows is the proof of \eqref{seq}. We first start with $%
M_{2k(T+1)+2(T+1)}$ part in $\omega _{k+1}$. Suppose there exists an $1\leq
n_{0}\leq T+1$ such that $z_{2k(T+1)+n_{0}}<AM_{2k(T+1)+n_{0}-1}$. Then, by
the assumption (2), $M_{2k(T+1)+n_{0}+T}\leq \beta M_{2k(T+1)+n_{0}-1}$.
Since $M_{n}$ is non-increasing in $n$, we deduce that
\begin{equation*}
M_{2k(T+1)+2(T+1)}\leq M_{2k(T+1)+n_{0}+T}\leq \beta M_{2k(T+1)+n_{0}-1}\leq
\beta M_{2k(T+1)}.
\end{equation*}%
Suppose that $z_{2k(T+1)+n}\geq AM_{2k(T+1)+n-1}$ for all $1\leq n\leq T+1$.
Then by the assumption (1), we deduce that $z_{2k(T+1)+T+1}\leq \theta ^{T}{z%
}_{2k(T+1)+1}$. This immediately yields that
\begin{equation*}
M_{2k(T+1)+2(T+1)}\leq M_{2k(T+1)+T}\leq \frac{z_{2k(T+1)+T+1}}{A}\leq
\theta ^{T}\frac{z_{2k(T+1)+1}}{A}
\end{equation*}%
and thus we deduce that
\begin{equation}
M_{2k(T+1)+2(T+1)}\leq \max \{\beta M_{2k(T+1)},\theta ^{T}\frac{%
z_{2k(T+1)+1}}{A}\}.  \label{Mseq}
\end{equation}%
We now turn to $z_{2k(T+1)+2(T+1)+1}$ part in $\omega _{k+1}$. Using the
assumption (3), we observe that
\begin{equation*}
\begin{split}
\frac{z_{2k(T+1)+2(T+1)+1}}{A}& \leq \max \{\frac{C\theta }{A}%
M_{2k(T+1)+2(T+1)-1},\theta \frac{z_{2k(T+1)+2(T+1)}}{A}\} \\
& \leq \max \{\frac{C\theta }{A}M_{2k(T+1)+2(T+1)-1},\frac{C\theta ^{2}}{A}%
M_{2k(T+1)+2(T+1)-2},\dots \\
& \quad \quad \quad \quad \dots ,\frac{C\theta ^{2T+2}}{A}M_{2k(T+1)},\theta
^{2T+2}\frac{z_{2k(T+1)+1}}{A}\}.
\end{split}%
\end{equation*}%
Since $M_{n}$ is non-increasing in $n$ and $C\leq A$, we deduce that
\begin{equation}
\frac{z_{2k(T+1)+2(T+1)+1}}{A}\leq \max \{\theta M_{2k(T+1)},\theta ^{2T+2}%
\frac{z_{2k(T+1)+1}}{A}\}.  \label{Zseq}
\end{equation}%
Combining \eqref{Mseq} and \eqref{Zseq}, and letting $\gamma =\max \{\beta
,\theta \}<1$, we obtain \eqref{seq}.
\end{proof}

\

We denote a neighborhood of the singular set by $S$:\
\begin{equation*}
S:=\{(x,v)\in \Omega :|x|+|v|^{3}\leq \rho ^{3}\text{ or }|x-1|+|v|^{3}\leq
\rho ^{3}\}
\end{equation*}%
for $\rho >0$ a fixed small (not necessarily too small) number. Let $Q$ be
the complement of $S$:
\begin{equation*}
Q:=\Omega \setminus S
\end{equation*}%
and we further introduce the extended $Q$ by
\begin{equation*}
Q_{E}:=\Omega \setminus \frac{1}{2}S
\end{equation*}%
so that $Q\subset Q_{E}\subset \Omega $.

We use $\zeta _{s}(t)$ to denote the supremum of $f$ on S:
\begin{equation*}
\zeta _{s}(t):=\Vert f(\cdot ,t)\Vert _{L^{\infty }(S)}
\end{equation*}%
and as before we will use $\Vert f(\cdot ,t)\Vert _{\infty }$ to denote the
supremum of $f$ on the entire phase space $\Omega $. We use $M(t)$ to denote
the total mass at time $t$:
\begin{equation*}
M(t):=\int_{\Omega }f(x,v,t)dxdv.
\end{equation*}

We know that $M(t)$ is non-increasing in $t$. Our goal is to prove that $%
M(t) $ decays exponentially to zero by showing that part of the mass escapes
to the boundary and that the solution decays also on the singular set.

The following lemma concerns the behavior of a solution on $S$.

\begin{lemma}
\label{singBeh}There exist $\rho >0$ and $0<\theta <1$ such that
\begin{equation}
\zeta _{s}(t)\leq \theta \Vert f(\cdot ,\bar{t})\Vert _{\infty }\;\text{ for
}\;t\geq \bar{t}+1,  \label{4.6}
\end{equation}%
where $\theta <1$.
\end{lemma}

\begin{proof}
Recall that $f(x,v,t)\leq C\hat{f}_{0}(x,v,t),$ where $\hat{f}_{0}$ is given
in \eqref{fhat}. Then by the self-similar structure of the super-solution $%
\hat{f}_{0}$, we see that there exists a small $\rho >0$ so that for $t\geq
\bar{t}+1$,
\begin{equation*}
\sup_{S}f(x,v,t)\leq \theta \Vert f(\cdot ,\bar{t})\Vert _{\infty }.
\end{equation*}%
Since a similar argument holds near $\left( 1,0\right) ,$ this completes the
proof.
\end{proof}

Next we obtain the following estimate on $\sup_{Q}f$ from the
hypoellipticity of $f$.

\begin{lemma}
\label{hyp}There exists $C_{s}>0$ such that
\begin{equation}
\sup_{(x,v)\in Q}f(x,v,t)\leq C_{s}\int_{Q_{E}}f(x,v,\bar{t})dxdv\;\text{
for }\;t\geq \bar{t}+1,  \label{4.7}
\end{equation}%
where $C_{s}$ depends only on the size of the singular set $S$.
\end{lemma}

\begin{proof}
It follows from Theorem \ref{MainTheorem} and the bound by $L^{1}$ norm was
given in the proof of Theorem \ref{MainTheorem}, as mentioned in Remark \ref%
{Rem}.
\end{proof}

As a direct consequence of the two lemmas above, we derive the following
property of $\zeta _{s}(t)$.

\begin{lemma}[Verification of assumptions (1) and (3)]
\label{lem47}For any $t\geq 1$,
\begin{equation*}
\zeta _{s}(t+1)\leq \theta \max \{\zeta _{s}({t}),C_{s}M({t}-1)\},
\end{equation*}%
where $\theta >0$ is given in \eqref{4.6} and $C_{s}$ is given in \eqref{4.7}%
. In particular, if $\zeta _{s}(t)>C_{s}M(t-1)$, $\zeta _{s}(t+1)\leq \theta
\zeta _{s}(t)$.
\end{lemma}

\begin{proof}
It follows from Lemma \ref{singBeh} and Lemma \ref{hyp}.
\end{proof}

Lemma \ref{lem47} asserts that if the amplitude of a solution on the
singular set is much greater than the total mass at an earlier time, the
amplitude at a later time should decrease.

In the next lemma, we show that mass does not move far away over time.

\begin{lemma}
\label{tightness}\bigskip (Tightness lemma) Let $f$ be a strong solution of (%
\ref{VFP})-(\ref{BC1}) with the initial data $f_{0}\in L^{1}\cap L^{\infty
}\left( \Omega \right) $ with $f_{0}\geq 0$. For a given $t>0$ and $\delta
>0,$ there exists a constant $B>0$ depending on $t,\delta ,$ and $%
\int_{\Omega }f\left( x,v,t\right) dxdv$ such that%
\begin{equation*}
\int_{\left\vert v\right\vert \leq B}f\left( x,v,t\right) dxdv\geq \left(
1-\delta \right) \int_{\Omega }f\left( x,v,t\right) dxdv.
\end{equation*}
\end{lemma}

\begin{proof}
We may assume that $\int_{\Omega }f_{0}dxdv=1.$ \ Let mass at time $t$ be $M$%
: $\int_{\Omega }f(x,v,t)dxdv=M.$ We first split the initial data into two
parts: $f_{0}=f_{0,1}+f_{0,2}$ with $f_{0,1}\geq 0,f_{0,2}\geq 0$, where supp%
$\left( f_{0,1}\right) \subset \left[ 0,1\right] \times \left[ -B,B\right] $
and $\int_{\Omega }f_{0,1}dxdv$ $\geq 1-\frac{\delta M}{2}.$ Then there
exist strong solutions $f_{1}$ and $f_{2}$ corresponding to the initial data
$f_{0,1}$ and $f_{0,2}$ respectively and $f=f_{1}+f_{2}.$ Let $\bar{f}\left(
x,v,t\right) =e^{\theta t}e^{-A\sqrt{v^{2}+1}},$ where $\theta =\theta
\left( A\right) $. Then by a direct calculation, it is easy to see that
\begin{equation*}
\left\vert \bar{f}_{vv}\left( x,v,t\right) \right\vert \leq C\left( A\right)
e^{\theta t}e^{^{-A\sqrt{v^{2}+1}}}.
\end{equation*}%
Thus it is now easy to see that $\bar{f}\left( x,v,t\right) \,$\ is a
super-solution of (\ref{VFP}) provided we choose $\theta \left( A\right) >0$
sufficiently large. Since $f_{0,1}$ has a compact support, there exists $K>0$
such that $f_{0,1}\left( x,v\right) \leq K\bar{f}\left( x,v,0\right) .$ Then
using the comparison property for strong solutions in Lemma \ref{maxStrong}
and applying it to $g=f_{1}-K\bar{f}$, we get%
\begin{equation*}
f_{1}\left( x,v,t\right) \leq K\bar{f}\left( x,v,t\right) .
\end{equation*}%
This implies that%
\begin{equation*}
\int_{\left\vert v\right\vert \geq B}f_{1}\left( x,v,t\right) dxdv\leq
\int_{\left\vert v\right\vert \geq B}K\bar{f}\left( x,v,t\right)
dxdv=Ke^{\theta t}\int_{\left\vert v\right\vert \geq B}e^{-A\sqrt{v^{2}+1}%
}dxdv\leq \frac{\delta M}{2}
\end{equation*}%
if we choose $B>0$ sufficiently large. Indeed, we can choose $B=C\left(
1+\ln \frac{1}{\delta M}+t\right) $, where $C$ depends on $A$ and $K$. Thus
we derive%
\begin{equation*}
\int_{\left\vert v\right\vert \geq B}f_{1}\left( x,v,t\right) dxdv\leq \frac{%
\delta M}{2}.
\end{equation*}%
For $f_{2}$, we have
\begin{equation*}
\int_{\Omega }f_{2}\left( x,v,t\right) dxdv\leq \int_{\Omega
}f_{0,2}(x,v)dxdv\leq \frac{\delta M}{2}.
\end{equation*}%
Therefore, we obtain
\begin{equation*}
\int_{\left\vert v\right\vert \geq B}f\left( x,v,t\right)
dxdv=\int_{\left\vert v\right\vert \geq B}f_{1}\left( x,v,t\right)
dxdv+\int_{\left\vert v\right\vert \geq B}f_{2}\left( x,v,t\right) dxdv\leq
\delta M.
\end{equation*}%
This completes the proof.
\end{proof}

In particular, we have the following.

\begin{corollary}
For any $t>0,$ there exists a $\tilde{B}>0$ such that%
\begin{equation*}
\int_{\left\vert v\right\vert \leq \tilde{B}}f\left( x,v,t\right) dxdv\geq
\frac{1}{2}\int_{\Omega }f\left( x,v,t\right) dxdv.
\end{equation*}
\end{corollary}

\begin{proof}
It follows immediately from Lemma \ref{tightness}.
\end{proof}

Next we show that if $f$ is comparable to the mass in a small ball away from
the singular set at the present time, then the amount of mass comparable to
the mass at the present time escapes to the boundary at some later times.

\begin{lemma}[Escape of mass to the boundary]
\label{esc}\label{escape}Let $f_{0}(x,v)\geq \ep M\chi _{B_{\rho
}(x_{0},v_{0})}(x,v)$ be given where $B_{\rho (x_{0},v_{0})}$ is an interior
ball with center $\left( x_{0},v_{0}\right) $ and radius $\rho >0$ and $%
M=\int f_{0}(x,v)dxdv$. Let $f(x,v,t)$ be a solution to the Fokker-Planck
equation (\ref{VFP}) in the unit interval $[0,1]$ with absorbing boundary
conditions. Then there exists $\alpha =\alpha (\rho ,\varepsilon )<1$,
independent of $x_{0}$ and $v_{0}$ such that
\begin{equation*}
\int f(x,v,1)dxdv\leq \alpha \int f_{0}(x,v)dxdv.
\end{equation*}
\end{lemma}

\begin{proof}
Without loss of generality, we may assume $v_{0}>0$ since the other case can
be treated similarly. First we show that $f_{0}(x,v)\geq \varepsilon
_{1}M\chi _{B_{\rho _{1}}(x_{1},v_{1})}(x,v)$ for some $\left(
x_{1},v_{1}\right) \in \Omega $ with $v_{1}\geq 1$. We may assume that $%
0<v_{0}\leq 1$ since we are done otherwise. We can also assume $\rho \leq 1.$
To prove the statement above, let $H\left( v,t\right) :=\int_{0}^{1}f\left(
x,v,t\right) dx.$ Then $H\left( v,t\right) $ satisfies the following
equation:%
\begin{equation*}
H_{t}-H_{vv}=-vf\left( 1,v,t\right) \chi _{\left\{ v>0\right\} }+vf\left(
0,v,t\right) \chi _{\left\{ v<0\right\} }=:-g\left( v,t\right) \leq 0,
\end{equation*}%
where $\chi $ is the characteristic function.

If $g\left( v,t\right) >\gamma M$ for some $\left( v,t\right) \in \left(
-\infty ,\infty \right) \times \left[ 0,1/2\right] ,~$where $\gamma >0$
small and depending only on $\varepsilon $ is to be determined, then there
exists a ball in which $g\left( v,t\right) >\frac{\gamma M}{2}$ and this
implies that $\int_{0}^{1/2}\int_{-\infty }^{\infty }g\left( v,t\right)
dvdt\geq \gamma _{1}M,$ where $\gamma _{1}>0$ depends on $\gamma $. We then
have $M\left( 1\right) =\int f(x,v,1)dxdv\leq M\left( 1/2\right) \leq \left(
1-\gamma _{1}\right) M$ and we are done.

If we now assume that $g\left( v,t\right) dvdt\leq \gamma M$ for all $\left(
v,t\right) \in \left( -\infty ,\infty \right) \times \left[ 0,1/2\right] .$
Then we have the following integral representation for $H\left( v,t\right) $:%
\begin{equation*}
H\left( v,t\right) =\int_{-\infty }^{\infty }\mathcal{K}\left( v-w,t\right)
H\left( w,0\right) dw-\int_{0}^{t}\int_{-\infty }^{\infty }\mathcal{K}\left(
v-w,t-s\right) g\left( w,s\right) dwds,
\end{equation*}%
where $\mathcal{K}\left( v,t\right) $ is the one dimensional heat kernel. We
use the assumptions on $f_{0}~$and on $g\left( v,t\right) $ to get, for $%
1\leq v\leq 2,$%
\begin{eqnarray*}
H\left( v,1/2\right) &\geq &\int_{B_{\rho }\left( x_{0},v_{0}\right) }%
\mathcal{K}\left( v-w,1/2\right) f\left( x,w,0\right) dwdx-\frac{\gamma M}{2}
\\
&\geq &\varepsilon M\mathcal{K}\left( 3,1/2\right) -\frac{\gamma M}{2}=\left[
\varepsilon \mathcal{K}\left( 3,1/2\right) -\frac{\gamma }{2}\right] M \\
&\geq &\frac{\mathcal{K}\left( 3,1/2\right) }{2}\varepsilon M=:\varepsilon
_{0}M,
\end{eqnarray*}%
if we choose $\gamma =\varepsilon \mathcal{K}\left( 3,1/2\right) .$ Since $%
\int_{1}^{2}\int_{0}^{1}f\left( x,v,1/2\right) dxdv\geq \varepsilon _{0}M$
and $f$ is continuous on $\left[ 0,1\right] \times \left[ 1,2\right] $ from
the hypoellipticity, there exists $\rho _{1}>0$ such that $f\left(
x,v,1/2\right) \geq \frac{\varepsilon _{0}}{2}M\chi _{B_{\rho
_{1}}(x_{1},v_{1})}(x,v)$ with $v_{1}\geq 1$ and $\left( x_{1},v_{1}\right)
\in \left[ 0,1\right] \times \left[ 1,2\right] .$

Now with abuse of notation, we use $\rho ,x_{0},v_{0}$ and $t=0$ with $f_{0}$
being continuous instead of $\rho _{1},x_{1},v_{1}$ and $t=1/2$. We look for
a sub-solution $F\left( x,v,t\right) \in C_{x,v,t}^{1,2,1}\left( \Omega
_{T}\right) $ to the Fokker-Planck equation (\ref{VFP}) of the form
\begin{equation}
F\left( x,v,t\right) =e^{-\lambda t}h\left( x-vt,v\right) ,  \label{sub-sol}
\end{equation}%
where $\lambda >0$ will be chosen later and $h\in C_{x,v}^{1,2}\left( \Omega
\right) $. We plug in (\ref{sub-sol}) into $\mathcal{M}f\leq 0$ to get%
\begin{equation}
h_{vv}+t^{2}h_{xx}-2th_{xt}\geq -\lambda h.  \label{h_vv}
\end{equation}%
Since $\varepsilon M\chi _{B_{\rho }(x_{0},v_{0})}(x,v)\leq f_{0}(x,v)$ for
all $\left( x,v\right) \in \Omega ,$ then there exists $\delta >0$ such that
$f_{0}>\frac{\varepsilon }{2}M$ for $\left( x,v\right) \in B_{\rho +\delta
}\left( x_{0},v_{0}\right) $ since $f_{0}$ is continuous. We then can find $%
\frac{\varepsilon }{2}M\chi _{B_{\rho }(x_{0},v_{0})}(x,v)\leq h\left(
x,v\right) \leq f_{0}(x,v)$ for all $\left( x,v\right) \in \Omega .$ For
instance, $h=\frac{\varepsilon }{2}M\cos \left( \lambda ^{1/4}\left(
v-v_{0}\right) \right) \cos \left( \lambda ^{1/4}\left( x-x_{0}\right)
\right) +\frac{\varepsilon }{2}M$ on $B_{\rho }\left( x_{0},v_{0}\right) $,
where $\lambda =\left( \frac{2\pi }{\rho }\right) ^{4}>0$ for $\lambda $
sufficiently large so as to satisfy (\ref{h_vv}) and for $0\leq t\leq 1$
(This is possible since we can make $\rho >0$ as small as possible). Then we
take values of $h$ in such a way that $h_{vv}\geq 0.$ This can be done by
first taking values of $h$ as $0$ for $\mathbb{R}^{2}\smallsetminus B_{\rho
+\delta }(x_{0},v_{0})$ and then by connecting the points on $\partial
B_{\rho }(x_{0},v_{0})$ and the points on $\partial B_{\rho +\delta
}(x_{0},v_{0})$ with the points $\left( v,h\left( v\right) \right) =\left(
v_{0}\pm \rho _{0},\frac{\varepsilon }{2}M\right) $ in a convex way. Note
that $\lambda $ depends only on $\rho $. Then by the maximum principle of $%
\mathcal{M}$ (Lemma \ref{maxStrong}), we have $F\left( x,v,t\right) \leq
f\left( x,v,t\right) $ for all $t>0.$ In particular, at $x=1,$ there exists $%
t_{0}>0$ such that $x_{0}+v_{0}t_{0}=1$ (which implies $t_{0}\leq 1$) and
\begin{eqnarray*}
f\left( 1,v,t_{0}\right) &\geq &F\left( 1,v,t_{0}\right) =e^{-\lambda
t_{0}}h\left( 1-vt_{0},v\right) \\
&\geq &\frac{e^{-\lambda t_{0}}\varepsilon }{2}M\geq \frac{e^{-\lambda
}\varepsilon }{2}M=:\varepsilon _{1}M,
\end{eqnarray*}%
for $\left\vert v-v_{0}\right\vert <\rho _{1}$ with some $\rho _{1}=\frac{%
\rho }{\sqrt{1+t_{0}^{2}}}\geq \frac{\rho }{\sqrt{2}}>0,\varepsilon _{1}=%
\frac{e^{-\lambda }\varepsilon }{2}>0.$ Note that $\varepsilon _{1}$ does
not depend on $\left( x_{0},v_{0}\right) $ but on $\rho $ and $\varepsilon .$
By the continuity of $f,$
\begin{equation*}
f\left( 1,v,t\right) \geq \frac{\varepsilon _{1}M}{2}=:\varepsilon _{2}M,
\end{equation*}%
for $\left\vert v-v_{0}\right\vert <\frac{\rho _{1}}{2}$ and $t\in \left[
t_{0}-\delta _{0},t_{0}\right] .$ Then integrating (\ref{VFP}) in $x,v,t$
yields
\begin{eqnarray*}
\int f(x,v,1)dxdv &\leq &\int f_{0}(x,v)dxdv-\int_{t_{0}-\delta
_{0}}^{t_{0}}\int_{v_{0}-\rho /2\sqrt{2}}^{v_{0}+\rho /2\sqrt{2}}vf\left(
1,v,t\right) dvdt \\
&\leq &\left( 1-\frac{\rho \varepsilon _{2}}{\sqrt{2}}\right) M=:\alpha M,
\end{eqnarray*}%
where $\alpha <1$ depends only on $\rho $ and $\varepsilon $. We used $%
v_{0}\geq 1$ and $\rho \leq 1.$ This completes the proof.
\end{proof}

We will now show that if $\zeta _{s}(t)$ is bounded by a multiple of $M(t-1)$%
, then the total mass after a finite time should be decreasing by a uniform
factor which is strictly less than one.

\begin{lemma}[Verification of assumption (2)]
\label{lem46} Given $S,A$ ($A$ arbitrarily large), there exist $0<\beta
=\beta (A,S)<1$ and $T=T(A,S)$ such that if $\zeta _{s}(n)<AM(n-1)$ for some
$n\geq 1$, then $M(n+T)\leq \beta M(n-1)$.
\end{lemma}

\begin{proof}
We define $\ep_0$ by means of $4C_s|S|\ep_0<1$ and $4\ep_0<1$. Let $T>0$ be
a given positive integer to be determined. If $M(n+T)\leq \frac12 M(n-1)$,
we are done. Suppose then that $M(n+T)> \frac12 M(n-1)$. Since $M(n)$ is
decreasing, we first see that
\begin{equation}  \label{4.8}
M(l)\geq M(n+T) >\frac{M(n-1)}{2} \;\text{ for }\; l=n, n+1,\dots, n+T.
\end{equation}
We have two cases.

\begin{enumerate}
\item There exists an $l_{0}\in \{n,n+1,\dots ,n+T\}$ such that
\begin{equation*}
\int_{Q_{E}}fdxdv\Big|_{t=l_{0}}\geq \ep_{0}M(n-1).
\end{equation*}

\item For all $l\in \{n,n+1,\dots ,n+T\}$,
\begin{equation*}
\int_{Q_{E}}fdxdv\Big|_{t=l}<\ep_{0}M(n-1).
\end{equation*}
\end{enumerate}

In the case of (1), we apply Lemma \ref{esc} to prove that at least some
part of the mass occupied in $Q_{E}$ at time $l_{0}$ should escape to the
boundary at a later time $l_{0}+1$, which would in turn imply $%
M(l_{0}+1)\leq \beta _{0}M(n-1)$ for some $\beta _{0}<1$.

We now turn to the case (2). We will show that this case is impossible if $T$
is chosen appropriately. To show a contradiction, we exploit the property of
$\zeta _{s}$. The first claim is the following
\begin{equation}
\zeta _{s}(n+1)\leq A\theta M(n-1).  \label{4.9}
\end{equation}%
To see it, notice that
\begin{equation*}
\begin{split}
\Vert f(\cdot ,n)\Vert _{\infty }& \leq \max \{\zeta _{s}(n),\Vert f(\cdot
,n)\Vert _{L^{\infty }(Q)}\} \\
& \leq \max \{AM(n-1),C_{s}M(n-1)\}\text{ by the assumption on }\zeta _{s}(n)%
\text{ and }\eqref{4.7} \\
& \leq AM(n-1)\text{ by choosing }A>C_{s}.
\end{split}%
\end{equation*}%
Now by \eqref{4.6} we can easily deduce the above assertion \eqref{4.9}.

For $l\geq n$, we observe that
\begin{equation*}
M(l)=\int_{S}fdxdv+\int_{Q}fdxdv\leq \int_{S}fdxdv\Big|_{t=l}+%
\int_{Q_{E}}fdxdv\Big|_{t=l}.
\end{equation*}%
Since $\int_{Q_{E}}fdxdv\Big|_{t=l}<\ep_{0}M(n-1)$ by the given assumption
of the case (2) and since $M(n-1)\leq 2M(l)$ by the assumption \eqref{4.8},
we see that
\begin{equation*}
M(l)\leq \int_{S}fdxdv\Big|_{t=l}+2\ep_{0}M(l)
\end{equation*}%
and hence for $2\ep_{0}<1/2$, we deduce that for $l\geq n,$
\begin{equation}
M(l)\leq 2\int_{S}fdxdv\Big|_{t=l}\leq 2|S|\zeta _{s}(l).  \label{4.10}
\end{equation}%
As a consequence of \eqref{4.10}, we derive that for $l\geq n+1$,
\begin{equation*}
\begin{split}
\sup_{Q}f(x,v,l)& \leq C_{s}\int_{Q_{E}}fdxdv\mid _{t=l-1}\text{ by }%
\eqref{4.7} \\
& \leq C_{s}\ep_{0}M(n-1)\text{ by the assumption in the case of (2)} \\
& \leq 2C_{s}\ep_{0}M(l)\text{ by }\eqref{4.8} \\
& \leq 4\ep_{0}C_{s}|S|\zeta _{s}(l)\text{ by }\eqref{4.10}.
\end{split}%
\end{equation*}%
Then together with \eqref{4.6}, we obtain for $l\geq n+1$,
\begin{equation*}
\zeta _{s}(l+1)\leq \theta \max \{\zeta _{s}(l),\sup_{Q}f(x,v,l)\}\leq
\theta \max \{\zeta _{s}(l),4\ep_{0}C_{s}|S|\zeta _{s}(l)\}
\end{equation*}%
and since $4\ep_{0}C_{s}|S|<1$, we conclude that
\begin{equation}
\zeta _{s}(l+1)\leq \theta \zeta _{s}(l),\text{ for }l\geq n+1.  \label{4.12}
\end{equation}%
Hence, by iteration together with \eqref{4.9} we deduce that
\begin{equation*}
\zeta _{s}(n+T)\leq A\theta ^{T}M(n-1),
\end{equation*}%
which yields that from \eqref{4.10}, for $l=n+T,$
\begin{equation*}
M(n+T)\leq 2|S|A\theta ^{T}M(n-1).
\end{equation*}%
On the other hand, from \eqref{4.8}, we have that
\begin{equation*}
\frac{M(n-1)}{2}<M(n+T).
\end{equation*}%
But this is impossible for $T$ sufficiently large since $\theta <1$, which
is the desired contradiction. This finishes the proof of the lemma.
%\begin{equation}\label{4.12}
%\zeta_s(l+1)\leq \theta \max\{ \zeta_s(l), \zeta_s(l-1)\}
%\end{equation}
\end{proof}

\bigskip

We are now ready to prove Theorem 1.3 of the exponential decay in $L^{1}$
and $L^{\infty }$ sense in time of solutions for (\ref{VFP})-(\ref{BC1}).

\label{decay}

\begin{proof}[Proof of Theorem 1.4]
Part (i) is a consequence of Lemma \ref{exp}, \ref{lem47}, and \ref{lem46}.
Part (ii) follows immediately from Part (i) and the hypoellipticity.
\end{proof}

\bigskip

\begin{acknowledgement}
The authors would like to thank Professor Yan Guo for his encouragement to
study this problem and the Hausdorff Center of the University of Bonn where
part of this work was done.
%The authors would also like to thank the anonymous referee for
%valuable comments.
Hyung Ju Hwang is partly supported by the Basic Science Research Program
(2010-0008127) and (2013053914) through the National Research Foundation of
Korea (NRF). Juhi Jang is supported in part by NSF grants DMS-0908007 and
DMS-1212142.
\end{acknowledgement}

\end{document}